\documentclass[11pt]{amsart}
\usepackage{amsmath}
\usepackage{amscd,amsthm,amssymb,amsfonts}
\usepackage{mathrsfs}
\usepackage{dsfont}
\usepackage{stmaryrd}
\usepackage{euscript}
\usepackage{expdlist}
\usepackage{enumerate}

\input xy
\xyoption{all}

\usepackage[OT2,T1]{fontenc}

\theoremstyle{plain}
\newtheorem{thm}{Theorem}[section]

\newtheorem*{thma}{Theorem A}

\newtheorem*{thm*}{Theorem}

\newtheorem{lm}[thm]{Lemma}
\newtheorem{cor}[thm]{Corollary}
\newtheorem*{cor*}{Corollary}
\newtheorem{prop}[thm]{Proposition}
\newtheorem*{conj*}{Conjecture}

\theoremstyle{remark}

\theoremstyle{definition}
\newtheorem*{defn*}{Definition}
\newtheorem{Remark}[thm]{Remark}
\newtheorem{I_Remark*}{Remark}
\newtheorem{defn}[thm]{Definition}

\newtheorem{hypothesis}[thm]{Hypothesis}

\newcommand{\nc}{\newcommand}

\newcommand{\beq}{\begin{equation}}
\newcommand{\eeq}{\end{equation}}
\newcommand{\bpmx}{\begin{pmatrix}}
\newcommand{\epmx}{\end{pmatrix}}
\newcommand{\bbmx}{\begin{bmatrix}}
\newcommand{\ebmx}{\end{bmatrix}}
\newcommand{\wh}{\widehat}
\newcommand{\wtd}{\widetilde}

\newcommand{\beqcd}[1]{\begin{equation*}\label{#1}\tag{#1}}
\newcommand{\eeqcd}{\end{equation*}}

\numberwithin{equation}{section}

\def\parref#1{\ref{#1}}
\def\thmref#1{Theorem~\parref{#1}}

\def\propref#1{Proposition~\parref{#1}}
\def\corref#1{Corollary~\parref{#1}}     \def\remref#1{Remark~\parref{#1}}

\def\lmref#1{Lemma~\parref{#1}}
\def\subsecref#1{\S\parref{#1}}
\def\defref#1{Definition~\parref{#1}}

\def\makeop#1{\expandafter\def\csname#1\endcsname
  {\mathop{\rm #1}\nolimits}\ignorespaces}
\makeop{Hom}   \makeop{End}   \makeop{Aut}   %\makeop{Isom}
\makeop{Pic} \makeop{Gal}       \makeop{Div} \makeop{Lie}
\makeop{PGL}   \makeop{Corr} \makeop{PSL} \makeop{sgn} \makeop{Spf}
 \makeop{Tr} \makeop{Nr} \makeop{Fr} \makeop{disc}
\makeop{Proj} \makeop{supp} \makeop{ker}   \makeop{Im} \makeop{dom}
\makeop{coker} \makeop{Stab} \makeop{SO} \makeop{SL} \makeop{SL}
\makeop{Cl}    \makeop{cond} \makeop{Br} \makeop{inv} \makeop{rank}
\makeop{id}    \makeop{Fil} \makeop{Frac}  \makeop{GL} \makeop{SU}
\makeop{Trd}   \makeop{Sp} \makeop{Tr}    \makeop{Trd} \makeop{Res}
\makeop{ind} \makeop{depth} \makeop{Tr} \makeop{st} \makeop{Ad}
\makeop{Int} \makeop{tr}    \makeop{Sym} \makeop{can} \makeop{SO}
\makeop{torsion} \makeop{GSp} \makeop{Tor}\makeop{Ker} \makeop{rec}
\makeop{Ind} \makeop{Coker}
 \makeop{vol} \makeop{Ext} \makeop{gr} \makeop{ad}
 \makeop{Gr}\makeop{corank} \makeop{Ann}
\makeop{Hol} %Holomorphic
\makeop{Fitt} \makeop{Mp} \makeop{CAP}

\def\ord{{ord}}

\def\Ord{{\mathrm{ord}}}
\def\Spec{\mathrm{Spec}\,}

\DeclareMathAlphabet{\mathpzc}{OT1}{pzc}{m}{it}
\DeclareSymbolFont{cyrletters}{OT2}{wncyr}{m}{n}
\DeclareMathSymbol{\SHA}{\mathalpha}{cyrletters}{"58}

\def\makebb#1{\expandafter\def
  \csname bb#1\endcsname{{\mathbb{#1}}}\ignorespaces}
\def\makebf#1{\expandafter\def\csname bf#1\endcsname{{\bf
      #1}}\ignorespaces}
\def\makegr#1{\expandafter\def
  \csname gr#1\endcsname{{\mathfrak{#1}}}\ignorespaces}
\def\makescr#1{\expandafter\def
  \csname scr#1\endcsname{{\EuScript{#1}}}\ignorespaces}
\def\makecal#1{\expandafter\def\csname cal#1\endcsname{{\mathcal
      #1}}\ignorespaces}
\def\doLetters#1{#1A #1B #1C #1D #1E #1F #1G #1H #1I #1J #1K #1L #1M
                 #1N #1O #1P #1Q #1R #1S #1T #1U #1V #1W #1X #1Y #1Z}
\def\doletters#1{#1a #1b #1c #1d #1e #1f #1g #1h #1i #1j #1k #1l #1m
                 #1n #1o #1p #1q #1r #1s #1t #1u #1v #1w #1x #1y #1z}
\doLetters\makebb   \doLetters\makecal  \doLetters\makebf
\doLetters\makescr
\doletters\makebf   \doLetters\makegr   \doletters\makegr

\normalsize

\makeop{Ram} \makeop{Rep} \makeop{mass}

\makeop{Bl}
\def\abs#1{\left|#1\right|}
\def\norm#1{\lVert#1\rVert}

\def\Qp{\Q_p}
\def\Qbar{\ol{\Q}}
\def\Zbar{\ol{\Z}}

\def\Zp{\Z_p}

\def\rmT{{\mathrm T}}
\def\rmN{{\mathrm N}}

\def\rmd{{\mathrm d}}
\def\rmf{{\mathrm f}}

\def\bdsf{{\boldsymbol{e}}}
\def\bdsf{{\boldsymbol{f}}}
\def\bdsg{{\boldsymbol{g}}}
\def\bdsh{\boldsymbol h}

\def\bdsE{\boldsymbol{E}}

\def\bdsG{\boldsymbol{G}}

\def\cA{{\mathcal A}}  %automorphic forms
\def\cB{\EuScript B}
\def\cD{\mathcal D}
\def\cE{{\mathcal E}}
\def\cL{{\mathcal L}}

\def\cJ{\mathcal J}
\def\cK{{\mathcal K}}  %imaginary quadratic field
\def\cM{\mathcal M}

\def\cO{\mathcal O}
\def\cS{{\mathcal S}}
\def\cf{{\mathcal f}}
\def\cW{{\mathcal W}}

\def\cZ{\mathcal Z}

\def\cP{{\mathcal P}}
\def\cC{\mathcal C}
\def\cJ{\mathcal J}

\def\cQ{\mathcal Q}
\def\cU{\mathcal U}

\def\bfc{\mathbf c}

\def\bfK{\mathbf K}

\def\bfM{\mathbf M}
\def\bfU{\mathbf U}
\def\bfT{{\mathbf T}}

\def\bda{\mathbf a}

\def\bfu{\mathbf u}

\def\bdc{\mathbf c}

\def\bftheta{\boldsymbol{\theta}}

\def\bfal{\boldsymbol{\alpha}}

\def\sF{\mathscr F}
\def\sL{\mathscr L}

\def\sX{\mathscr X}

\def\sA{\mathscr A}

\def\sU{\mathscr U}

\def\sB{\mathscr B}

\def\bbI{\mathbb I}

\newcommand{\Z}{\mathbf Z}
\newcommand{\Q}{\mathbf Q}
\newcommand{\R}{\mathbf R}
\newcommand{\C}{\mathbf C}
\newcommand{\A}{\mathbf A}    % for adele

\def\fraka{{\mathfrak a}}

\def\frakc{{\mathfrak c}}
\def\frakd{\mathfrak d}
\def\frakf{\mathfrak f}

\def\frakp{{\mathfrak p}}

\def\frakq{\mathfrak q}

\def\frakt{\mathfrak t}

\def\frakm{\mathfrak m}
\def\frakn{\mathfrak n}

\def\frakD{\mathfrak D}

\def\frakH{{\mathfrak H}}

\def\frakX{\mathfrak X}

\def\frakN{\mathfrak N}

\def\wbar{\bar{w}}

\def\wbar{\ol{w}}

\def\padic{\text{$p$-adic }}

\def\BS{Bruhat-Schwartz }

\def\Teich{Teichm\"{u}ller }

\def\ot{\otimes}

\def\hookto{\hookrightarrow}
\def\longto{\longrightarrow}
\def\ol{\overline}  \nc{\opp}{\mathrm{opp}} \nc{\ul}{\underline}

\newcommand{\pair}[2]{\< #1, #2\>}

\newcommand{\pairing}{\pair{\,}{\,}}

\def\XYmatrix{\xymatrix@M=8pt} % make \xymatrix not too cluttered
\def\ncmd{\newcommand}
\ncmd{\xysubset}[1][r]{\ar@<-2.5pt>@{^(-}[#1]\ar@<2.5pt>@{_(-}[#1]}
\ncmd{\XYmatrixc}[1]{\vcenter{\XYmatrix{#1}}}
\ncmd{\xyto}[1][r]{\ar@{->}[#1]}
\ncmd{\xyinj}[1][r]{\ar@{^(->}[#1]}
\ncmd{\xysurj}[1][r]{\ar@{->>}[#1]}
\ncmd{\xyline}[1][r]{\ar@{-}[#1]}
\ncmd{\xydotsto}[1][r]{\ar@{.>}[#1]}
\ncmd{\xydots}[1][r]{\ar@{.}[#1]}
\ncmd{\xyleadsto}[1][r]{\ar@{~>}[#1]}
\ncmd{\xyeq}[1][r]{\ar@{=}[#1]} \ncmd{\xyequal}[1][r]{\ar@{=}[#1]}
\ncmd{\xyequals}[1][r]{\ar@{=}[#1]}
\ncmd{\xymapsto}[1][r]{l\ar@{|->}[#1]}\ncmd{\xyimplies}[1][r]{\ar@{=>}[#1]}
\ncmd{\xyiso}{\ar[r]_-{\sim}}
\def\injxy{\ar@{^(->}}

\newcommand{\pMX}[4]{\begin{pmatrix}
{#1}& {#2}\\
{#3}&{#4}\end{pmatrix} }

 \newcommand{\pDII}[2]{\begin{pmatrix}{#1}&0
 \\0&{#2}\end{pmatrix}}

\newcommand{\seesaw}[4]{{#1}\ar@{-}[rd]\ar@{-}[d]&{#2}\ar@{-}[d]\\
{#3}\ar@{-}[ru]&{#4}}

\def\ie{i.e. }

\def\cf{\mbox{{\it cf.} }}

\def\uf{\varpi} %uniformizer
\def\Abs{{|\!\cdot\!|}} %adelic absolute value
\def\Sg{{\varSigma}}  %%CM type

\def\ndivides{\nmid}

\def\x{{\times}}

\def\onehalf{{\frac{1}{2}}}
\def\al{\alpha}

\def\Lam{\Lambda}

\def\om{\omega}
\def\dirlim{\varinjlim}
\def\prolim{\varprojlim}
\def\iso{\simeq}
\def\con{\equiv}
\def\bksl{\backslash}
\newcommand\stt[1]{\left\{#1\right\}}
\def\ep{\epsilon}

\def\lam{\lambda}
\def\vep{\varepsilon}

\def\sg{\sigma}

\def\disjoint{\sqcup}
\def\bigot{\bigotimes}

\newcommand{\powerseries}[1]{\llbracket{#1}\rrbracket}

\renewcommand\pmod[1]{\,(\mbox{mod }{#1})}

\renewcommand\Re{\text{Re}\,}
\newcommand\Dmd[1]{\left<{#1}\right>} %Diamond operator

\def\Cp{\C_p}

\def\ram{{\mathrm{ram}}}

\def\Mat{\mathrm{M}}

\usepackage{hyperref}
\usepackage[english]{babel}
\usepackage{color}

\def\I{\mathbb{I}}
\def\<{\langle}
\def\>{\rangle}

\def\x{\times}

\def\<{\langle}
\def\>{\rangle}
\def\Om{{\boldsymbol \omega}}

\def\cyc{{\boldsymbol \epsilon}_{\rm cyc}}

\def\ep{\epsilon}
\def\Qx{{Q}}

\def\Qz{{P}}

\def\divides{\mid}
\def\Mu{\mu}
\def\Nu{\nu}
\def\eord{e}
\def\pmq{\ell}
\def\LV{C}

\def\cyc{{\boldsymbol \varepsilon}_{\rm cyc}}

\def\crit{{\rm crit}}
\def\p{\frakp}
\def\pbar{{\ol{\frakp}}}
\def\CMP{\theta}
\def\cmpt{\varsigma}
\def\cmptv{\cmpt_q}
\def\pme{q}
\def\Qq{\Q_q}
\def\w{\frakq}

\def\wbar{\ol{\frakq}}

\def\p{\frakp}
\def\setn{{(n)}}
\def\cK{F}
\def\Abs{\boldsymbol{\alpha}}
\newcommand\MS[2]{\stt{#1}-\stt{#2}}
\def\pmq{\frakq}
\def\Fv{F_\pmq}
\def\ufv{\uf_\pmq}
\def\qv{q_\pmq}
\def\Ov{\cO_\pmq}
\def\bdpsi{{\boldsymbol \psi}}
\def\Sg{\Sigma}
\def\brch{\phi}

\def\Cl{{\rm Cl}^+}
\def\wtsp{\sX}
\def\rmH{{\rm H}}

\title[Restriction of Hilbert-Eisenstein series]
{Restriction of Eisenstein series and Stark-Heegner points}\author{Ming-Lun Hsieh}
\author{Shunsuke Yamana}
\date{\today}
\address{Institute of Mathematics, Academia Sinica and National Center for Theoretic Sciences, Taipei 10617, Taiwan
}
\email{mlhsieh@math.sinica.edu.tw}
\address{Department of Mathematics, Graduate School of Science, Osaka City University, 3-3-138 Sugimoto, Sumiyoshi-ku, Osaka 558-8585, Japan}
\email{yamana@sci.osaka-cu.ac.jp}
\thanks{Hsieh was partially supported by a MOST grant 108-2628-M-001-009-MY4.
Yamana is partially supported by JSPS Grant-in-Aid for Scientific Research (C) 18K03210 and (B) 19H01778. 
This work is partially supported by Osaka City University Advanced Mathematical Institute (MEXT Joint Usage/Research Center on Mathematics and Theoretical Physics) and the Research Institute for Mathematical Sciences, an International Joint Usage/Research Center located in Kyoto University.}
\begin{document}

\maketitle

\begin{abstract}
In a recent work of Darmon, Pozzi and Vonk, the authors consider a particular $p$-adic family of Hilbert-Eisenstein series $E_k(1,\brch)$ associated with an odd character $\brch$ of the narrow ideal class group of a real quadratic field $F$ and compute the first derivative of a certain one-variable twisted triple product $p$-adic $L$-series attached to $E_k(1,\brch)$ and an elliptic newform $f$ of weight $2$ on $\Gamma_0(p)$. In this paper, we generalize their construction to include the cyclotomic variable and thus obtain a two-variable twisted triple product $p$-adic $L$-series. Moreover, when $f$ is associated with an elliptic curve $E$ over $\Q$, we prove that the first derivative of this $p$-adic $L$-series along the weight direction is a product of the $p$-adic logarithm of a Stark-Heegner point of $E$ over $F$ introduced by Darmon and the cyclotomic $p$-adic $L$-function for $E$.  
\end{abstract}

\section{Introduction}

In the work \cite{DPV}, to each odd character $\brch$ of the narrow ideal class group of a real quadratic field $F$, the authors associate a one-variable $p$-adic family $E^{(p)}_k(1,\brch)$ of Hilbert-Eisenstein series on $\Gamma_0(p)$ over a real quadratic field $F$ and investigate the connection between the spectral decomposition of the ordinary projection of the diagonal restriction $G_{k}(\brch)$ of $E^{(p)}_k(1,\brch)$ around $k=1$ and the $p$-adic logarithms of Gross-Stark units and Stark-Heegner points. In particular, if $p$ is inert in $F$ and let $G'_1(\brch)_{\Ord}$ denote the elliptic modular form of weight two obtained by the taking ordinary projection of the first derivative $\frac{d}{dk} G_k(\brch)|_{k=1}$, then it is proved in \cite[Theorem C(2)]{DPV} that the coefficient $\lam_f'$ of each normalized Hecke eigenform $f$ of weight two on $\Gamma_0(p)$ in the spectral decomposition of $G_1'(\brch)_\Ord$ can be expressed in terms of the product of special values of the $L$-function for $f$ and the $p$-adic logarithms of Stark-Heegner points or Gross-Stark units over $F$ introduced in \cite{Darmon01Ann} and \cite{DD06Ann}. 

The purpose of this paper is to provide some partial generalizations of this work to the two-variable setting by introducing the cyclotomic variable. To begin with,  %For each integer $N$ coprime to $p$, we shall construct a two-variable $p$-adic Hilbert-Eisenstein series $\bdsE_\brch(k,s)$ on $\Gamma_0(pN)$, which interpolates the $p$-depletion of $E^{(p)}_{k/2}(1,\brch)$ at $s=2$  in the case $N=1$.  We show the 
we let $F$ be a real quadratic field with different $\frakd$ over $\Q$. Let $x\mapsto \ol{x}$ denote the non-trivial automorphism of $F$ and let $\rmN:F\to\Q,\,\rmN(x)=x\ol{x}$ be the norm map. Let $\Delta_F$ be the discriminant of $F/\Q$. Let $\Cl(\cO_F)$ be the narrow ideal class group of $F$. Let $\brch:\Cl(\cO_F)\to \Qbar^\x$ be an odd narrow ideal class character, \ie $\brch((\delta))=-1$ for any $\delta\in\cO_F$ with $\ol{\delta}=-\delta$. Let $L(s,\brch)$ be the Hecke $L$-function attached to $\brch$.  
\def\bdsigma{{\boldsymbol \sigma}}
Fix an odd rational prime $p$ unramified in $F$. For $x\in\Zp^\x$, let $\Om(x)$ be the \Teich lift of $x\pmod{p}$ and let $\Dmd{x}:=x\Om^{-1}(x)\in 1+p\Zp$. Let $\wtsp:=\{x\in\Cp\mid \abs{x}_p\leq 1\}$ be the $p$-adic closed unit disk and let $A(\wtsp)$ be the ring of rigid analytic functions on $\wtsp$. Fix an embedding $\iota_p:\Qbar\hookto\Cp$ throughout. For each ideal $\frakm\lhd\cO_F$ coprime to $p$, define $\bdsigma_\brch(\frakm)\in A(\wtsp\times\wtsp)$ by  \[\bdsigma_{\brch}(\frakm)(k,s)=\sum_{\fraka\lhd\cO_F,\,\fraka\divides\frakm}\brch(\fraka)\Dmd{\rmN(\fraka)}^\frac{k-s}{2}\Dmd{\rmN(\frakm\fraka^{-1})}^\frac{s-2}{2}.\]
 Let $\wtsp^{\rm cl}:=\stt{k\in\Z^{\geq 2}\mid k\con 2\pmod{2(p-1)}}$ be the set of classical points in $\wtsp$. Let $h=\#\Cl(\cO_F)$. Fix a set $\stt{\frakt_\lam}_{\lam=1,\ldots,h}$ of representatives of the narrow ideal class group $\Cl(\cO_F)$ with $(\frakt_\lam,p\cO_F)=1$. For each classical point $k\in\wtsp^{\rm cl}$, the classical Hilbert-Eisenstein series $E_{\frac{k}{2}}(1,\brch)$ on $\SL_2(\cO_F)$ of parallel weight $\frac{k}{2}$ is determined by the normalized Fourier coefficients 
\[c(\frakm,E_\frac{k}{2}(1,\brch))= \bdsigma_\brch(\frakm)(k,2),\quad c_\lam(0,E_\frac{k}{2}(1,\brch))=4^{-1}L(1-k/2,\brch). \]

%For $s_0\in\wtsp$, let ${\rm ev}^*_{s_0}:\cA(\wtsp\times\wtsp)\toA(\wtsp)$ be the pull-back of the evaluation at $s=s_0$.

Let $I_F$ be the set of integral ideals of $F$. 
Assume that $\frakn\in I_F$ and $p$ are coprime. 
Let $\cM(\frakn)$ be the space of two-variable $p$-adic families of Hilbert modular forms 
%\footnote{Recall that a Hilbert semi-cusp form is a Hilbert modular form having no constant in the Fourier expansion around the cusps at the infinity.} 
of tame level $\frakn$, which consists of functions
\begin{align*}
I_F&\to \sA(\wtsp\times \wtsp), & &\frakm\mapsto c(\frakm,f), \\
\Cl(\cO_F)&\to\sA(\wtsp\times \wtsp), & &\fraka\mapsto c_0(\fraka,f)
\end{align*} such that the specialization $f(k,s)=\stt{c(\frakm,f)(k,s)}$ is the set of normalized Fourier coefficients of a $p$-adic Hilbert modular form of parallel weight $k$ on $\Gamma_0(p\frakn)$ for $(k,s)$ in a $p$-adically dense subset $U \subset \Zp\times \Zp$ (\cf\cite[p.535-536]{Wiles88}). 
Define  $\bdsE^{\stt{p}}_\brch:I_F\to A(\wtsp\times\wtsp)$ by the data
\begin{align*}
c(\frakm,\bdsE^{\stt{p}}_\brch)&=\bdsigma_\brch(\frakm)\text{ if }(\frakm,p\cO_F)=1,\\
c(\frakm,\bdsE^{\stt{p}}_\brch)&=0\text{ otherwise}, \\
c_0(\fraka, \bdsE^{\stt{p}}_\brch)&=0
\end{align*}
%The $q$-expansion 
% \powerseries{q^{\frakd^{-1}_+}}$ defined by \[\wh E^{\stt{p}}_\brch(k,s)_\lam(q)=\sum_{\beta\in(\frakd^{-1}\frakt_\lam)+,\,((\beta),p)=1} \bdsigma_{\phi}(\beta\frakd\frakt_\lam)(k,s)q^\beta.\] 
By definition, for $(k,s)\in\wtsp^{\rm cl}\times \wtsp^{\rm cl}$ with $k\geq 2s$, we have \[\bdsE^{\stt{p}}_\brch(k,s)=\Dmd{\Delta_F}^\frac{s-2}{2}\cdot \theta^\frac{s-2}{2}E^{\stt{p}}_{\frac{k+4-2s}{2}}(1,\phi),\]
where $E^{\stt{p}}_k(1,\brch)$ is the \emph{$p$-depletion} of $E_k(1,\brch)$ and $\theta$ is Serre's differential operator $\theta(\sum_\beta a_\beta q^\beta)=\sum_\beta \rmN_{F/\Q}(\beta)a_\beta q^\beta$. Therefore, $\bdsE^{\stt{p}}_\brch(k,s)$ is a $p$-adic Hilbert modular form of parallel weight $k$ for all $(k,s)\in\Zp^2$, and $\bdsE^{\stt{p}}_\brch\in \cM(\cO_F)$. For each prime ideal $\frakq$, define $\bfU_\frakq\colon \cM(\frakn)\to \cM(\frakn\frakq)$ by $c(\frakm,\bfU_\frakq f):=c(\frakm\frakq,f)$. Let $N$ be a positive integer such that $p\ndivides N$ and 
\beqcd{Splt}N\cO_F=\frakN\ol{\frakN},\quad (\frakN,\ol{\frakN})=1.\eeqcd
Define $\bdsE_\brch\in \cM(\frakN)$ by 
\[\bdsE_\phi:=\prod_{\frakq\divides \frakN}(1-\brch(\frakq)^{-1}\Dmd{\rmN(\frakq)}^\frac{2s-2-k}{2}\bfU_\frakq)\cdot \bdsE^{\stt{p}}_\brch\] 
and the \emph{diagonal restriction} $\bdsG_\brch\in A(\wtsp\times\wtsp)\powerseries{q}$ of $\bdsE_\phi$ by \[\bdsG_\brch:=\sum_{n>0}\Big(\sum_{\beta\in\frakd_+^{-1},\Tr(\beta)=n}c(\beta\frakd,\bdsE_\phi)\Big)q^n, \]
where $\frakd^{-1}_+$ is the additive semigroup of totally positive elements in $\frakd^{-1}$. 
 
By definition $\bdsG_\brch(k,s)$ is the $q$-expansion of a $p$-adic elliptic modular form on $\Gamma_0(pN)$ of weight $k$ obtained from the diagonal restriction of $\bdsE^{\stt{p}}_\brch(k,s)$ for $(k,s)\in\wtsp^{\rm cl}\times \wtsp^{\rm cl}$ with $k\geq 2s$. Let $\sU$ be an appropriate neighborhood around $2\in \wtsp$. Let $\bfS^\Ord(N)$ be the space of ordinary $A(\sU)$-adic elliptic cusp forms on $\Gamma_0(Np)$, consisting of $q$-expansion $\bdsf=\sum_{n>0}c(n,\bdsf)q^n\in A(\sU)\powerseries{q}$ whose weight $k$ specialization $\bdsf_k$ is a $p$-ordinary cusp form of weight $k$ on $\Gamma_0(pN)$ for $k\in\wtsp^{\rm cl}$. By Hida theory $\bfS^\Ord(N)$ is a free $A(\sU)$-module of finite rank. It can be shown that the image $\eord \bdsG_\brch$ under Hida's ordinary projector $e$ actually belongs to $\bfS^\Ord(N)\wh\ot_{A(\sU)}  A(\sU\times \wtsp)$, where $A(\sU)$ is regarded as a subring of $A(\sU\times\wtsp)$ via the pull-back of the first projection $\sU\times\wtsp\to\sU$. We can thus decompose
\[\eord\bdsG_\brch=\sum_{\bdsf}  \cL_{\bdsE_\brch,\bdsf}\cdot \bdsf+(\text{old forms}),\quad \cL_{\bdsE_\brch,\bdsf}\in A(\sU\times\wtsp),\]
where $\bdsf$ runs over the set of primitive Hida families of tame conductor $N$. 
We shall call $\cL_{\bdsE_\brch,\bdsf}\in A(\sU\times\wtsp)$ the twisted triple product $p$-adic $L$-function attached to the $p$-adic Hilbert-Eisenstein series $\bdsE_\brch$ and a primitive Hida family $\bdsf$. 

The arithmetic significance of this $p$-adic $L$-function stems from its connection with the Stark-Heegner points of elliptic curves. Let $E$ be an elliptic curve over $\Q$ of conductor $pN$. Assume that $p$ is inert in $F$. In \cite{Darmon01Ann}, Darmon introduced Stark-Heegner points of elliptic curves over real quadratic fields. These are local points in $E(F_p)$ but conjectured to be rational over ray class fields of $F$. The rationality of Stark-Heegner points has been one of the major open problems in algebraic number theory. Now let $\bdsf\in A(\sU)\powerseries{q}$ be a primitive Hida family of tame level $N$ such that the weight two specialization $f:=\bdsf_2$ is the elliptic newform associated with $E$. In the special case $N=1$, \cite[Theorem C(2)]{DPV} implies that that $\cL_{\bdsE_\brch,\bdsf}(2,1)=0$ and the first derivative of $\cL_{\bdsE_\brch,\bdsf}(k,1)$ at $k=2$ is essentially a product of the $p$-adic logarithm $\log_E P_\brch$ of the twisted Stark-Heegner point $P_\brch\in E(F_p)\ot\Q(\brch)$ introduced in \cite[(182)]{Darmon01Ann} and the central value $L(E,1)$ of the Hasse-Weil $L$-function of $E$. As pointed out in \cite[Remark 3]{DPV}, this connection has potential of providing a geometric approach to Stark-Heegner points via the $K$-theory of Hilbert modular surfaces. The main result of this paper (\thmref{T:main}) is to offer the following generalization of \cite[Theorem C(2)]{DPV} to include the cyclotomic variable and the case $N>1$.   
\begin{thma}Suppose that $p$ is inert in $F$ and the conductor $N$ satisfies \eqref{Splt}. Then $\cL_{\bdsE_\brch,\bdsf}(2,s)=0$ and 
\begin{align*}\frac{\partial\cL_{\bdsE_\brch,\bdsf}}{\partial k}(2,s+1)&=\onehalf(1+\brch(\frakN)^{-1}w_N) \log_E P_\brch\cdot  L_p(E,s)\frac{m_E^22^{\al(E)}}{c_f}\Dmd{\Delta_F}^\frac{s-1}{2},\end{align*}
where 
\begin{itemize}
\item $w_N\in\stt{\pm 1}$ is the sign of the Fricke involution at $N$ acting on $f$,
\item $L_p(E,s)$ is the Mazur-Tate-Teitelbaum $p$-adic $L$-function for $E$,
\item $c_f\in \Z^{>0}$ is the congruence number for $f$, $m_E\in\Q^\x$ is the Manin constant for $E$ and $2^{\al(E)}=[\rmH_1(E(\C),\Z):\rmH_1(E(\C),\Z)^+\oplus \rmH_1(E(\C),\Z)^-]$.
\end{itemize}
\end{thma}
Our main motivation for this two-variable generalization is that we have the non-vanishing of the $p$-adic $L$-function $L_p(E,s)$ thanks to Rohrlich's theorem \cite{Roh84Inv}, so we can still compute $\log_E P_\brch$ from the twisted triple product $p$-adic $L$-function even when the central value $L(E,1)$ vanishes. 
\begin{Remark}\noindent
\begin{itemize}
\item Note that the definition of Stark-Heegner points $P_\brch$ for odd $\brch$ in \cite{Darmon01Ann} depends on a choice of the purely imaginary period $\Omega_E^-$. In the above theorem, we require $(\sqrt{-1})^{-1}\Omega_E^-$ to be positive.
\item The Eisenstein contribution in the spectral decomposition in Part (2) of \cite[Theorem C]{DPV} is connected with the $p$-adic logarithms of Gross-Stark units over $F$, while in our two-variable setting, $\eord \bdsG_\brch$ is a $p$-adic family of cusp forms, so we do not get any information for Gross-Stark units. 
\item Theorem A only applies to real quadratic fields $F$ possessing a totally positive fundamental unit due to the existence of odd characters of the narrow ideal class group of $F$.  \end{itemize}
\end{Remark}
%\begin{Remark}In \thmref{T:main.7}, the prime $p$ is allowed to be split in $F$ and $\brch$ can be a general ring class character of conductor $C\cO_F$ split in $F$. all $p$-adic $L$-functions are constructed as an element in Iwasawa algebras.
%Roughly speaking, for $k\in\wtsp^{\rm cl}$, $C_\Xi(k)$ is the ratio of the product of plus/minus Shimura's periods $\Omega_{\bdsf_k}^\pm$ of $\bdsf_k$ and the Petersson norm of $\bdsf_k$. 
%\end{Remark}

We briefly outline the proof. Let $\cL_p(\bdsf/F,\phi,k)$ be the (odd) square-root $p$-adic $L$-function associated with the primitive Hida family $\bdsf$ and the character $\brch$ constructed in \cite[Definition 3.4]{BD09Ann} with $w_\infty=-1$ and let $L_p(\bdsf,k,s)$ be the Mazur-Kitagawa two-variable $p$-adic $L$-function so that $L_p(\bdsf,2,s)$ is the cyclotomic $p$-adic $L$-function for $\bdsf_2$.  The main point of the proof is to prove the following factorization formula of $\cL_{\bdsE_\brch,\bdsf}$:
\beq\label{E:fac.1}C^*(k)\cdot \cL_{\bdsE_\brch,\bdsf}(k,s+1)=4\Dmd{\Delta_F}^\frac{s-k+1}{2}\cdot\cL_p(\bdsf/F,\phi,k)\cdot L_p(\bdsf,k,s),\eeq
where $C^*(k)$ is a meromorphic function on $\wtsp$ holomorphic at all classical points $k\in\wtsp^{\rm cl}$ with $C^*(2)=1$. By construction, the square root $p$-adic $L$-function $\cL_p(\bdsf/F,\phi,k)$ interpolates the toric period integrals $B^\brch_{\bdsf_k}$. Thus we get $\cL_{\bdsE_\brch,\bdsf}(2,s)=\cL_p(\bdsf/F,\phi,2)=0$ by a classical theorem of Saito and Tunnell. Moreover, from the formula \cite[Corollary 2.6]{BD09Ann}, it is not difficult to deduce that the first derivative of $\cL_p(\bdsf/F,\phi,k)$ at $k=2$ is $\frac{1}{2}(1+w_N\brch(\frakN)^{-1})\log_EP_\brch$, and hence we obtain Theorem A from \eqref{E:fac.1}. \\

The factorization formula \eqref{E:fac.1}, proved in \thmref{T:main.7}, is established by the inspection of the explicit interpolation formulae on both sides. In particular, the interpolation formula of $\cL_{\bdsE_\brch,\bdsf}(k,s)$ (\propref{P:interp1}) is the most technical part of this paper. Roughly speaking, for $(k,s)\in \wtsp^{\rm cl}\times \wtsp^{\rm cl}$ with $k\geq 2s$, Hida's $p$-adic Rankin-Selberg method shows that $\cL_{\bdsE_\brch,\bdsf}(k,s)$ is interpolated by the inner product between the diagonal restriction of a nearly holomorphic Hilbert-Eisenstein series $\bdsE_\brch(k,s)$ and $\bdsf_k$. Therefore, a result of Keaton and Pitale \cite[Proposition 2.3]{KP19DM} tells us that $\cL_{\bdsE_\brch,\bdsf}(k,s)$ is a product of \begin{itemize}
\item[(i)] the toric period integral  $B^\brch_{\bdsf_k}$ of $\bdsf_k$ over $F$ twisted by $\brch$ (see \eqref{E:toric}),
\item[(ii)] the special value $L(\bdsf_k,s)$ of the $L$-function for $\bdsf_k$;
\item[(iii)] local zeta integrals $Z_\cD(s,B_{W_v})$ for every place of $\Q$ in \eqref{E:locazeta}. \end{itemize} It is known that items (i) and (ii) are basically interpolated by $\cL_p(\bdsf/K,\brch,k)$ and $L_p(\bdsf,k,s)$, so our main task is to evaluate explicitly these local zeta integrals in item (iii). These calculations occupy the main body of Section 4. From the explicit interpolation formulae of these $p$-adic $L$-functions, we can deduce that the ratio $C^*$ between $\cL_p(\bdsf/F,\phi,k)\cdot L_p(\bdsf,k,s)$ and $\cL_{\bdsE_\brch,\bdsf}(k,s+1)$ is independent of $s$, and hence $C^*$ is a meromorphic function in $k$ only. Finally, by a standard argument using Rohrlich's result on the non-vanishing of the cyclotomic $p$-adic $L$-functions for elliptic modular forms, we can conclude that $C^*(k)$ is holomorphic at all $k\in\wtsp^{\rm cl}$ and $C^*(2)$ is essentially the congruence number. 

This paper is organized as follows. After preparing the basic notation for modular forms and automorphic forms in Section 2,  we give  the construction of Hilbert-Eisenstein series and compute the Fourier coefficients in Section 3.  In Section 4, we compute the inner product between the diagonal restriction of Hilbert-Eisenstein series and a $p$-stabilized newform. The main local calculations are carried out in \propref{P:spl} for the split case, \propref{P:nonsplit} for the non-split, and \propref{P:padic} for the $p$-adic case. In Section 5, we use $p$-adic Rankin-Selberg method to construct the $p$-adic $L$-function $\cL_{\bdsE_\brch,\bdsf}$ and obtain the interpolation formula in \propref{P:interp1} by combining the local calculations in Section 4. In order to make the comparison between $p$-adic $L$-functions easier, the interpolation formulae shall be presented in terms of automorphic $L$-functions in this paper. In Section 6, we review the theory of $\Lam$-adic modular symbols in \cite{Kitagawa94} and the construction of the square root $p$-adic $L$-function $\cL_p(\bdsf/F,\brch,k)$. Our treatment for modular symbols is \emph{semi-adelic}, which allows simple descriptions of Hecke actions and are amenable to the calculations from the automorphic side. The connection with  Greenberg-Stevens' approach \cite{GS1993} is explained in \subsecref{R:GS.6}. In \propref{P:interp2}, we give the complete interpolation formula for $\cL_p(\bdsf/F,\brch,k)$, including the evaluation at finite order characters of $p$-power conductors. Finally, we deduce the factorization formula and the derivative formula for $\cL_{\bdsE_\brch,\bdsf}$ in Section 7.
%Let $\mu_\bdsf^{\rm GS}\in{\rm MS}_{\Gamma_0(N)}(\cD((\Zp^2)')$ be a $\Lam$-eigensymbol with valued in the distribution $\cD((\Zp^2)')$ on the set of primitive elements in $\Zp^2$. After a normalization, we may assume 
%${\rm sp}_2(\Xi^\pm)$ is $\frac{1}{\Omega^\pm_{\bdsf_2}}\xi^\pm_{\bdsf_2}$.
%We recall the construction of the square root of $p$-adic $L$-function for $\bdsf/\cK$ in \cite{BD09Ann}. Let $\cE_{N}$ be the set of optimal embeddings from $\cO_F$ into $R_N$. There is a bijection between the classes of ideals $[\fraka]\in\Cl(\cO_F)$ and the optimal embeddings $\Psi_\fraka\in \cE_{N}/\Gamma_0(N)$ (\cf\cite[Prop. 6.2.1]{Popa06Comp}). 
%For each ideal class $[\fraka]\in\Cl(\cO_F)$, let $P_\fraka(X,Y)$ be the quadratic form associated with the symmetric matrix $\pMX{0}{-1}{1}{0}\Psi_{\fraka}(\sqrt{\Delta_F}/2)$.  Let $\ep_F$ be the totally positive fundamental unit of $\cO_F^\x$.  Taking any base point $r\in\bfP^1(\Q)$, we define the $p$-adic analytic function $\cL_{\Xi/F\ot\phi}\colon\wtsp\to\Cp$ by 
%\begin{align*}\cL_{\Xi/F\ot\phi}(k)&=\sum_{[\fraka]\in\Cl(\cO_F)}\phi(\fraka)\Dmd{\rmN(\fraka)}^\frac{2-k}{2}\\
%&\,\times\int_{P_{\fraka}(x,y)\in\Zp^\x}\Dmd{P_{\fraka}(x,y)}^{\frac{k-2}{2}}\mu^{\rm GS}_{\Xi}(\MS{\Psi_\fraka(\ep_c)r}{r})(x,y).\end{align*}

\subsection*{Acknowledgements} We thank the referees for careful reading and helpful suggestions on the improvement of the earlier version of the paper. This paper was written during the first author's visit to Osaka City University and RIMS in January 2020. He is grateful for their hospitality.  

\section{Classical modular forms and automorphic forms}
\label{S:modularforms}
In this section, we recall basic definitions and standard facts about classical elliptic modular forms and automorphic forms on $\GL_2(\A)$, following the notation in \cite[\S 2]{Hsieh2017} which we reproduce here for the reader's convenience. The main purpose of this section is to set up the notation and introduce some Hecke operators on the space of automorphic forms which will be frequently used in the construction of $p$-adic $L$-functions.

\subsection{Notation} 
We denote by $\Z$, $\Q$, $\R$, $\C$, $\A$, $\R_+$ the ring of rational integers, the field of rational, real, complex numbers, the ring of adeles of $\Q$ and the group of strictly positive real numbers. 
Let $\mu_n(F)$ denote the group of $n$th roots of unity in a field $F$.
For a rational prime $\ell$ we denote by $\Z_\ell$, $\Q_\ell$ and ${\rm ord}_\ell : \Q_\ell \rightarrow \Z$ the ring of $\ell$-adic integers, the field of $\ell$-adic numbers and the additive valuation normalized so that ${\rm ord}_\ell(\ell)=1$. 
Put $\widehat{\Z}=\prod_\ell\Z_\ell$. 
Define the idele $\uf_\ell=(\uf_{\ell,v})\in\A^\x$ by $\uf_{\ell,\ell}=\ell$ and $\uf_{\ell,v}=1$ if $v\not =\ell$. 

Let $F$ be a number field.  
We denote its integer ring by $\cO_F$. 
We write $\rmT_{F/\Q}$ and $\rmN_{F/\Q}$ for the trace and norm from $F$ to $\Q$. 
For each place $v$ of $F$ we denote by $F_v$ the completion of $F$ with respect to $v$. 
Let $\A_F=\A\otimes_\Q F$ be the adele ring of $F$. 
Given $t\in \A_F^{\times}$, we write $t_v \in F_v^{\times}$ for its $v$-component. 
We shall regard $F_v$ (resp. $F_v^{\times}$) as a subgroup of $\A_F$ (resp. $\A_F^{\times}$) in a natural way. 
Let $\Abs_{F_v}=|\mbox{ }|_{F_v}$ be the normalized absolute value on $F_v$. 
If $v=\pmq$ is finite, then $|\varpi_\pmq|_{F_\pmq}=q_\pmq^{-1}$, where $\varpi_\pmq$ is a generator of the prime ideal of the integral ring $\cO_\pmq$ of $F_\pmq$ and $q_\pmq$ denotes the cardinality of the residue field of $\cO_\pmq$. 
% so that $|\mbox{ }|_\infty$ is the usual absolute value of $\R$, and $|\ell|_{\Q_\ell}=\ell^{-1}$. Let $|\mbox{ }|_\A$ be the absolute value on $\A^{\times}$ given by $|a|_\A=\prod_v |a_v|_{\Q_v}$. 
Define the complete Dedekind zeta function by $\zeta_F(s)=\prod_v\zeta_{F_v}(s)$, where $\zeta_\R(s)=\pi^{-s/2}\Gamma\bigl(\frac{s}{2}\bigl)$, and if $v=\pmq$ is finite, then $\zeta_{F_\pmq}(s)=(1-q_\pmq^{-s})^{-1}$. 
When $F=\Q$, we will write $\Abs_v=|\mbox{ }|_v$ and $\zeta_v(s)=\zeta_{\Q_v}(s)$. 
Let $\bdpsi: \A / \Q \rightarrow \C^{\times}$ be the additive character whose archimedean component is $\bdpsi_\infty(x)=e^{2\pi \sqrt{-1}x}$ and whose local component at $\ell$ is denoted by $\bdpsi_\ell:\Q_\ell \rightarrow \C^{\times}$. 
We define the additive character $\bdpsi_F=\prod_v\bdpsi_{F_v}: \A_F / F \rightarrow \C^{\times}$ by setting $\bdpsi_F:=\bdpsi\circ\Tr_{F/\Q}$. 
Let $\cS(\A_F^m)=\otimes_v'\cS(F_v^m)$ denote the space of Schwartz functions on $\A_F^m$. 

For any set $X$ we denote by $\mathbb{I}_X$ the characteristic function of $X$.
If $R$ is a commutative ring and $G=\GL_2(R)$, we define homomorphisms $\bft:R^\times\to G$ and $\bfn:R\to G$ by 
\begin{align*}
\bft(a)&=\pDII{a}{1}, & \bfn(x)&= \pMX{1}{x}{0}{1}. 
\end{align*}
The identity matrix in $G$ is denoted by ${\bf1}_2$. Denote by $\rho$ the right translation of $G$ on the space of $\C$-valued functions on $G$, i.e., $\rho(g)f(g')=f(g'g)$, and by ${\bf 1} : G \rightarrow \C$ the constant function ${\bf 1}(g)=1$. For a function $f : G \rightarrow \C$ and a character $\omega : R^{\times} \rightarrow \C^{\times}$, let $f\otimes \omega : G \rightarrow \C$ denote the function $f\otimes \omega(g)=f(g)\omega(\det g)$. 
The subgroup $B(R)$ (resp. $N(R)$) of $\GL_2(R)$ consists of upper triangular (resp. upper triangular unipotent) matrices. 
%Let $G_{\Q}={\rm Gal}(\overline{\Q}/ \Q)$ be the absolute Galois group of $\Q$ and if $\chi : (\Z / N\Z)^{\times} \rightarrow \C^{\times}$ is a Dirichlet character modulo $N$, denote by $c_\ell(\chi)\leq {\rm ord}_\ell(N)$ the $\ell$-component of the conductor of $\chi$. We shall identify $\chi$ with the Galois character $\chi : G_{\Q} \rightarrow \C^{\times}$ via class field theory.

\subsection{Characters}\label{SS:Dirichlet}
If $F$ is a number field and $\chi : F^{\times}\backslash \A_F^{\times}\rightarrow \overline{\Q}^{\times}$ is a Hecke character of $\A_F^\x$,  we denote by $\chi_v: F_v^\x \rightarrow \C^{\times}$ the local component of $\chi$ at a place $v$ of $F$.
When $\omega$ is a Hecke character of $\A^\x$, we denote by $\omega_F:=\omega\circ\rmN_{F/\Q}:F^\x\bksl\A_F^\x\to\C^\x$ \emph{the base change} of $\omega$. 

If $v$ is non-archimedean and $\lam:F_v^\x\to\C^\x$ is a character, let $c(\lam)$ be the exponent of the conductor of $\lam$.
%Define the $p$-adic cyclotomic character by \[\cyc:\Q^\x\bksl \A^\x\to\Zp^\x,\quad \cyc(a)=\abs{a}_\A a_\infty^{-1}a_p.\] Let $\Om:\Q^\x\bksl \A^\x\to\mu_{p-1}(\Qbar_p)$ be the \Teich character.

%Let $\chi_\A$ be the \emph{adelization} of $\chi$, the unique finite order Hecke character $\chi_\A=\prod\chi_v: F^{\times}\backslash \A_F^{\times}/(F\ot\R)_+(1+\frakc(\chi)\widehat{\cO}_F)^{\times}\rightarrow\C^{\times}$ of conductor $\frakc(\chi)$ such that for any finite place $\frakl\nmid \frakc(\chi)$,
 %\[\chi_\frakl(\uf_\frakl)=\chi(\frakl)^{-1}.\]
%For every prime $\ell$, write $\frakc(\chi)=\ell^eC'$ with $\ell\ndivides C'$. Then we can decompose $\chi=\chi_{(\ell)}\chi^{(\ell)}$ into a product of two Dirichlet characters $\chi_{(\ell)}$ and $\chi^{(\ell)}$ of conductors $\ell^e$ and $N'$ respectively. We call $\chi_{(\ell)}$ the $\ell$-primary component of $\chi$. The $\ell$-primary component of a finite order Hecke character can be defined likewise.
 
%Throughout this paper, we often identify Dirichlet characters with their adelization whenever no confusion arises. 

\subsection{Automorphic forms on $\GL_2(\A)$}\label{SS:auto}
Fix a positive integer $N$.  
Define open compact subgroups of $\GL_2(\widehat{\Z})$ by
\begin{align*}
U_0(N)=&\left \{g\in\GL_2(\widehat{\Z})\mbox{ }\left \vert \mbox{ }  g\equiv \pMX{*}{*}{0}{*} (\mbox{mod }{N\widehat{\Z}}) \right .\right \},\\
U_1(N)=&\left \{g\in U_0(N)\mbox{ }\left \vert \mbox{ }g\equiv \pMX{*}{*}{0}{1} (\mbox{mod }{N\widehat{\Z}}) \right . \right \}.
\end{align*}
Let $\omega:\Q^{\times}\backslash \A^{\times}\rightarrow\C^{\times}$ be a finite order Hecke character of level $N$. 
We extend $\omega$ to a character of $U_0(N)$ defined by $\omega \left (\pMX{a}{b}{c}{d} \right )=\prod_{\ell \mid N}\omega_{\ell}(d_{\ell})$ for $\pMX{a}{b}{c}{d} \in U_0(N)$. 
%Denote by $\mathcal{A}(\omega)$ the space of automorphic forms on $\GL_2(\A)$ with central character $\omega$. 
For any integer $k$ the space $\mathcal{A}_k(N,\omega)$ of automorphic forms on $\GL_2(\A)$ of weight $k$, level $N$ and character $\omega$ consists of automorphic forms $\varphi:\GL_2(\A)\rightarrow \C$ such that 
\begin{align*}
\varphi(z\gamma g\kappa_{\theta} u_{\rm f})=&\omega(z)\varphi(g)e^{\sqrt{-1}k\theta}\omega(u_{\rm f}), & \kappa_{\theta}&=\pMX{\cos\theta}{\sin\theta}{-\sin\theta}{\cos\theta}
\end{align*}
for $z\in\A^\times$, $\gamma\in\GL_2(\Q)$, $\theta\in\R$ and $u_{\rm f}\in U_0(N)$. 
Let $\mathcal{A}^0_k(N,\omega)$ be the space of cusp forms in $\mathcal{A}_k(N,\omega)$.

Next we introduce important local Hecke operators on automorphic forms. At the archimedean place, let $V_{\pm}:\mathcal{A}_k(N,\omega)\rightarrow\mathcal{A}_{k\pm 2}(N,\omega)$ be the normalized weight raising/lowering operator in \cite[page 165]{JLbook} given by 
\begin{align*}%\begin{aligned}
V_{\pm}=\frac{1}{(-8\pi)}\left(\pMX{1}{0}{0}{-1}\otimes 1\pm\pMX{0}{1}{1}{0}\otimes \sqrt{-1}\right)\in{\rm Lie}(\GL_2(\R))\otimes_\R\C.
%\LR_{\pm}:=&\frac{1}{(-8\pi)}V_\pm.\end{aligned}
\end{align*}
Define the operator ${\bf U}_\ell$ acting on $\varphi\in\mathcal{A}_k(N,\omega)$ by 
\[{\bf U}_\ell\varphi=\sum_{x\in\Z_\ell/\ell\Z_\ell}\rho \left(\pMX{\varpi_\ell}{x}{0}{1} \right)\varphi,\]
and the level-raising operator $V_\ell:\mathcal{A}_k(N,\omega)\rightarrow\mathcal{A}_k(N\ell,\omega)$ at a finite prime $\ell$ by \[V_\ell\varphi(g):=\rho(\bft(\varpi_{\ell}^{-1}))\varphi.\]
Note that ${\bf U}_\ell V_\ell\varphi=\ell\varphi$ and that if $\ell\mid N$, then ${\bf U}_\ell\in{\rm End}_\C\mathcal{A}_k(N,\omega)$. For each prime $\ell\nmid N$, let $T_\ell\in{\rm End}_\C\mathcal{A}_k(N,\omega)$ be the usual Hecke operator defined by
\[T_\ell={\bf U}_\ell+\omega_\ell(\ell)V_\ell.\]
%Let $\mathcal{A}^0(\omega)$ be the space of cusp forms in $\mathcal{A}(\omega)$ and let $\cA^0_k(N,\om)=\cA^0(\om)\cap\cA_k(N,\om)$. 
Define the $\GL_2(\A)$-equivariant pairing $\<\, , \,\>:\mathcal{A}^0_{-k}(N,\omega)\ot \mathcal{A}_k(N,\omega^{-1})\to\C$ by
\[\<{\varphi},{\varphi'}\>=\int_{\A^{\times}\GL_2(\Q)\backslash\GL_2(\A)}\varphi(g)\varphi'(g)\,\rmd^\tau g, \]
 where $\rmd^\tau g$ is the Tamagawa measure of $\PGL_2(\A)$. 
 Note that $\pair{T_\ell \varphi}{\varphi'}=\pair{\varphi}{T_\ell \varphi'}$ for $\ell\ndivides N$. 
 
\subsection{Classical modular forms}\label{SS:classical}
We recall a \emph{semi-adelic} description of classical modular forms. Let $C^\infty(\frakH)$ be the space of $\C$-valued smooth functions on the half complex plane $\frakH:=\{z\in\C\;|\;{\rm Im}(z)>0\}$.  
The group $\GL_2(\R)^+:=\{g\in\GL_2(\R)\;|\;\det g>0\}$ acts on $\frakH$ and the automorphy factor is given by 
\begin{align*}
\gamma(z)&=\frac{az+b}{cz+d}, &  
J(\gamma,z)&=cz+d 
\end{align*}
for $\gamma=\pMX{a}{b}{c}{d}\in\GL_2(\R)^+$ and $z\in\frakH$. 

Let $k$ be any integer.  
The Maass-Shimura differential operators $\delta_k$ and $\varepsilon$ on $C^\infty(\frakH)$ are defined by 
\begin{align*}
\delta_k&=\frac{1}{2\pi\sqrt{-1}}\biggl(\frac{\partial }{\partial z}+\frac{k}{2\sqrt{-1}y}\biggl), & 
\varepsilon&=-\frac{1}{2\pi\sqrt{-1}}y^2\frac{\partial}{\partial \overline{z}} 
\end{align*}
(cf. \cite[(1a,\,1b) page 310]{Hida93Blue}), where $y={\rm Im}(z)$ is the imaginary part of $z$. 
Let $\chi$ be a Dirichlet character of level $N$. 
For a non-negative integer $m$ let $\mathcal{N}^{[m]}_k(N,\chi)$ denote the space of nearly holomorphic modular forms of weight $k$, level $N$ and character $\chi$. 
In other words $\mathcal{N}^{[m]}_k(N,\chi)$ consists of smooth slowly increasing functions $f:\frakH\times \GL_2(\wh\Q)\to\C$ such that 
 \begin{itemize}
 \item $f(\gamma z,\gamma g_\rmf u)=(\det \gamma)^{-1}J(\gamma,z)^kf(z,g_\rmf)\chi^{-1}(u)$
 for any $\gamma\in\GL_2(\Q)^+$ and $u\in U_0(N)$;
\item $\varepsilon^{m+1} f(z,g_\rmf)=0$ \end{itemize}
 (cf. \cite[page 314]{Hida93Blue}). 
 Let $\mathcal{N}_k(N,\chi)=\bigcup_{m=0}^\infty \mathcal{N}^{[m]}_k(N,\chi)$ (cf. \cite[(1a), page 310]{Hida93Blue}). 
By definition $\mathcal{N}^{[0]}_k(N,\chi)$ coincides with the space $\mathcal{M}_k(N,\chi)$ of classical holomorphic modular forms of weight $k$, level $N$ and character $\chi$.  Denote by $\mathcal{S}_k(N,\chi)$ the space of cusp forms in $\mathcal{M}_k(N,\chi)$. 
Let $\delta_k^m=\delta_{k+2m-2}\cdots\delta_{k+2}\delta_k$. 
If $f\in \mathcal{N}_k(N,\chi)$, then $\delta_k^mf\in\mathcal{N}_{k+2m}(N,\chi)$ (\cite[page 312]{Hida93Blue}). 
Given a positive integer $d$, we define 
\begin{align*}
V_df(z,g_\rmf)&=f(dz,g_\rmf); & 
{\bf U}_df(z,g_\rmf)&=\sum_{j=0}^{d-1} f\biggl(z,g_\rmf\pMX{d}{j}{0}{1}\biggl). 
\end{align*}
The classical Hecke operators $T_\ell$ for primes $\ell\nmid N$ are given by
\[T_\ell f={\bf U}_\ell f+\chi_\ell(\ell^{-1})\ell^{k-2}V_\ell f.\]
We say that $f\in \mathcal{N}_k(N,\chi)$ is a \emph{Hecke eigenform} if $f$ is an eigenfunction of all the Hecke operators $T_\ell$ for $\ell\nmid N$ and the operators ${\bf U}_\ell$ for $\ell\mid N$.

\subsection{}\label{S:2.3} To every nearly holomorphic modular form $f\in\mathcal{N}_k(N,\chi)$ we associate a unique automorphic form $\mathit{\Phi}(f)\in\mathcal{A}_k(N,\chi^{-1})$ defined by the formula
\begin{align}\label{E:MA2}\mathit{\Phi}(f)(g):= f(g_\infty(\sqrt{-1}) ,g_{\rmf})J(g_\infty,\sqrt{-1})^{-k}(\det g_\infty) \abs{\det g}_\A^{\frac{k}{2}-1}\end{align}
for $g=g_\infty g_\rmf\in \GL_2(\R)\GL_2(\wh\Q)$ (cf. \cite[\S 3]{Casselman1973}). Conversely, we can recover the form $f$ from $\mathit{\Phi}(f)$ by 
\begin{align}\label{E:MA1}f(x+\sqrt{-1} y,\,g_\rmf)=y^{-{k}/{2}}\mathit{\Phi}(f)\left(\pMX{y}{x}{0}{1}g_\rmf\right)\abs{\det g_\rmf}_\A^{1-\frac{k}{2}}.\end{align}
We call $\varPhi(f)$ the \emph{adelic lift} of $f$.

The weight raising/lowering operators are the adelic avatar of the differential operators $\delta_k^m$ and $\varepsilon$ on the space of automorphic forms. A direct computation shows that the map $\varPhi$ from the space of modular forms to the space of automorphic forms is \emph{equivariant} for the Hecke action in the sense that
\begin{align}\label{E:diffop1}
\mathit{\Phi}(\delta_k^mf)&=V_+^m\mathit{\Phi}(f), & 
\mathit{\Phi}(\varepsilon f)&=V_-\mathit{\Phi}(f),
\intertext{and for a finite prime $\ell$}\notag
\mathit{\Phi}(T_\ell f)&=\ell^{{k}/{2}-1}T_\ell\mathit{\Phi}(f), & 
\mathit{\Phi}({\bf U}_\ell f)&=\ell^{{k}/{2}-1}{\bf U}_\ell\mathit{\Phi}(f).
\end{align}
In particular, $f$ is holomorphic if and only if $V_-\mathit{\Phi}(f)=0$. 
%For $f\in \mathcal{M}_k(N,\chi)$ and $\kappa$ a Dirichlet character modulo a $\ell$-power, one verifies that
%\begin{align}\mathit{\Phi}(f|[\kappa])=\theta_\ell^\kappa\mathit{\Phi}(f)\otimes\kappa_\A^{-1}.\end{align}

\subsection{Preliminaries on irreducible representations of $\GL_2(\Q_v)$}

\subsubsection{Measures}\label{SS:measure}
We shall normalize the Haar measures on $F_v$ and $F_v^{\times}$ as follows. 
Let $\rmd x_v$ be the self-dual Haar measures of $F_v$ with respect to $\bdpsi_{F_v}$. 
Put $\rmd^\x x_v=\zeta_{F_v}(1)\frac{\rmd x_v}{|x_v|_{F_v}}$. 
If $F=\Q$, then $\rmd a_\infty$ denote the usual Lebesgue measure on $\R$ and $\rmd a_\ell$ be the Haar measure on $\Q_\ell$ with ${\rm vol}(\Z_\ell,\rmd a_\ell)=1$. 
The Tamagawa measure of $\A_F$ is $\rmd x=\prod_v\rmd x_v$ while the Tamagawa measure of $\A^\x_F$ is defined by $\rmd^\x x=c_F^{-1}\prod_v\rmd x_v^\x$, where $c_F$ denotes the residue of $\zeta_F(s)$ at $s=1$. 
Define the compact subgroup $\bfK=\prod_v\bfK_v$ of $\GL_2(\A)$ by ${\bf K}_\infty={\rm SO}(2,\R)$ and ${\bf K}_\ell=\GL_2(\Z_\ell)$. 
Let $\rmd u_v$ be the Haar measure on ${\bf K}_v$ so that ${\rm vol}({\bf K}_v,\rmd u_v)=1$. Let $\rmd g_v$ be the Haar measure on $\PGL_2(\Q_v)$ given by $\rmd^\tau g_v=\abs{a_v}_v^{-1}\rmd x_v\rmd^{\times} a_v\rmd u_v$ for $g_v=\pMX{a_v}{x_v}{0}{1}u_v$ with $a_v\in\Q_v^{\times}$, $x_v\in\Q_v$ and $u_v\in {\bf K}_v$. 
The Tamagawa measure on $\PGL_2(\A)$ is given by $\rmd^\tau g=\zeta_\Q(2)^{-1}\prod_v\rmd^\tau g_v$. 

\subsubsection{Representations of $\GL_2(\Q_v)$}
Denote by ${\varrho}\boxplus{\upsilon}$ the irreducible principal series representation of $\GL_2(\Q_v)$ attached to two characters $\varrho,\upsilon:\Q_v^{\times}\rightarrow\C^{\times}$ such that $\varrho\upsilon^{-1}\neq \Abs_v^{\pm1}$. 
If $v=\infty$ is the archimedean place and $k\geq 1$ is an integer, denote by $\mathcal{D}_0(k)$ the discrete series of lowest weight $k$ if $k\geq 2$ or the limit of discrete series if $k=1$ with central character $\sgn^k$ (the $k$-th power of the sign character $\sgn(x)=\frac{x}{|x|_\infty}$ of $\R^\times$).
%If $v$ is finite, we denote by ${\rm St}_v$ the Steinberg representation of $\GL_2(\Q_v)$ and by $\varrho{\rm St}_v$ its twist by $\varrho$.

%\subsubsection{Conductors and new vectors}Let $\ell$ be a prime. Let $(\pi,\mathcal{V}_\pi)$ be an irreducible admissible infinite dimensional representation of $\GL_2(\Q_\ell)$, where $\mathcal{V}_\pi$ a realization of $\pi$. For a non-negative integer $n$, let 
%\cU_0(\ell^n)=\GL_2(\Z_\ell)\cap \pMX{\Z_\ell}{\Z_\ell}{\ell^n\Z_\ell}{\Z_\ell};\quad  \mathcal{U}_0(\ell^n)=\GL_2(\Z_\ell)\cap \pMX{\Z_\ell}{\Z_\ell}{\ell^n\Z_\ell}{\Z_\ell}, \quad 
%\[\mathcal{U}_1(\ell^n)=\GL_2(\Z_\ell)\cap \pMX{\Z_\ell}{\Z_\ell}{\ell^n\Z_\ell}{1+\ell^n\Z_\ell}.\] 
%Let $c(\pi)$ be the exponent of the conductor of $\pi$. By definition, $c(\pi)$ is the smallest integer such that $\mathcal{V}_\pi^{\mathcal{U}_1(\ell^{c(\pi)})}$ the space of $\mathcal{U}_1(\ell^{c(\pi)})$-fixed vectors is non-zero. 

\subsubsection{Whittaker models and the normalized Whittaker newforms}\label{SS:Wnew}

Every irreducible admissible infinite dimensional representation $\pi$ of $\GL_2(\Q_v)$ admits a Whittaker model $\mathcal{W}(\pi)=\mathcal{W}(\pi,\bdpsi_v)$ with respect to $\bdpsi_v$. 
Recall that $\mathcal{W}(\pi)$ is a subspace of smooth functions $W:\GL_2(\Q_v)\rightarrow\C$ such that
\begin{itemize}
\item $W(\bfn(x) g)=\bdpsi_v(x)W(g)$ for all $x\in\Q_v$,
\item if $v=\infty$ is archimedean, then there exists an integer $M$ such that 
\[W (\bft(a))=O(|{a}|_\infty^M)\text{ as }|{a}|_\infty\to\infty.\] \end{itemize}
The group $\GL_2(\Q_v)$ (or the Hecke algebra of $\GL_2(\Q_v)$) acts on $\mathcal{W}(\pi)$ via the right translation $\rho$. We introduce the (normalized) \emph{local Whittaker newform} $W_{\pi}$ in $\mathcal{W}(\pi)$ in the following way: 
if $v=\infty$ and $\pi=\mathcal{D}_0(k)$, then $W_{\pi}\in\mathcal{W}(\pi)$ is defined by
\begin{align}
\label{E:Winfty.1}
W_{\pi}\left(z\pMX{y}{x}{0}{1}\kappa_\theta\right)
=\I_{\R_+}(y)\cdot \frac{y^{k/2}}{e^{2\pi y}}\cdot{\rm sgn}(z)^k\bdpsi_\infty(x)e^{\sqrt{-1}k\theta}
\end{align}
for $y,z\in\R^{\times}$ and $x,\theta\in\R$.  Here one should not confuse the representation $\pi$ in the left hand side of the equation and the real number $\pi$ in the right hand side. If $v$ is finite, then $W_{\pi}$ is the unique function in $\mathcal{W}(\pi)^{\rm new}$ such that $W_{\pi}({\bf1}_2)=1$. The explicit formula for $W_\pi(\bft(a))$ is well-known (See \cite[page 21]{Schmidt2002} or \cite[Section 2.2]{Saha16} for example).

\subsubsection{$L$-factors and $\varepsilon$-factors}\label{SS:eps}
Given $a\in\Q_v^\times$, we define an additive character $\bdpsi^a_v$ on $\Q_v$ by $\bdpsi^a_v(x)=\bdpsi_v(ax)$ for $x\in\Q_v$.   
We associate to a character $\varrho:\Q_v^{\times}\rightarrow\C^{\times}$ the $L$-factor $L(s,\varrho)$ and the $\varepsilon$-factor $\varepsilon(s,\varrho,\bdpsi^a_v)$ (cf. \cite[Section 1.1]{Schmidt2002}). 
%Tate's local zeta integral is defined by \[Z(\varphi,\chi)=\int_{F^\times}\varphi(a)\chi(a)\, \rmd^\times a\] for $\varphi\in\cS(F)$. 
The gamma factor 
\[\gamma(s,\varrho,\bdpsi_v^a)=\varepsilon(s,\varrho,\bdpsi_v^a)\frac{L(1-s,\varrho^{-1})}{L(s,\varrho)}\]
is obtained as the proportionality constant of the functional equation
\beq
\gamma(s,\varrho,\bdpsi_v^a)\int_{\Q_v^\times}\varphi(a)\varrho(a)|a|_v^s\,\rmd^\times a=\int_{\Q_v^\times}\widehat{\varphi}(a)\varrho(a)^{-1}|a|_v^{1-s}\,\rmd^\times a \label{tag:fq}
%Z(\widehat{\varphi},\chi^{-1}\Abs_F^{1-s})=\gamma(s,\chi,\bdpsi)Z(\varphi,\chi\Abs_F^s), \label{tag:fq}
\eeq
for $\varphi\in\cS(\Q_v)$, where 
\[\widehat{\varphi}(y)=\int_{\Q_v}\varphi(x_v)\bdpsi_v(yx_v)\,\rmd x_v\] 
is the Fourier transform with respect to $\bdpsi_v$. 
When $a=1$, we write 
\begin{align*}
\varepsilon(s,\varrho)&=\varepsilon(s,\varrho,\bdpsi_v), & 
\gamma(s,\varrho)&=\gamma(s,\varrho,\bdpsi_v). 
\end{align*}
When $v=\ell$ is a finite prime, we denote the exponent of the conductor of $\varrho$ by $c(\varrho)$. 
Recall that 
\beq
\varepsilon(s,\varrho,\bdpsi^a_v)=\varrho(a)|a|_v^{-1}\varepsilon(0,\varrho)\ell^{-c(\varrho)s}. \label{tag:conductor}
\eeq

Let $\pi$ be an irreducible admissible representation of $\GL_2(\Q_v)$ with central character $\om$. 
Denote by $L(s,\pi)$ and $\varepsilon(s,\pi)=\varepsilon(s,\pi,\bdpsi_v)$ its $L$-factor and $\vep$-factor relative to $\bdpsi_v$ defined in \cite[Theorem 2.18]{JLbook}. 
We write $\pi^\vee$ for the contragredient representation of $\pi$. 
The gamma factor 
\[\gamma(s,\pi)=\varepsilon(s,\pi)\frac{L(1-s,\pi^\vee)}{L(s,\pi)}\] 
is obtained as the proportionality constant of the functional equation 
\[\gamma\left(s+\frac{1}{2},\pi\right)\int_{\Q_v^\times}W(\bft(a)g)\abs{a}_v^s\rmd^\times a=\int_{\Q_v^\times}W(\bft(a)J_1^{-1}g)\om(a)^{-1}\abs{a}_v^{-s}\rmd^\times a \]
for every $W\in \cW(\pi)$. 
%If $v$ is a finite prime, then we let $c(\pi)$ be the exponent of the conductor of $\pi$. 

\subsection{$p$-stabilized newforms}\label{SS:Whittaker}

Let $\pi$ be an irreducible cuspidal automorphic representation of $\GL_2(\A)$. 
% and let $\mathcal{A}(\pi)$ be the $\pi$-isotypic part in the space of automorphic forms on $\GL_2(\A)$. 
The Whittaker function of $\varphi\in\pi$ with respect to the additive character $\bdpsi$ is given by 
\begin{align*}
W_\varphi(g)=\int_{\A/\Q}\varphi(\bfn(x)g)\bdpsi(-x)\,\rmd x
\end{align*}
for $g\in\GL_2(\A)$, where $\rmd x$ is the Haar measure with ${\rm vol}(\A/\Q,\rmd x)=1$. We have the Fourier expansion:
\[\varphi(g)=\sum_{\beta \in\Q^{\times}}W_\varphi(\bft(\beta)g)\]
(cf. \cite[Theorem 3.5.5]{Bump1998}). 
Let $f=\sum_{n}{\bf a}(n,f)q^n\in\mathcal{S}_k(N,\chi)$ be a normalized Hecke eigenform whose adelic lift
$\mathit{\Phi}(f)$ generates $\pi=\otimes_v'\pi_v$ of $\GL_2(\A)$, having central character $\chi^{-1}$. 
If $f$ is a newform, then the conductor of $\pi$ is $N$, the adelic lift $\mathit{\Phi}(f)$ is the normalized new vector in $\pi$ and the Mellin transform 
\[\int_{\A^{\times}/\Q^{\times}}\mathit{\Phi}(f)(\bft(y))|{y}|_{\A}^s\,\rmd^{\times} y=L\biggl(s+\onehalf,\pi\biggl)\]
is the automorphic $L$-function of $\pi$. 
Here $\abs{y}_\A=\prod_v\abs{y^{}_v}_v$ and $\rmd^{\times} y$ is the product measure $\prod_v\rmd^{\times} y_v$.

\begin{defn}[$p$-stabilized newform]{\rm Let $p$ be a prime and fix an isomorphism $\iota_p:\C\simeq \overline{\Q}_p$. We say that a normalized Hecke eigenform $f=\sum_{n=1}^\infty\bfa(n,f)q^n\in \mathcal{S}_k(Np,\chi)$ is an ordinary \emph{$p$-stabilized newform} with respect to $\iota_p$ if $f$ is new outside $p$ and the eigenvalue of ${\bf U}_p$, {\rm i.e.} the $p$-th Fourier coefficient $\iota_p({\bf a}(p,f))$, is a $p$-adic unit. The prime-to-$p$ part of the conductor of $f$ is called \emph{the tame conductor} of $f$.}\end{defn}

The Whittaker function of $\mathit{\Phi}(f)$ is a product of local Whittaker functions in $\mathcal{W}(\pi_v,\bdpsi_v)$ by the multiplicity one for new and ordinary vectors. 
To be precise, we have 
\[W_{\mathit{\Phi}(f)}(g)=W_{\pi_p}^{\rm ord}(g_p)\prod_{v\neq p} W_{\pi_v}(g_v)\]
for $g=(g_v)\in \GL_2(\A)$. 
Here $W_{\pi_v}$ is the normalized Whittaker newform of $\pi_v$ and $W_{\pi_p}^\Ord$ is the ordinary Whittaker function characterized by  
\[W_{\pi_p}^\Ord(\bft(a))=\varrho_f(a)\abs{a}^\onehalf\cdot \bbI_{\Zp}(a)\text{ for }a\in\Qp^\x, \]
where $\varrho_f:\Qp^\x\to\C^\x$ is the unramified character with $\varrho_f(p)=\bfa(p,f)\cdot p^{({1-k})/{2}}$ (See \cite[Corollary 2.3, Remark 2.5]{Hsieh2017}).

%Comparing the Fourier expansions of $\mathit{\Phi}(f)$ and $f$ via (\ref{E:MA1}), we find that 
%\begin{align}W_{\pi_{f,\ell}}\left(\pDII{\ell}{1}\right)={\bf a}(\ell,f)\ell^{-{k}/{2}} \text{ if }\ell\neq p;\quad W_{\pi_{f,p}}^{\rm ord}\left(\pDII{p}{1}\right)={\bf a}(p,f)p^{-{k}/{2}}.\end{align}

\section{The construction of Hilbert-Eisenstein series}\label{S:4}

\subsection{Eisenstein series}
We recall the construction of Eisenstein series described in \cite[\S 19]{JLbook2}. 
Let $F$ be a real quadratic field with integer ring $\cO_F$. 
We denote the set of real places of $F$ by $\Sigma_\R=\{\sigma_1,\sigma_2\}$, the different of $F$ by $\frakd$, the discriminant of $F$ by $\Delta_F$ and the unique non-trivial automorphism of $F$ by $x\mapsto\bar x$. 
For each finite prime $\pmq$ of $F$ we write $\cO_\pmq$ for the integer ring of $F_\pmq$. 

Let $(\Mu,\Nu)$ be a pair of unitary Hecke characters of $\A_F^\x$. 
For each place $v$ we write $\cB(\Mu_v,\Nu_v,s)$ for the space of smooth functions $f_v:\GL_2(F_v)\to\C$ which satisfy 
\[f_v\biggl(\pMX{a}{b}{0}{d}g\biggl)=\Mu_v(a)\Nu_v(d)\abs{\frac{a}{d}}_{F_v}^{s+\onehalf}f_v(g)\]
for $a,d\in F_v^\x$ and $b\in F_v$. 
%Let $\cB(\Mu,\Nu,s)$ denote the space consisting of right $\bfK$-finite functions $f\colon \GL_2(\A_F)\to\C$ such that \[f\biggl(\pMX{a}{b}{0}{d}g\biggl)=\Mu(a)\Nu(d)\abs{\frac{a}{d}}_\A^{s+\onehalf}f(g).\] We form the Eisenstein seres associated to $f_s=\otimes_vf_{v,s}\in\cB(\Mu,\Nu,s)$ by \[E_\A(g,f_s):=\sum_{\gamma\in B(F)\bksl \GL_2(F)}f_s(\gamma g).\]The above series converges absolutely for ${\rm Re}(s)\gg 0$ and has meromorphic continuation to $s \in \C$. Define the Whittaker function of $f_s$ by \begin{align*} W(g,f_s) =&\prod_vW(g_v,f_{v,s}) \intertext{for $g=(g_v)\in\GL_2(\A_F)$, where} W(g_v,f_{v,s})= &\int^{\rm st}_{F_v}f_{v,s}(J_1\bfn(x_v)g_v)\bdpsi_{F_v}(-x_v)\,\rmd x_v, & J_1&=\pMX{0}{-1}{1}{0}. \end{align*}
Recall that $\cS(F_v^2)$ denotes the space of Schwartz functions on $F_v^2$. 
We associate to $\Phi_v\in\cS(F_v^2)$ the Godement section $f_{\Mu_v,\Nu_v,\Phi_v,s}\in\cB(\Mu_v,\Nu_v,s)$ by  
\begin{multline*}
f_{\Mu_v,\Nu_v,\Phi_v,s}(g_v)\\
=\Mu_v(\det g_v)\abs{\det g_v}_{F_v}^{s+\onehalf}\int_{F_v^\x}\Phi_v((0,t_v)g_v)(\Mu_v\Nu_v^{-1})(t_v)\abs{t_v}_{F_v}^{2s+1}\,\rmd^\x t_v. 
\end{multline*}
Let $\Phi=\ot_v\Phi_v\in \cS(\A_F^2)$. 
Define a function $f_{\Mu,\Nu,\Phi,s}:\GL_2(\A_F)\to\C$ by $f_{\Mu,\Nu,\Phi,s}(g)=\prod_vf_{\Mu_v,\Nu_v,\Phi_v,s}(g_v)$. 
The series
\[E_\A(g,f_{\Mu,\Nu,\Phi,s})=\sum_{\gamma\in B(F)\bksl \GL_2(F)}f_{\Mu,\Nu,\Phi,s}(\gamma g)\]
converges absolutely for ${\rm Re}(s)\gg 0$ and has meromorphic continuation to $s \in \C$. 
It admits the Fourier expansion
\beq\label{E:FE.1}
E_\A(g,f_{\Mu,\Nu,\Phi,s})
=f_{\Mu,\Nu,\Phi,s}(g)+f_{\Nu,\Mu,\wh\Phi,-s}(g)+\sum_{\beta\in F^\x}W(\bft(\beta)g,f_{\Mu,\Nu,\Phi,s}),
\eeq
where $\wh\Phi:=\ot_v\wh\Phi_v$ is the symplectic Fourier transform defined by
$$\widehat{\Phi}_v(x,y) = \iint_{F_v^2}\Phi_v(z,u)\bdpsi_{F_v}(zy-ux)\,\rmd z\rmd u.$$ 
We tentatively write $f_{v,s}=f_{\Mu_v,\Nu_v,\Phi_v,s}$. 
There exists an open compact subgroup $\cU$ of $F_v$ such that for any open compact subgroup $\cU'$ containing $\cU$
\[\int_\cU f_{v,s}(J_1\bfn(x_v)g_v)\bdpsi_{F_v}(-x_v)\,\rmd x_v=\int_{\cU'} f_{v,s}(J_1\bfn(x_v)g_v)\bdpsi_{F_v}(-x_v)\,\rmd x_v, \]
where $J_1=\pMX{0}{-1}{1}{0}$. 
We define the regularized integral by 
\begin{align*}W(g_v,f_{\Mu_v,\Nu_v,\Phi_v,s})&=\int^{\rm st}_{F_v}f_{v,s}(J_1\bfn(x_v)g_v)\bdpsi_{F_v}(-x_v)\,\rmd x_v\\
:&=\int_\cU f_{v,s}(J_1\bfn(x_v)g_v)\bdpsi_{F_v}(-x_v)\,\rmd x_v. \end{align*}
Then $W(g,f_{\Mu,\Nu,\Phi,s})=\prod_vW(g_v,f_{v,s})$ for $g=(g_v)\in\GL_2(\A_F)$. 
\subsection{The Eisenstein series $E_k(\Mu,\Nu)$}

%It is easy to verify that \[\wh\phi_\Mu(x)=\Mu^{-1}(x)\bbI_{\uf^{-c(\mu)}\Ov^\x}(x)\cdot\varepsilon(1,\Mu^{-1}).\]

Let $N$ and $\LV$ be positive integers such that $N\Delta_F$ and $\LV$ are coprime. 
We assume that 
\beqcd{Spl}\text{ every prime factor of $N\LV$ splits in $\cK$.}\eeqcd
Then there are ideals $\frakN$ and $\frakc$ of $\cO_F$ such that 
\begin{align}
N\cO_\cK&=\frakN\ol{\frakN}, \quad (\frakN,\ol{\frakN})=1& 
\LV\cO_\cK&=\frakc\ol{\frakc}, \quad (\frakc,\ol{\frakc})=1. 
\label{tag:Heeg}
\end{align}
Fix a positive integer $k$. 
Assume that %$\Nu_{\sigma_i}^2=\Mu_{\sigma_i}^2=1$ and 
$\Nu_{\sigma_i}\Mu_{\sigma_i}=\sgn^k$ for $i=1,2$. 
%Let $k$ be a positive integer such that $(\Mu_v^{}\Nu_v)(-1)=(-1)^k$ for all $v\in \Sigma_\R$. 
We recall a construction of a certain classical Eisenstein series $E_k(\Mu,\Nu)$ of parallel weight $k$, level $\Gamma_1(N\LV)$ and central character $\Mu\Nu$, following \cite{JLbook2}. We impose the following hypotheses for $(\Mu,\Nu)$:
\begin{hypothesis}\label{H:31}\noindent
\begin{itemize}\item $\mu$ is unramified outside $p$,
\item the prime-to-$p$ part of the conductor of $\nu$ has a decomposition $\frakc \frakc'$ with $\ol{\frakc}\subset\frakc'$.\end{itemize}
\end{hypothesis}
\begin{defn}\label{D:BS.1} 
Let $k\geq 2$ be an integer. 
The quintuple \[\cD:=(\Mu,\Nu,k,\frakN,\frakc)\] is called an Eisenstein datum of weight $k$. 
The Fourier transform of $\phi\in\cS(F_v)$ is defined by 
\[\wh\phi(x):=\int_{F_v}\phi(y)\bdpsi_{F_v}(yx)\,\rmd y, \]
where the Haar measure $\rmd y$ is so chosen that $\wh{\wh\phi}(x)=\phi(-x)$. 
When $\pmq$ is a finite prime, we associate to a character $\chi:F_\pmq^\x\to\C$ a function $\phi_\chi\in\cS(F_\pmq)$ by $\phi_\chi(x)=\bbI_{\Ov^\x}(x)\chi(x)$. 
We associate to $\cD$ the \BS function \[\Phi_\cD=\bigot_{v}\Phi_{\cD,v}\in\cS(\A_F^2)\] defined as follows:
\begin{itemize}
\item $\Phi_{\cD,v}(x,y)=2^{-k}(x+\sqrt{-1}y)^ke^{-\pi(x^2+y^2)}$ if $v\in\Sigma_\R$,
\item $\Phi_{\cD,v}(x,y)=\phi_{\Mu_v^{-1}}(x)\wh\phi_{\Nu_v^{-1}}(y)$ if $v\divides p$,
\item $\Phi_{\cD,v}(x,y)=\bbI_{\frakN \frakc\cO_v}(x)\phi_{\Nu_v}(y)$ if $v\divides \frakN\frakc$,
\item $\Phi_{\cD,v}(x,y)=\bbI_{\cO_v}(x)\wh\phi_{\Nu_v^{-1}}(y)$ if $v\divides  \ol{\frakc}$,
\item $\Phi_{\cD,v}(x,y)=\bbI_{\frakd^{-1}\cO_v}(x)\bbI_{\frakd^{-1}\cO_v}(y)\cdot \abs{\Delta_F}_v^\onehalf$ if $v\ndivides p \frakN c$. 
\end{itemize}
These particular choices of \BS functions are inspired by \cite[Definition 4.1]{HsiehChen19} used in the construction of primitive $p$-adic Rankin-Selberg $L$-functions. We define the associated Godement section by $f_{\cD,s}=f_{\Mu,\Nu,\Phi_\cD,s}$ and $f_{\cD,s,v}=f_{\Mu_v,\Nu_v,\Phi_{\cD,v},s}$. 
\end{defn}

%\begin{lm}We have 
%\begin{itemize}\item \[f_{\cD,s,\infty}^{[k]}(g\pMX{\cos\theta}{\sin\theta}{-\sin\theta}{\cos\theta})=f_{\cD,s,\infty}(g) e^{\sqrt{-1} k\theta}\]
%\item \[f_{\cD,s,\pmq}(g\pMX{a}{b}{c}{d})=f_{\cD,s,\pmq}(g)\Mu_{1,\pmq}\Mu_{2,\pmq}(d).\]
%\end{itemize}
%\end{lm}
\begin{Remark}\label{R:section} 

If $v\in\Sigma_\R$, then $f_{\cD,s,v}$ is the unique function in $\cB(\Mu_v,\Nu_v,s)$ such that \[f_{\cD,s,v}(\kappa_\theta) = e^{\sqrt{-1}k\theta}\cdot 2^{-k}(\sqrt{-1})^k\pi^{-(s+\frac{k+1}{2})}\Gamma \left (s+ \frac{k+1}{2} \right)\]
(see the proof of Lemma \ref{L:FEinfty}). 
If $v=\pmq$ is a finite place, then for any integer $M$, let $\cU_1(M)$ be the open-compact subgroup of $\GL_2(\Ov)$ given by \[\mathcal{U}_1(M)=\GL_2(\Ov)\cap \pMX{\Ov}{\Ov}{M\Ov}{1+M\Ov},\]
and
$f_{\cD,s,\pmq}\in\cB(\Mu_\pmq,\Nu_\pmq,s)$ is invariant by $\cU_1(p^rN\LV)$ under the right translation for some sufficiently large $r$.
\end{Remark}

\begin{defn}\label{D:Eisenstein}
Define the classical Eisenstein series $E^\pm_k(\Mu,\Nu):\frakH^{\Sg_\R}\to\C$ by 
\begin{align*}
E^\pm_k(\Mu,\Nu)(x+y\sqrt{-1}):=&y^{-\frac{k}{2}}E_\A\biggl(\pMX{y}{x}{0}{1},\,f_{\cD,s}\biggl)\biggl|_{s=\pm\frac{k-1}{2}} &(x&\in \R^2,\,y\in \R_+^2).\end{align*}
 Then $E^\pm_k(\Mu,\Nu)$ is a Hilbert modular form of parallel weight $k$, level $p^rN\LV$ and character $\Mu^{-1}\Nu^{-1}$. 
 By definition 
\[\varPhi(E^\pm_k(\Mu,\Nu)|\frakH)(g)=E_\A((g,g),f_{\cD,s})|_{s=\pm\frac{k-1}{2}}\]
for $g\in\GL_2(\A)$, where $\varPhi$ is the adelic lift defined in \eqref{E:MA1}.
\end{defn}

\begin{prop}\label{P:shift.1}For every non-negative integer $t$, we have 
\[\varPhi(\delta_k^t E^\pm_k(\Mu,\Nu))=E_\A(f_{\cD_t,s})|_{s=\pm\frac{k-1}{2}},\]
where $\cD_t=(\Mu,\Nu,k+2t,\frakN,\frakc)$ is an Eisenstein datum of weight $k+2t$.
\end{prop}
\begin{proof}
Recall the differential operator $V_+$ defined in \S \ref{SS:auto}. 
Proposition \ref{P:shift.1} follows from \eqref{E:diffop1} in view of the relation $V_+^tf_{\cD,s,\infty}=f_{\cD_t,s,\infty}$ (see \cite[Lemma 5.6 (iii)]{JLbook}).
\end{proof}

\subsection{Fourier coefficients of Eisenstein series} 

\begin{lm}\label{L:FEinfty}
For $a\in\R^\x$, we have
 \begin{align*}
W(\bft(a),f_{\cD,s,\infty})|_{s=\frac{k-1}{2}}=&W(\bft(a),f_{\cD,s,\infty})|_{s=\frac{1-k}{2}}= a^{\frac{k}{2}}e^{-2\pi a}\cdot \bbI_{\R_+}(a).
%=&(-1)\Mu(a)\cdot a^ke^{-2\pi a}\cdot\bbI_{\R_+}(a).
\end{align*}
\end{lm}

\begin{proof}
By definition, $W(\bft(a),f_{\cD,s,\infty})$ equals
\begin{align*}
&2^{-k}\Mu\Abs^{s+\onehalf}(a)
\int_\R\int_{\R^\x}t^k(a+\sqrt{-1}x)^ke^{-\pi t^2(x^2+a^2)}\sgn(t)^k\abs{t}^{2s+1}\bdpsi_\infty(-x)\,\rmd^\x t\rmd x\\
=&\Mu\Abs^{s+\onehalf}(a)\cdot (-2\sqrt{-1})^{-k}\cdot  \Gamma\biggl(s+\frac{k+1}{2}\biggl)\pi^{-(s+\frac{k+1}{2})} \\
&\times \int_{\R}(x+\sqrt{-1}a)^{-(s+\frac{k+1}{2})}(x-\sqrt{-1}a)^{-(s-\frac{k-1}{2})}\bdpsi_\infty(-x)\,\rmd x.
\end{align*}
By Cauchy's integral formula we find that 
\begin{align*}
W(\bft(a),f_{\cD,s,\infty})|_{s=\frac{k-1}{2}}
=&\Mu\Abs^{\frac{k}{2}}(a)\cdot (-2\pi\sqrt{-1})^{-k}\cdot  \Gamma(k)
 \int_{\R}\frac{e^{-2\pi\sqrt{-1}x}}{(x+\sqrt{-1}a)^{k}}\,\rmd x\\
 =&\Mu(a)\cdot a^{\frac{k}{2}}e^{-2\pi a}\cdot \bbI_{\R_+}(a),
\end{align*}
and that
\begin{align*}
W(\bft(a),f_{\cD,s,\infty})|_{s=\frac{1-k}{2}}
=&\Mu\Abs^{1-\frac{k}{2}}(a)(-2\sqrt{-1})^{-k}\pi^{-1}\int_\R \frac{(x-\sqrt{-1}a)^{k-1}e^{-2\pi\sqrt{-1}x}}{x+\sqrt{-1}a}\,\rmd x\\
=&\Mu(a)\cdot a^\frac{k}{2}e^{-2\pi a}\cdot\bbI_{\R_+}(a).
\end{align*}
Since $\mu$ is a quadratic character, the lemma follows.
\end{proof}

Let $q_\pmq=|\ufv|^{-1}=\sharp(\cO_F/\pmq)$ denote the cardinality of the residue field.  

\begin{lm}\label{L:FEfinite} 
Let $v=\pmq$ be a prime ideal of $\cO_F$.  
Let $a\in \Fv^\x$. 
Put 
\begin{align*}
\chi_\pmq&=\Mu_{\pmq}^{-1}\Nu_{\pmq}, & 
\gamma_\pmq&=\chi_\pmq(\ufv), &
q_\pmq&=|\ufv|^{-1}=\sharp(\cO_F/\pmq), & 
m&=\Ord_\pmq(a). 
\end{align*}
Then $W(\bft(a),f_{\cD,s,\pmq})$ equals 
\begin{gather*}
\tag{$\pmq\ndivides p\frakN c$} 
\Mu_\pmq(a)|a|^{s+\onehalf}\sum_{j=0}^{m+\Ord_\pmq(\frakd)}(\gamma_\pmq q_\pmq^{2s})^j, \\
\tag{$\pmq \divides \frakN$} 
\Mu_\pmq(a)|a|^{s+\onehalf}\biggl(\sum_{j=0}^{m-\Ord_\pmq(\frakN)}(\gamma_\pmq q_\pmq^{2s})^j-q_\pmq^{-1}\sum_{j=-1}^{m-\Ord_\pmq(\frakN)}(\gamma_\pmq q_\pmq^{2s})^j\biggl), \\
%=&\Mu_\pmq(a)|a|^{s+\onehalf}\cdot \bbI_{\frakN\frakq^{-1}\cO_\pmq}(a) \biggl(-\gamma_\pmq^{-1} q_\pmq^{-2s-1}+(1-q_\pmq^{-1})\sum_{j=0}^{m-\Ord_\pmq(\frakN)}(\gamma_\pmq q_\pmq^{2s})^j\biggl).
\tag{$\pmq\divides \frakc$} 
\Mu_\pmq(a)|a|^{s+\onehalf}\cdot\varepsilon(-2s,\chi_\pmq)^{-1}\cdot\bbI_{\Ov}(a), \\
\tag{$\pmq\divides \ol{\frakc}$} 
\Mu_\pmq(a)|a|^{s+\onehalf}\bbI_{\cO_\pmq}(a), \\
\tag{$\frakq=\frakp\divides p$}  
\bbI_{\cO_\frakp^\x}(a). 
\end{gather*}
\end{lm}

\begin{proof}
Fix a local uniformizer $\ufv\in\cO_\pmq$ of the prime ideal $\pmq$. 
%We write $\mu=\mu_\pmq$ and $\nu=\nu_\pmq$ for simplicity. 
Note that if $\Phi=\Phi_1\ot\Phi_2\in\cS(\Fv^2)$, then 
\beq
f_{\cD,s,\pmq}\biggl(\pMX{0}{-1}{1}{0}\pMX{a}{x}{0}{1}\biggl)=\Mu_\pmq(a)|a|^{s+\onehalf}\int_{\Fv^\x}\Phi_1(ta)\Phi_2(tx)(\mu^{}_\pmq\nu_\pmq^{-1})(t)|t|^{2s+1}\,\rmd^\x t \label{tag:Godement}
\eeq
and hence 
\[W(\bft(a),f_{\cD,s,\pmq})=\Mu_\pmq(a)|a|^{s+\onehalf}\int_{\Fv^\x}\Phi_1(ta)\wh\Phi_2(-t^{-1})(\mu^{}_\pmq\nu_\pmq^{-1})(t)|t|^{2s}\,\rmd^\x t.\]
If $\pmq\ndivides p\frakN\frakc$, then $\Phi_{\cD,\pmq}=\bbI_{\frakd^{-1}\Ov}\ot\bbI_{\frakd^{-1}\Ov}$, and hence
\begin{align*}
W(\bft(a),f_{\cD,s,\pmq}) =&\Mu_\pmq(a)|a|^{s+\onehalf}\int_{\Fv^\x}\bbI_{\frakd^{-1}\Ov}(t^{-1}a)\bbI_{\Ov}(-t)\chi_\pmq(t)|t|^{-2s}\,\rmd^\x t\\
=&\Mu_\pmq(a)|a|^{s+\onehalf}\sum_{j=0}^{m+\Ord_\frakq(\frakd)}\chi_\pmq(\ufv^j) q_\pmq^{2sj}.
\end{align*}
If $\pmq\divides \frakN \frakc$, then $\Mu_\pmq$ is unramified by assumption. 
It is easy to verify that 
\[\wh\phi_{\Nu_\pmq}(x)=\begin{cases}\bbI_{\Ov}(x)-q_\pmq^{-1}\bbI_{\pmq^{-1}\Ov}(x)&\text{ if }\pmq\divides \frakN,\\[1em]
\varepsilon(1,\Nu_\pmq^{-1})\Nu_\pmq^{}(x^{-1})\bbI_{\ufv^{-c(\Nu_\pmq)}\Ov^\x}(x)&\text{ if }\frakq\divides \frakc. 
\end{cases}\]
One can readily prove the case $\pmq\divides\frakN$.  
%\begin{align*} \Mu_\pmq(a)^{-1}|a|^{-s-\onehalf}W(\bft(a),f_{\cD,s,\pmq})=&\sum_{j=0}^{m-\Ord_\pmq(\frakN)}\chi_\pmq(\ufv^j)q_\pmq^{2sj}-\abs{\ufv}\sum_{j=-1}^{m-\Ord_\pmq(\frakN)}\chi_\pmq(\ufv^j)q_\pmq^{2sj}. 
%\\=& \bbI_{\frakN\pmq^{-1}\Ov}(a)\biggl(-\chi\Abs^{2s+1}(\varpi_\pmq)+(1-\abs{\ufv})\sum_{j=0}^{m-\Ord_\pmq(\frakN)}\chi\Abs^{-2s}(\ufv^j)\biggl).\end{align*}
If $\frakc$ is divisible by $\pmq$, then  
\begin{align*}
\Mu_\pmq(a)^{-1}|a|^{-s-\onehalf}W(\bft(a),f_{\cD,s,\pmq})
&=\int_{F_\pmq^\x}\bbI_{\LV\Ov}(at)\wh \phi_{\Nu_\pmq}(-t^{-1})(\mu^{}_\pmq\nu_\pmq^{-1})(t)|t|^{2s}\,\rmd^\x t \\
%=&\varepsilon(1,\Nu^{-1})\Nu(-1)\int_{F_\pmq^\x}\bbI_{\Ov}(at)\bbI_{\ufv^{c(\Nu_\pmq)}\Ov^\times}(t)\Mu\Abs^{2s}(t)\,\rmd^\x t
&=\varepsilon(1,\Nu_\pmq^{-1})\Nu_\pmq(-1)\Mu_\pmq\bigl(\ufv^{c(\Nu_\pmq)}\bigl)q_\pmq^{-2sc(\Nu_\pmq)}\bbI_{\LV\ufv^{-c(\Nu_\pmq)}\Ov}(a). 
\end{align*}
Note that $C\Ov=\ufv^{-c(\Nu_\pmq)}\Ov$ for $\pmq\divides \frakc$ by our assumption on the conductor of $\nu$ and that \[\varepsilon(1,\Nu_\pmq^{-1})\Nu_\pmq^{}(-1)\Mu_\pmq\bigl(\ufv^{c(\Nu_\pmq)}\bigl)q_\pmq^{-2sc(\Nu_\pmq)}=\Nu_\pmq^{}(-1)\varepsilon(1+2s,\chi_\pmq^{-1})=\varepsilon(-2s,\chi_\pmq)^{-1}\] by (\ref{tag:conductor}). 
If $\pmq\divides \ol{\frakc}$, then $W(\bft(a),f_{\cD,s,\pmq})$ equals
\[\Mu_\pmq(a)|a|^{s+\onehalf}\int_{F_\pmq^\x}\bbI_{\Ov}(at)\phi_{\Nu_\pmq^{-1}}(t^{-1})(\Mu_\pmq^{}\Nu_\pmq^{-1})(t)|t|^{2s}\,\rmd^\x t=\Mu_\pmq(a)|a|^{s+\onehalf}\bbI_{\cO_\pmq}(a).\]

Finally, if $v=\frakp|p$, then we find that $W(\bft(a),f_{\cD,s,\frakp})$ equals
\begin{align*}
&\Mu_\frakp(a)|a|^{s+\onehalf}\int_{F_\frakp^\x}\phi_{\Mu_\frakp^{-1}}(at)\phi_{\Nu_\frakp^{-1}}(t^{-1})(\Mu_\frakp^{}\Nu_\frakp^{-1})(t)|t|^{2s}\,\rmd^\x t=\bbI_{\cO_\frakp^\x}(a)
\end{align*}
by a similar calculation. 
 %The assertion follows immediately from the above expressions of $W(f_{\cD,s,\ell},\pDII{a}{1})$. 
 \end{proof}
 
For each non-zero element $\beta\in F^\times$ we define the polynomials $\cP_{\beta,\pmq}$ and $\cQ_{\chi,\pmq}$ in $\Z_{(q_\pmq)}[X,X^{-1}]$ by
\beq\label{E:poly.3}\begin{aligned}
\cP_{\beta,\pmq}(X)&=\begin{cases}
\sum\limits_{j=0}^{\Ord_\pmq(\beta\frakd)}\qv^{-j}X^j&\text{ if $\pmq\ndivides p\frakN\frakc$, } \\
\sum\limits_{j=0}^{\Ord_\pmq(\beta\frakN^{-1})}\qv^{-j}X^j-\sum\limits_{j=-1}^{\Ord_\pmq(\beta\frakN^{-1})}\qv^{-(j+1)}X^j&\text{ if $\pmq\divides \frakN$, }
\end{cases}\\
\cQ_{\chi,\pmq}(X)&=\varepsilon(0,\chi_\pmq)^{-1}\cdot (q_\pmq X^{-1})^{c(\chi_\pmq)}. 
\end{aligned}
\eeq
Let $\beta\in F$. 
We write $\beta>0$ if $\sigma_i(\beta)>0$ for $i=1,2$. 

\begin{cor}\label{C:FE}
We have the following Fourier expansion around the infinity cusp: 
\[E^\pm_k(\Mu,\Nu)(\tau_1,\tau_2)=\sum_{0<\beta\in\frakd^{-1},\,(p,\beta)=1}\sigma^\pm_\beta(\Mu,\Nu,k)\cdot e^{2\pi \sqrt{-1}(\tau_1\sigma_1(\beta)+\tau_2\sigma_2(\beta))},\]
where\begin{align*}
\sigma_\beta^+(\Mu,\Nu,k)=&\Mu_p^{-1}(\beta)\prod_{\pmq\ndivides \frakc p}\cP_{\beta,\pmq}(\gamma_\pmq\cdot\qv^k)\prod_{\pmq\divides (\frakc, \beta)}\cQ_{\Mu^{-1}\Nu,\pmq}(\qv^k), \\
\sigma_\beta^-(\Mu,\Nu,k)=&\rmN_{F/\Q}(\beta)^{k-1}\cdot \Mu_p^{-1}(\beta)\prod_{\pmq\ndivides \frakc p}\cP_{\beta,\pmq}(\gamma_\pmq\cdot\qv^{2-k})
\prod_{\pmq\divides (\frakc,\beta)}\cQ_{\Mu^{-1}\Nu,\pmq}(\qv^{2-k}).
\end{align*}
\end{cor}

\begin{proof}
Note that if $\Phi=\phi_1\ot\phi_2\in\cS(F_v^2)$, then $\wh\Phi(x,y)=\wh\phi_2(-x)\wh\phi_1(y)$. 
Since $\Phi_{\cD,\p}(0,y)=0$ and $\wh\Phi_{\cD,\p}(0,y)=\phi_{\Nu_\p^{-1}}(0)\wh\phi_{\Mu_\p^{-1}}(y)=0$ for a prime $\p$ lying above the distinguished prime $p$, we see that 
\beq
f_{\cD,s,\p}(g)=f_{\Nu_{\p},\Mu_{\p},\wh\Phi_{\cD,\p},-s}(g)=0\text{ for }g\in B(F_\frakp). \label{E:constant.1}
\eeq  
This in particular implies that 
\[f_{\cD,s}\biggl(\pMX{y}{x}{0}{1}\biggl)=f_{\Nu_{\p},\Mu_{\p},\wh\Phi_{\cD,\p},-s}\biggl(\pMX{y}{x}{0}{1}\biggl)=0.\] 
In view of \eqref{E:FE.1} and \lmref{L:FEinfty}, we find that 
\[\sigma_\beta^\pm(\Mu,\Nu,k)=\rmN(\beta)^\frac{k}{2}\prod_{\pmq<\infty} W(\bft(\beta),f_{\cD,s,\pmq})|_{s=\pm\frac{k-1}{2}}.\]
The assertion follows from \lmref{L:FEfinite} by noting that $\Mu_\pmq^{-1}\Nu_\pmq(\ufv)=\Mu\Nu^{-1}(\pmq)$ if $\pmq$ is the prime induced by $v$.
\end{proof}

\section{Restriction of Eisenstein series}
In this section, we study a certain global zeta integral of $Z_\cD(s,\varphi)$ introduced in \subsecref{SS:42}. This zeta integral naturally appears in the spectral decomposition of the restriction of the Eisenstein series $E_k^\pm(\mu,\nu)$, and the main result (\thmref{T:formula}), which will be used in the explicit interpolation formula of our twisted triple product $p$-adic $L$-functions, shows that this integral is essentially a product of the toric period integral in \eqref{E:toric} and an automorphic $L$-function for $\GL_2$. 
\subsection{Optimal embeddings}\label{SS:optimal}
Let $F$ be a real quadratic field whose discriminant is denoted by $\Delta_F$. 
Define $\CMP\in F$ by $\CMP=\frac{D'-\sqrt{\Delta_F}}{2}$, where $D'=\Delta_F$ or $\frac{\Delta_F}{2}$ according to whether $\Delta_F$ is odd or even. 
Then $\cO_\cK=\Z+\Z\CMP$, and if $q$ is ramified in $F$, then $\CMP$ is a local uniformizer of $\cO_q$. Denote by $x\mapsto \ol{x}$ the unique non-trivial automorphism of $\Gal(F/\Q)$. Put
\[\delta:=\ol{\CMP}-\CMP=\sqrt{\Delta_F}.\]
We choose an embedding $\sg_1:F\hookto \R$ such that $\sigma_1(\delta)>0$. 
Define an algebraic group $T$ over $\Q$ by $T(R)=(F\ot R)^\x$ for any commutative field $R$ of characteristic zero. 
We view $T$ as a maximal torus of $\GL_2$ via the embedding $\Psi\colon F\hookto \Mat_2(\Q)$ defined by 
\[\Psi(\CMP)=\pMX{\rmT_{F/\Q}(\theta)}{-\rmN_{F/\Q}(\theta)}{1}{0}.\]
Put 
\[\eta:=\pMX{1}{-\CMP}{-1}{\ol{\CMP}}\delta^{-1}=\pMX{\ol{\CMP}}{\CMP}{1}{1}^{-1}\in\GL_2(F). \]
%If $v=\mathfrak v\bar{\mathfrak v}$ is split in $\cK$ and $t=(t_\bfv,t_{\bar{\mathfrak v}})\in \cK_v:=\cK\ot_\Q\Q_v=\cK_{\mathfrak v}\oplus\cK_{\bar{\mathfrak v}}$, 
It is important to note that for $t\in F$
\beq\label{E:cm2.W}\eta\Psi(t)\eta^{-1}=\pDII{\bar t}{t}. \eeq
%\beq\label{E:cm2.W}\cmpt_v^{-1}\Psi(t)\cmpt_v^{}=\pDII{t_{\bar{\mathfrak v}}}{t_\bfv}\in\Mat_2(\cK_{\mathfrak v}). \eeq

Let $N$ and $\LV$ be positive integers such that 
\begin{itemize}
\item $\LV$ and $N\Delta_\cK$ are coprime;
\item Every prime factor of $N\LV$ is split in $F$.
\end{itemize}
Fix decompositions $N\cO_\cK=\frakN\ol{\frakN}$ and $\LV\cO_\cK=\frakc\ol{\frakc}$ once and for all. 
Fix a prime ideal $\p$ of $\cO_F$ lying above $p$. 

\begin{defn}\label{D:cmpt}We define special elements $\cmpt,\,\cmpt^{(\LV)}$ and $\cmpt^{(\LV p^n)}$ in $\GL_2(\A)$ as follows:
\begin{itemize}
\item At the archimedean place, put
\[\cmpt_\infty=\pMX{\sg_2(\CMP)}{\sg_1(\CMP)}{1}{1}\in\GL_2(\R).\]
%\item At the $p$-adic place, \[\cmpt_p:=\begin{cases}\pMX{\ol{\CMP}}{\CMP}{1}{1}\in\GL_2(F_\p)&\text{ if $p=\p\pbar$ is split in $\cK$}\\
%1&\text{ if $p$ is inert in $\cK$.}
%\end{cases}\]
\item For each rational prime $\pme$ we fix a prime ideal $\frakq$ of $\cO_F$ above $q$ and define $\cmptv\in \GL_2(\Qq)$ by
\[\begin{aligned}
%\cmptv=&\pDII{1}{\pme^{\Ord_\pme(c)}}\text{ if $\pme\ndivides c^+$ or $q=p$ is non-split,}\\
\cmptv=&\pMX{\ol{\CMP}}{\CMP}{1}{1}\delta^{-1}\in\GL_2(\cK_\w)=\GL_2(\Qq)\text{ if $\pme=\w\wbar$ is split, }\\% with $\w\divides\p\frakN\frakc$.}\\
\cmptv=&1\text{ otherwise. }%\pme\ndivides pN\LV.
\end{aligned}\]
\item Put  
\begin{align*}\cmpt_q^{(\LV)}&=\pMX{\LV}{-1}{0}{1}\in\GL_2(\Q_q);\\
\cmpt_p^\setn&=\begin{cases}\pMX{p^n}{-1}{0}{1}\in\GL_2(F_\p)&\text{ if $p=\p\pbar$ is split in $\cK$}\\
\pMX{0}{1}{-p^n}{0}\in\GL_2(\Qp)&\text{ if $p$ is inert in $\cK$.}
\end{cases}\end{align*}
Finally, we define 
\[\cmpt=\prod_v \cmpt_v
,\quad\cmpt^{(C)}:=\cmpt\prod_{\pme \divides \LV}\cmpt_q^{(\LV)};\quad \cmpt^{(\LV p^n)}:=\cmpt^{(\LV)}\cmpt_p^\setn.\]
\end{itemize}
\end{defn}
Let $\cO_\LV=\Z+\LV\cO_\cK$ be the order of $\cK$ of conductor $\LV$. It is not difficult to verify immediately that the inclusion map $\Psi:F\hookto \Mat_2(\Q)$ is an optimal embedding of $\cO_\LV$ into the Eichler order $R_N:=\Mat_2(\Q)\cap \cmpt^{(\LV)}\Mat_2(\wh\Z)(\cmpt^{(\LV)})^{-1}$ of level $N$. In other words,
\[\Psi^{-1}(R_N)\cap F=\cO_\LV. \]
%For each non-negative integer $n$, we choose $\cmpt_\frakp^\setn\in G(\Qp)$ as follows.
%If $\frakp=\frakP\ol{\frakP}$ splits in $\cK$, we put
%\begin{align}\label{E:op1.W}\cmpt_\frakp^\setn=&\pMX{\CMP}{-1}{1}{0}\pDII{\p^n}{1}\in\GL_2(\cK_\frakP)=\GL_2(\Qp).
%\intertext{If $\frakp$ is inert or ramified in $\cK$, then we put}
%\label{E:op2.B}\cmpt_\frakp^\setn=&\pMX{0}{1}{-1}{0}\DII{\p^n}{1}.\end{align}
%Define $x_n:\AK^\x\to G(\A)$ by
%\beq\label{E:op3.W}x_n(a):=a\cdot \cmpt^\setn\quad(\cmpt^\setn:=\cmpt^\setn_\frakp\prod_{\pme\not=\frakp}\cmptv).\eeq
%This collection $\stt{x_n(a)}_{a\in\AK^\x}$ of points is called %\emph{Gross points} of conductor $\frakp^n$ associated to $\cK$.

\subsection{A result of Keaton and Pitale}\label{SS:42}

Let $\pi\simeq\otimes'_v\pi^{}_v$ be an irreducible cuspidal automorphic representation of $\GL_2(\A)$ generated by a cusp form $\mathit{\Phi}(f)\in\cA_{2k}^0(N,\omega)$. 
 Let $\Mu$ and $\Nu$ be unitary Hecke characters of $\A_F^\x$ such that $\Mu$ has $p$-power conductor and such that the restriction of $\Mu\Nu$ to $\A^\x$ is $\om$. 
Define the Hecke character $\chi:F^\x\bksl \A_F^\x\to\C^\x$ by 
\[\chi(x):=\Mu(x)\Nu(\bar x). \]
Given $\varphi\in\pi$, we define the global zeta integral by 
\[Z_\cD(s,\varphi)=\int_{\A^\x\GL_2(\Q)\bksl \GL_2(\A)}E_\A(g,f_{\cD,s})\varphi(g)\om(\det g)^{-1}\,\rmd^\tau g, \]
where $f_{\cD,s}$ is the section defined in Definition \ref{D:BS.1} associated with the datum $\cD=(\Mu,\Nu,k,\frakN,\frakc)$. 
This integral converges absolutely for all $s$ away from the poles of $E_\A(g,f_{\cD,s})$ and defines a meromorphic function in $s$. 

We define the Tamagawa measures $\rmd^\x x$ of $\A_F^\times$ and $\rmd^\x a$ of $\A^\x$ as in \S \ref{SS:measure}, and define the Tamagawa measure $\rmd t$ of $T(\A)$ as the quotient measure of $\rmd^\x x$ and $\rmd^\x a$. 
Let $\rmd g$ denote the quotient measure of $\rmd^\tau g$ and $\rmd t$. 
Given $\varphi\in\pi$, we define the toric period integral by
\beq\label{E:toric}
B_\varphi^\chi(g)=\int_{\A^\x F^\x\bksl \A^\x_F}\varphi(\Psi(t)g)\chi(t)^{-1}\,\rmd t. 
\eeq
%For each place $v$ of $\Q$, we fix a non-zero element $\ell_v\in\Hom_{\GL_2(\Q_v)}(\pi_v,\cS(\chi_v))$ and write $B_{\varphi_v}(g_v):=\ell_v(\pi_v(g_v)\varphi_v)$. We normalize $\ell_v$ so that $B_{\varphi_v}(\bfone_2)=\ell_v(\varphi_v)(\bfone_2)=1$ if $\varphi_v$ is the spherical vector $v\ndivides pN$. 
%Let $\Phi=\bigot_v\Phi_{\cD,v}\in\cS(\A_F^2)$ be the \BS function introduced in \defref{D:BS.1} with $\cD=(\Mu,\Nu,k,\frakN,\frakc)$.

\begin{thm}[Keaton and Pitale]\label{T:KeatonPitale}
%Fix an isomorphism $\pi\iso \ot_v'\pi_v$. 
Let $\varphi\in\pi$. 
Then
\[Z_\cD(s,\varphi)=\int_{T(\A)\backslash\GL_2(\A)}B_\varphi^\chi(g)\,\rmd g. \]
\end{thm}

\begin{proof}
This is nothing but Proposition 2.3 of \cite{KP19DM}. 
%But note that we use a different $\eta$ here. 
\end{proof}

\subsection{Global setting}\label{SS:4.3}

Now we let $f=\sum_{n=1}^\infty\bfa(n,f)q^n\in\cS_{2k}(Np^r,\om^{-1})$ be a $p$-stabilized newform and $\varphi=\varPhi(f)\in\cA_{2k}^0(N,\om)$ be the automorphic form associated with $f$ in \eqref{E:MA2}. 
For each prime factor $q$ of $\LV$ we choose a root $\al_q(f)$ of the Hecke polynomial $X^2-\bfa(q,f)X+\om^{-1}(q)q^{2k-1}$. Let $\breve f$ be the unique form in $\cS_{2k}(N\LV p^r,\om^{-1})[f]$ such that $\bfa(1,\breve f)=1$ and $\bfU_q \breve f=\al_q(f)\breve f$. Let $\breve \varphi=\varPhi(\breve f)$ be the adelic lift of $\breve f$. We impose the following assumptions: 
\begin{itemize}
\item $\om$ has a square root $\om^\onehalf$;
\item $\Mu$ and $\om$  are unramified outside $p$;  
\item $\LV\cO_F$ is the conductor of $\chi\om_F^{-\onehalf}$ ($\om_F^\onehalf:=\om^\onehalf\circ\rmN$). 
\end{itemize}
Note that these assumptions imply that the $\LV\cO_F$ is the prime-to-$p$ part of the conductor of $\Nu$. Define the matrices $\cJ_\infty$ and $t_n$ for each integer $n$ in $\GL_2(\A)$ by  
\beq\label{E:1.4} \cJ_\infty=\pDII{-1}{1}\in\GL_2(\R),\quad t_n=\pMX{0}{p^{-n}}{-p^n}{0}\in\GL_2(\Qp).\eeq

\subsection{Local zeta integrals}

For each place $v$ of $\Q$ we set $f_{\cD,s,v}(g_v)=\prod_{\bfv|v}f_{\cD,s,\bfv}(g_\bfv)$ for $g_v=(g_\bfv)_{\bfv|v}\in\prod_{\bfv|v}\GL_2(F_\bfv)$. 
Assume that $\varphi$ has the factorizable Whittaker function $W_\varphi(g)=\prod_vW_v(g_v)$ for $g=(g_v)\in\GL_2(\A)$. 
We associate to each Whittaker function $W_v\in\cW(\pi_v,\bdpsi_v)$ a Bessel function $B_{W_v}:\GL_2(\Q_v)\to\C$ by 
\[B_{W_v}(g_v)=\int_{\Q_v^\x \bksl F_v^\x}W_v(\cmpt_v^{-1} \Psi(t_v)g_v)\chi_v(t_v)^{-1}\,\rmd t_v\]
unless $v=p$ is inert in $F$. 
Here $\rmd t_v$ is the quotient measure of $\rmd^\x x_v$ and $\rmd^\x a_v$ (see \S \ref{SS:measure}). 
This integral is absolutely convergent (see the proof of Proposition \ref{P:spl}). 
If $v=p$ is inert in $F$, then we will explicitly choose a Whittaker function $\wtd W\in\cW(\pi_p^\vee,\bdpsi_p^{-1})$ in the proof of Proposition \ref{P:padic} so that $\rho(t)\wtd W=\chi_p(t)^{-1}\wtd W$. 
Recall the standard $\GL_2(\Qp)$-invariant pairing $\pair{\;}{\;}:\cW(\pi_p^{},\bdpsi_p^{})\times \cW(\pi^\vee_p,\bdpsi_p^{-1})\to\C$ defined by 
\[\pair{W_1}{W_2}=\int_{\Q_p^\x}W_1(\bft(a_p))W_2(\bft(a_p))\,\rmd^\times a_p. \]
Define the Bessel function $B_{W_p}:\GL_2(\Qp)\to\C$ by $B_{W_p}(g):=\pair{\rho(g)W_p}{\wtd W}$. 
 %, and the map $W_v\mapsto B_{W_v}$ defines a nonzero intertwining map from $\pi_v$ to $\cS(\chi_v)$. There exists a non-zero $\C$-linear functional $\ell_v\in\Hom_{T(\Q_v)}(\pi_v,\chi_v)$ by (\ref{tag:Heeg}).  For $v\neq p$ we normalize $\ell_v$ so that $\ell_v(\phi_{\pi_v})=1$. We define a function $B_v:\GL_2(\Q_v)\to\C$ by $B_{\varphi_v}(g_v):=\ell_v(\pi_v(g_v)\varphi_v)$. Let $\cS(\chi_v)$ be the space consisting of functions $B_v:\GL_2(\Q_v)\to\C$ which satisfy $B_v(t_vg_v)=\chi_v(t_v)B_v(g_v)$ for $t_v\in F_v^\times$ and $g_v\in\GL_2(\Q_v)$. 
The integral
\beq\label{E:locazeta}
Z_\cD(s,B_{W_v})=\int_{T(\Q_v)\bksl \GL_2(\Q_v)}f_{\cD,s,v}(\eta g_v)\om_v(\det g_v)^{-1}B_{W_v}(g_v)\,\rmd g_v
\eeq
makes sense by (\ref{E:cm2.W}), where $\rmd g_v$ is the quotient measure of $\rmd^\tau g_v$ and $\rmd t_v$.  

\subsection{Convergence}

In this and next subsections we fix a place $v$ of $\Q$ and suppress the subscript $v$ from the notation. 
Thus 
\begin{align*}
F&=F\ot\Q_v, & 
\bdpsi&=\bdpsi_v, \quad\quad 
%\Abs&=\Abs_v, & 
|\cdot|=|\cdot|_v, \quad\quad
\Mu=\mu_v, \quad\quad 
\Nu=\nu_v, \\
\pi&=\pi_v, &
\varPhi_\cD&=\otimes_{\bfv|v}\varPhi_{\cD,\bfv}\in\cS(F^2), \dots. 
\end{align*}

\begin{lm}\label{L:conv}
The integral defining $Z_\cD(s,B_W)$ is absolutely convergent for $\Re s\gg 0$. 
\end{lm}

\begin{proof}
Put $T_q=T(\Q_q)$. 
For $W\in\cW(\pi)$ we have 
%We have $f_{\cD,s,q}=f_{\Phi_\frakq}\ot f_{\Phi_{\ol{\frakq}}}$ and $\Phi_\pme=\Phi_\frakq\ot\Phi_{\ol{\frakq}}$. 
\begin{align*}
Z_\cD(s,B_W)=&\int_{T_q\bksl \GL_2(\Q_q)}f_{\cD,s,\pme}(\eta g)\om(\det g)^{-1}B_W(g)\,\rmd g\\
=&\int_{T_q\bksl \GL_2(\Q_q)}f_{\cD,s,\pme}(\eta g)\om(\det h)^{-1}\int_{\Q_q^\x \bksl T_q}W(\cmpt_q^{-1}  tg)\chi(t)^{-1}\,\rmd t\rmd h\\
=&\int_{T_q\bksl \GL_2(\Q_q)}\int_{\Q_q^\x \bksl T_q}f_{\cD,s,\pme}(\eta tg)\om(\det(tg))^{-1}W(\cmpt_q^{-1}  tg)\,\rmd t\rmd g
\end{align*}
by definition. 
We combine the iterated integral to obtain 
\[Z_\cD(s,B_W)=\int_{\PGL_2(\Q_q)} f_{\cD,s,q}(\eta g)\om(\det g)^{-1}W(\cmpt_q^{-1} g)\,\rmd^\tau g.  \]

First assume that $v=q=\frakq\bar\frakq$ is split in $F$. 
Since $\eta\varsigma_q=\delta^{-1}$ and $\bar\eta\cmpt_q=\delta^{-1}\pMX{0}{1}{1}{0}$, we get 
\begin{align}
Z_\cD(s,B_W)&=\int_{\PGL_2(\Q_q)}f_{\cD,s,\frakq}(g)\om(\det g)^{-1}f_{\cD,s,\ol{\frakq}}\biggl(\pMX{0}{1}{1}{0}g\biggl) W(g)\,\rmd^\tau g\notag \\ \label{E:formula1}
&=\int_{N(\Q_q)\bksl \PGL_2(\Q_q)}f_{\cD,s,\frakq}(g)\om(\det g)^{-1}W_{\ol{\frakq}}\biggl(\pMX{1}{0}{0}{-1}g\biggl)W(g)\,\rmd^\tau g,
%\\ &=\int_{N\bksl G}\varPhi_{\cD,\frakq}((0,1)g)(\mu_{\frakq}\om^{-1}\Abs^{s+\frac{1}{2}})(\det g)W_{\ol{\frakq}}\biggl(\pMX{1}{0}{0}{-1}g\biggl)W(g)\,\rmd g, \label{E:formula1}
\end{align}
where $W_{\ol{\frakq}}(g):=W(g,f_{\cD,s,\ol{\frakq}})$. 
This is nothing but the local Rankin-Selberg integral for $\GL_2\times\GL_2$, which is absolutely convergent for $\Re s\geq 0$. 
 
Next assume that $v=q$ remains prime in $F$. 
It suffices to show that the integral 
\beq
\int_{\Q_q^\times}\int_{\Q_q}\frac{1}{\om(a)\abs{a}}f_{\cD,s,\pme}(\eta \bfn(x)\bft(a))\bdpsi(x)\,\rmd x\cdot W(\bft(a))\rmd^\x a \label{tag:conv}
\eeq
converges absolutely in view of the Iwasawa decomposition. 
Since $\eta=\delta^{-1}\pMX{1}{-\CMP}{-1}{\ol{\CMP}}$, 
%\[f_{\cD,s,\pme}(\eta \bfn(x)\bft(a))=\Mu(\delta^{-1}a)|\delta^{-1}a|_F^{s+\onehalf}\int_{F^\x}(\Mu\Nu^{-1})(t)|t|_F^{2s+1}\Phi\biggl((0,t)\pMX{1}{-\CMP}{-1}{\ol{\CMP}}\pMX{a}{x}{0}{1}\biggl)\bdpsi(x)\,\rmd^\x t \]by definition.  
the inner integral is 
\begin{multline*}
\Mu(\delta^{-1}a)\frac{|\delta^{-1}a|_F^{s+\onehalf}}{\om(a)\abs{a}}\\
\times\int_{\Q_\pme}\int_{F^\x}(\Mu\Nu^{-1})(t)|t|_F^{2s+1}\Phi\biggl((0,t)\pMX{1}{-\CMP}{-1}{\ol{\CMP}}\pMX{a}{x}{0}{1}\biggl)\bdpsi(x)\,\rmd^\x t\rmd x. 
\end{multline*}
Put $\xi:=\Mu\Nu^{-1}\Abs_F^{2s+1}$. 
Let $\Phi=\Phi_1\otimes\Phi_2$. 
We may assume that $|\Phi_1(x)|\leq 1$ and $\Phi_2(xc)=\Phi_2(x)$ for $x\in F$ and $c\in\cO_F^\times$. 
Since the integral
\begin{align*}
\int_{\Q_\pme}\int_{\Q_\pme^\x}|\xi(t)\Phi_1(-at)\Phi_2(t(\ol{\CMP}-x))|\,\rmd^\x t\rmd x
\leq\int_{\Q_\pme}\int_{\Q_\pme^\x}|\xi(t)\Phi_2(t\ol{\CMP}-x)|\,\rmd^\x t\frac{\rmd x}{|t|}
\end{align*}
converges for $\Re s\gg 0$, the double integral (\ref{tag:conv}) is absolutely convergent for $\Re s\gg 0$. 
\end{proof}

\subsection{Local calculations}

We compute the local zeta integrals $Z_\cD(s,B_{W_v})$ occurring in the factorization of the global integral $Z_\cD(s,\rho(\cJ_\infty t_n)\breve \varphi_f)$. 
Put $\Nu_+:=\Nu|_{\Q_v^\x}$. 
Recall the normalized Whittaker newform $W_\pi\in\cW(\pi,\bdpsi)$ (see \S \ref{SS:Wnew}). 
For each prime factor $v=q$ of $\LV$, if we write $\pi=\varrho_q\boxplus \upsilon_q$ with $\varrho_q(q)=\al_q(f)q^\frac{1-2k}{2}$, then 
\[\breve W_\pi:=W_\pi-\upsilon_q(q)\abs{q}^\onehalf\pi(\bft(q^{-1}))W_\pi. \]
Then $\breve W_\pi$ is characterized uniquely by the conditions: $\breve W_\pi({\bf1}_2)=1$ and $\bfU_q \breve W_\pi=\varrho_q(q)\abs{q}^{-\onehalf}\breve W_\pi$. In the case $v=p$, we denote by $W_\pi^\Ord$ an ordinary vector of eigenvalue $\bfa(p,f)p^{1-k}$.
By our assumptions, 
\begin{itemize}
\item $\Mu$ and $\Nu_+$ are unramified outside $p$;
\item $\chi\om_F^{-\onehalf}$ is only ramified at primes dividing $\LV$. 
\end{itemize}
%We shall consider the following special vectors in $\pi$:\begin{itemize}
%\item $\varphi_\infty\in\cD_0(k)$ is the vector of weight $k$, 
%\item for any place $v$, $\varphi_v^0$ is a new vector in $\pi$;\item for $v=q\divides \LV$, if we write $\pi=\varrho_q\boxplus \upsilon_q$, then \[\varphi^\dagger_q:=\varphi_q^0-\upsilon_q(q)\abs{q}^\onehalf\pi\biggl(\pMX{q^{-1}}{0}{0}{1}\biggl)\varphi_q^0; \]\item in the case $v=p$, we denote by $\varphi_p^\Ord$ an ordinary vector of eigenvalue $\bfa(p,f)p^{1-k}$.\end{itemize}

\begin{prop}\label{P:spl}
Let $v\neq p$ be a place of $\Q$ which is split in $F$. 
We have 
\begin{itemize}
\item If $v=\infty$ is the archimedean place, then $B_{\rho(\cJ_\infty)W_\pi}(\cmpt_\infty)\neq 0$, and 
\[Z_\cD(s,B_{\rho(\cJ_\infty)W_\pi})=4(-4\sqrt{-1})^{-k}\Nu_{\sg_1}(-1)\Gamma_\C(2s+k)\cdot B_{\rho(\cJ_\infty)W_\pi}(\cmpt_\infty).\]
\item If $v=q$ and $N\LV$ are coprime, then 
\[Z_\cD(s,B_{W_{\pi}})=L\biggl(2s+\frac{1}{2},\pi\otimes\nu_+^{-1}\biggl)B_{W_{\pi}}(\cmpt_q). \]
\item If $v=\pme$ is a prime factor of $N$, then $B_{W_{\pi}}(\cmpt_\pme)\neq 0$ and
\[Z_\cD(s,B_{W_{\pi}})=\frac{\zeta_q(2)\abs{N\LV}_{\Q_q}}{\zeta_q(1)}\cdot L\biggl(2s+\frac{1}{2},\pi\ot\Nu_+^{-1}\biggl)B_{W_{\pi}}(\cmpt_\pme).\]
\item When $v=\pme$ is a prime factor of $\LV$, then $B_{\breve W_\pi}(\varsigma_q^{}\cmpt_\pme^{(C)})\neq 0$ and 
\[Z_\cD(s,B_{W_{\pi}})=\frac{\zeta_q(2)\abs{N\LV}_{\Q_q}}{\zeta_q(1)}L\biggl(2s+\onehalf,\pi\ot\Nu_+^{-1}\biggl)\frac{\varepsilon\bigl(0,\chi_{\ol{\frakq}}^{-1}\bigl)}{ \zeta_q(1)}\cdot B_{\breve W_\pi}(\varsigma^{}_q\cmpt_\pme^{(\LV)}).\]
\end{itemize}
\end{prop}

\begin{proof}
We first treat the archimedean case. Let $W=\rho(\cJ_\infty)W_{\pi_\infty}$. By definition,  $B_W(\cmpt_\infty\bfn(x))$ equals
\begin{align*}&\int_{\R^\x}W_{\pi_\infty}(\bft(a)\bfn(x)\cJ_\infty)\Mu_{\sg_2}(a)^{-1}\Nu_{\sg_1}(a)^{-1}\,\rmd^\x a\\
=&(\Mu_{\sg_2}\Nu_{\sg_1})(-1)\int_0^\infty e^{-2\pi a(1+x\sqrt{-1})}a^{k}\,\rmd^\x a\\
=&(\Mu_{\sg_2}\Nu_{\sg_1})(-1)(2\pi)^{-k}(1+x\sqrt{-1})^{-k}\Gamma(k) 
\end{align*}
by (\ref{E:Winfty.1}), where we have shifted the contour of integration. 
%Note that $f_{\Phi_\infty}=f_{\Phi^{[k]}}\ot f_{\Phi^{[k]}}$. Then 
%\[f_{\Phi_\infty}(1)=f_{\Phi^{[k]}}(1)^2=4^{1-k}(-1)^k\pi^{-(2s+k+1)}\Gamma(s+\frac{k+1}{2})^2\]
By the Iwasawa decomposition $\GL_2(\R)=B(\R)\bfK_\infty$ and Remark \ref{R:section}, the local integral $Z_\cD(s,B_W)$ equals
\begin{align*}
&\int_{\R}f_{\cD,s,\sigma_1}(\bfn(x))f_{\cD,s,\sigma_2}\biggl(\pMX{0}{1}{1}{0}\bfn(x)\biggl)B_W(\cmpt_\infty\bfn(x))\,\rmd x\\
=&\frac{\Gamma(k)}{(2\pi \sqrt{-1})^k}\cdot 4^{-k}(-1)^k\frac{\Gamma\bigl(s+\frac{k+1}{2}\bigl)^2}{\pi^{2s+k+1}}\\
&\times\int_{\R}\Mu_{\sg_2}(-1)(1+x^2)^{-(s+1/2)}\left(\frac{x-\sqrt{-1}}{\sqrt{1+x^2}}\right)^{k}\frac{\rmd x}{(x-\sqrt{-1})^k}\\
=&\frac{\Gamma(k)}{(2\pi \sqrt{-1})^k}\cdot 2(2\pi)^{-(2s+k)}\Gamma(2s+k)\cdot 4(-4)^{-k}\Mu_{\sg_2}(-1). \end{align*}

Let $v=q=\frakq\ol{\frakq}$ be a finite split prime. 
Then 
\beq
B_{W_\pi}(\cmpt_q)=L\biggl(\frac{1}{2},\pi\otimes\chi_{\bar\frakq}^{-1}\biggl)=L\biggl(\frac{1}{2},\pi\otimes\mu_{\bar\frakq}^{-1}\nu_{\frakq}^{-1}\biggl). \label{E:Wnew}
\eeq 
If $q$ and $Nc$ are coprime, then  
\[Z_\cD(s,B_{W_\pi})=L\biggl(2s+\frac{1}{2},\pi_q\otimes\nu_{\frakq}^{-1}\nu_{\bar\frakq}^{-1}\biggl)L\biggl(\frac{1}{2},\pi_q\otimes\mu_{\bar\frakq}^{-1}\nu_{\frakq}^{-1}\biggl) \]
by (\ref{E:formula1}). 
The unramified case follows from (\ref{E:Wnew}). 

%\subsubsection*{The case $v=q$ is a finite prime different from $p$}
Suppose that $v=\pme$ divides $N\LV$. 
Then %, $W=W^0$ is the new vector in the Whittaker model, 
$\Mu_\frakq,\Mu_{\bar\frakq}$ are unramified, the conductors of $\Nu_\frakq$ and $\Nu_{\bar\frakq}$ are $\LV\Z_q$, and 
\[W_{\ol{\frakq}}(\bft(a))=\begin{cases}
\Mu_{\bar\frakq}(a)\abs{a}^{s+\onehalf}\bbI_{\Z_q}(a) &\text{ if $\pme \divides \LV$},\\
W_{\Pi_{\bar\frakq}}&\text{ if $\pme\ndivides \LV$}.
\end{cases}\] 
Here $W_{\Pi_{\bar\frakq}}$ is the spherical vector for $\Pi_{\bar\frakq}=\Mu_{\bar\frakq}\Abs^s\boxplus \Nu_{\bar\frakq}\Abs^{-s}$. 
Put 
\[U_0(NC\Z_q)=\biggl\{\pMX{*}{*}{c}{*}\in\GL_2(\Z_q)\;\biggl|\;c\in NC\Z_q\biggl\}. \]
We claim that $f_{\cD,s,\frakq}$ is supported in $B(\Q_q)U_0(N\LV\Z_q)$. Indeed, if $f_{\cD,s,\frakq}(g)\neq 0$ for $g=\pMX{a}{b}{c}{d}\in\GL_2(\Q_q)$, then $\Phi_{\cD,\frakq}((0,t)g)\neq 0$ for some $t\in \Q_q^\x$. According to the recipe in \defref{D:BS.1}, we find that $(tc,td)\in N\LV\Z_q\oplus\Z_q^\x$, and hence $cd^{-1}\in N\LV\Z_q$. 
Since $f_{\cD,s,\frakq}({\bf1}_2)=1$, we see that
\begin{align*}
&Z_\cD(s,B_{W_\pi})\\
=&[\GL_2(\Z_q):U_0(N\LV\Z_q)]\int_{\Q_q^\x}(\Mu_\frakq\om^{-1})(a)|a|^{s+\onehalf}W_
{\ol{\frakq}}(\bft(a))W_\pi(\bft(a))\rmd^\x a\\
=&\frac{\zeta_q(2)\abs{N\LV}}{\zeta_q(1)}L\biggl(2s+\onehalf,\pi\ot\Nu_+^{-1}\biggl)L\biggl(\onehalf,\pi\ot\Mu_{\bar\frakq}^{-1}\Nu_\frakq^{-1}\biggl).
\end{align*}
The case of a prime factor $\pme$ of $N$ follows from (\ref{E:Wnew}). 

Finally, we assume that $\LV$ is divisible by $\pme$. 
We have 
\begin{align*}B_{\breve W_\pi}(\varsigma_q^{}\cmpt_q^{(\LV)})&=\int_{\Q_q^\x}\breve W_\pi\biggl(\bft(a)\pMX{\LV}{-1}{0}{1}\biggl)\chi_{\ol{\frakq}}^{-1}(a)\rmd^\x a\\
&=\abs{\LV}^\onehalf\varrho_q(C)\int_{\Q_q^\x}\abs{a}^\onehalf \chi_{\ol{\frakq}}^{-1}\varrho_q(a)\varPhi(a)\rmd^\x a
\end{align*}
in view of $\breve W_\pi(\bft(a))=\varrho_q(a)\abs{a}^\onehalf\bbI_{\Z_q}(a)$, where $\varPhi(a)=\bdpsi_q(-a)\bbI_{\LV^{-1}\Z_q}(a)$. 
The integral above equals
\begin{align*}
&\gamma\biggl(\onehalf,\chi_{\ol{\frakq}}^{-1}\varrho_q^{}\biggl)^{-1}\int_{\Q_q^\x}|a|^\onehalf\wh\varPhi(a)(\chi_{\ol{\frakq}}\varrho_q^{-1})(a)\rmd^\x a
\\=&\frac{\vol(\LV^{-1}\Z_q,\rmd a)}{\varepsilon\bigl(\onehalf,\chi_{\ol{\frakq}}^{-1}\varrho_q\bigl)}\vol(1+\LV\Z_q,\rmd^\x a)\\
=&\frac{\zeta_q(1)}{\varepsilon\bigl(\onehalf,\chi_{\ol{\frakq}}^{-1}\varrho_q\bigl)}=\frac{\zeta_q(1)}{\varepsilon\bigl(0,\chi_{\ol{\frakq}}^{-1}\bigl)\varrho_q(C)\abs{C}^\onehalf}
\end{align*}
by the local functional equation (\ref{tag:fq}) for $\GL_1$. 
In the final stage we utilized (\ref{tag:conductor}). 
%, which completes the computation in the case when $q\divides \LV$.
\end{proof}

\begin{prop}\label{P:nonsplit}
If $v=\pme$ remains a prime in $F$ and does not divide $pN\LV$, then 
\[Z_\cD(s,B_{W_\pi})=\Mu(\delta)^{-1}\abs{\delta}_F^{-s-\onehalf}L\biggl(2s+\onehalf,\pi_\pme\ot\Nu_+^{-1}\biggl)B_{W_\pi}({\bf1}_2).\]
\end{prop}

\begin{proof}
%It suffices to consider the case where $p$ is non-split in $K$. 
By assumption $\pi$ and $\chi$ are both unramified. 
Thus $W_\pi=W^0$ is the normalized spherical Whittaker function, and so by the Iwasawa decomposition $\GL_2(\Q_q)=B(\Q_q)\GL_2(\Z_q)$, we have 
%Put \[B_W(g):=\int_{F^\x/\Q_q^\x}W(tg)\chi(t)^{-1}\rmd^\x t. \]
\begin{align*}
Z_\cD(s,B_{W^0})=\int_{\Q_q^\times}G(a)\cdot W^0(\bft(a))\,\rmd^\x a, 
\end{align*}
 where
\[G(a):=\frac{1}{\om(a)\abs{a}}\int_{\Q_\pme}f_{\cD,s,\pme}(\eta \bfn(x)\bft(a))\bdpsi(x)\,\rmd x. \]

Recall that $\Phi_{\cD,\pme}=\abs{\delta}_F^\onehalf\bbI_{\delta^{-1}\cO_{\pme}\oplus \delta^{-1}\cO_{\pme}}$ and $\eta=\delta^{-1}\pMX{1}{-\CMP}{-1}{\ol{\CMP}}$. 
Put $\Phi^0=\bbI_{\cO_q\oplus \cO_q}$ and $\xi:=\Mu\Nu^{-1}\Abs_F^{2s+1}$. 
The computation in the proof of Lemma \ref{L:conv} shows that 
%\[G_N(a)=\Mu(\delta^{-1}a)\frac{|\delta^{-1}a|_F^{s+\onehalf}}{\om(a)\abs{a}}\int_{\pme^{-N}\Z_\pme}\int_{F^\x}(\Mu\Nu^{-1})(t)|t|_F^{2s+1}\Phi^0\biggl((0,t)\pMX{1}{-\CMP}{-1}{\ol{\CMP}}\pMX{a}{x}{0}{1}\biggl)\bdpsi(x)\,\rmd^\x t\rmd x. \] 
\begin{align*}
\Mu(\delta)\abs{\delta}_F^{s}G(a)=&\frac{|a|_F^s}{\Nu(a)}\int_{\Q_\pme}\sum_{m\in\Z}\xi(\pme^m)\Phi^0(-a\pme^m,\pme^m(\ol{\CMP}-x))\bdpsi(x)\,\rmd x.
\end{align*}

If $F/\Q_q$ is unramified, then $B_{W^0}({\bf1}_2)=1$ and 
\begin{align*}
G(a)=\frac{|a|_F^s}{\Nu(a)}\sum_{m=0}^\infty\bbI_{\cO_F}(a\pme^m)\xi(\pme^m)\int_{\pme^{-m}\Z_\pme}\bdpsi(x)\,\rmd x
=\frac{\abs{a}^{2s}}{\Nu(a)}\bbI_{\Z_\pme}(a).
\end{align*}
It follows that
\begin{align*}
Z_\cD(s,B_{W^0})
=&\int_{\Q_\pme^\x}\frac{\abs{a}^{2s}}{\Nu(a)}\bbI_{\Z_\pme}(a)W^0(\bft(a))\,\rmd^\x a
=L\biggl(2s+\onehalf,\pi\ot\Nu_+^{-1}\biggl).
\end{align*}

Next we consider the case where $q$ is ramified in $F$. 
Then $\CMP$ is a uniformizer. 
We see that 
\[B_{W^0}({\bf1}_2)=W^0({\bf1}_2)\cdot\abs{\Delta_F}^\onehalf+W^0(\Psi(\CMP))\chi^{-1}(\CMP)\abs{\Delta_F}^\onehalf \]
from the decomposition $F^\x=\Q_q^\x\cO_F^\x\disjoint \Q_q^\x\cO_F^\x\CMP$ and $\vol(\cO_F^\x,\rmd t_q)=\abs{\Delta_F}^\onehalf$. 
Writing $\pi=\varrho\boxplus \upsilon$, $\al=\varrho(q)$ and $\beta=\upsilon(q)$, we get  
\[\abs{\delta}_F^{-\onehalf}B_{W^0}({\bf1}_2)=1+\chi(\CMP)^{-1}(\al+\beta)\abs{q}^\onehalf=1+(\Mu\Nu)(\CMP^{-1})\abs{q}^\onehalf(\al+\beta) \]
by the Iwasawa decomposition of $\Psi(\CMP)$. 
On the other hand, $\Mu(\delta)\abs{\delta}_F^{s}G(a)$ equals
\begin{align*}
&\frac{|a|_F^s}{\Nu(a)}\int_{\Q_\pme}\sum_{m\in\Z}\bigl\{\xi(\CMP^{2m})\Phi^0(a\CMP^{2m},\CMP^{2m}(\ol{\CMP}-x))\\
&\quad\quad\quad\quad\quad\quad\;+\xi(\CMP^{2m-1})\Phi^0(a\CMP^{2m-1},\CMP^{2m-1}(\ol{\CMP}-x))\bigl\}\bdpsi(x)\,\rmd x\\
=&\Nu(a)^{-1}\abs{a}^{2s}\left(\bbI_{\Z_q}(a)+\xi(\CMP)\bbI_{\Z_q}(a)+\xi(\CMP^{-1})\abs{q}\cdot\bbI_{q\Z_q}(a)\right).
\end{align*}
Since 
\[\sum_{m=1}^\infty W^0\biggl(\pDII{q^m}{1}\biggl)q^{-2ms}\nu(q)^{-m}=\frac{\abs{q}^{2s+\onehalf}}{\nu(q)}\bigl(\alpha+\beta-\alpha\beta\nu(q)^{-1}\abs{q}^{2s+\onehalf}\bigl),  \]
we conclude that
\begin{multline*}
Z_\cD(s,B_{W^0})
=\Mu(\delta^{-1})\abs{\delta}_F^{-s}L\biggl(2s+\onehalf,\pi\ot\Nu_+^{-1}\biggl)\\
\times\bigl\{1+\xi(\CMP)+\xi(\CMP^{-1})\abs{q}\cdot\abs{q}^{2s+\onehalf}\nu(q)^{-1}\bigl(\alpha+\beta-\alpha\beta\nu(q)^{-1}\abs{q}^{2s+\onehalf}\bigl)\bigl\}. 
\end{multline*}
Since $(\Mu\Nu)|_{\Q_q^\times}=\omega$, the second factor equals
\begin{align*}
&1+(\Mu\Nu^{-1})(\CMP)\abs{q}^{2s+1}+(\Mu\Nu)(\CMP^{-1})\abs{q}^\onehalf\bigl(\alpha+\beta-(\Nu^{-1}\om)(q)\abs{q}^{2s+\onehalf}\bigl)\\
&=1+(\Mu\Nu)(\CMP^{-1})\abs{q}^\onehalf(\al+\beta)=B_{W^0}({\bf1}_2)\abs{\delta}_F^{-\onehalf},
\end{align*}
which finishes the proof of the ramified case.
%In the $p$-adic case, $\Phi_p=\phi_{\Mu^{-1}}\ot\wh\phi_{\Nu^{-1}}$. let $\beta=\Mu\Nu^{-1}\Abs^{4s+2}(p)$.
%\begin{align*}
%Z_p=&\int_{p^{-N}\Zp}\sum_{m=0}^\infty\Phi_p(p^m(x+\ol{\CMP}),-p^ma+p^m(x+\ol{\CMP})y)\bdpsi(x)\cdot \beta^m\Mu\Abs^{2s+1}(a)W(\pDII{ap^{2n}}{1})\pMX{1}{0}{p^{2n}y}{1})\rmd x\rmd^\x a\rmd y.
%\end{align*}
\end{proof}

\begin{prop}\label{P:padic}
In the $p$-adic case, if $\pi=\varrho\boxplus \upsilon$ with $\varrho$ unramified and $\upsilon(-1)=1$, then for $n\gg 0$, we have $B_{W_\pi^\Ord}(\cmpt^\setn_p)\neq 0$ and 
\begin{align*} 
Z_\cD(s,B_{\rho(t_n)W_\pi^\Ord})&=\frac{B_{W_\pi^\Ord}(\cmpt^\setn_p)}{\gamma\bigl(2s+\frac{1}{2},\varrho\Nu_+^{-1}\bigl)}(\om^{-1}\varrho)(p^n)\abs{p^n}^\onehalf\frac{\zeta_p(2)}{\zeta_{F_p}(1)}. 
\end{align*}
%If $\pi$ is special and $p$ is inert, then $Z_\cD(s,B_{\rho(t_n)W_\pi^\Ord})=B_{W_\pi^\Ord}(\cmpt^\setn_p)=0$. 
\end{prop}

\begin{proof}
We first assume that $p=\p\pbar$ is split. 
Then $F=\Qp\oplus \Qp$ and  $\Phi_p=\Phi_\p\ot\Phi_\pbar$, where $\Phi_v=\phi_{\Mu_v^{-1}}\ot\wh\phi_{\Nu_v^{-1}}$ with $v=\p$ or $\pbar$. 
From \eqref{E:formula1}  
\[Z_\cD(s,B_{\rho(t_n)W^\Ord_\pi})=\int_{N(\Q_p)\bksl\PGL_2(\Q_p)}\frac{f_{\cD,s,\p}(g)}{\om(\det g)}W_\pbar\biggl(\pDII{1}{-1}g\biggl)W^\Ord_\pi(gt_n)\,\rmd^\tau g. \]
Put $\bfu(x)=\pMX{0}{-1}{-1}{-x}$. 
Using the integration formula 
\[\int_{\PGL_2(\Q_p)}h(g)\,\rmd^\tau g=\frac{\zeta_p(2)}{\zeta_p(1)}\int_{\Q_p}\int_{\Q_p^\times}\int_{\Q_p}h(\bfn(y)\bft(a)J_1\bfn(x))|a|^{-1}\,\rmd y\rmd^\times a\rmd x  \]
for an integrable function $h$ on $\PGL_2(\Q_p)$, we see that $Z_\cD(s,B_{\rho(t_n)W^\Ord_\pi})$ equals  
\begin{align*}&\frac{\zeta_p(2)}{\zeta_p(1)}\int_{\Qp^\x}\int_{\Qp}(\om^{-1}\mu_\p)(a)\abs{a}^{s-\onehalf}f_{\cD,s,\p}(J_1\bfn(x))W_\pbar(\bft(a)\bfu(x))\\
&\times W^\Ord_\pi\biggl(\pDII{ap^n}{p^{-n}}\pMX{1}{0}{-p^{2n}x}{1}\biggl)\,\rmd x\rmd^\x a. \end{align*}
Since $f_{\cD,s,\frakp}(J_1\bfn(x))=\wh\phi_{\Nu_\p^{-1}}(x)$ by (\ref{tag:Godement}), if $n\gg 0$, then 
$Z_\cD(s,B_{\rho(t_n)W^\Ord_\pi})$ equals $\frac{\zeta_p(2)}{\zeta_p(1)\om(p^n)}$ times
\[\int_{\Qp^\x}\int_{\Qp}(\om^{-1}\mu_\p)(a)\abs{a}^{s-\onehalf}\wh\phi_{\Nu_\p^{-1}}(x)W_\pbar(\bft(a)\bfu(x))W^\Ord_\pi(\bft(ap^{2n}))\,\rmd x\rmd^\x a. \]
Since the function $a\mapsto \wh\phi_{\Nu_\p^{-1}}(x)W_\pbar(\bft(a)\bfu(x))$ has a bounded support uniformly with respect to $x$, if $n\gg 0$, then the integral is equal to $\varrho(p^{2n})\abs{p^n}$ times 
\begin{align*}
&\int_{\Qp^\x}\int_{\Qp}(\om^{-1}\mu_\p)(a)\abs{a}^s\wh\phi_{\Nu_\p^{-1}}(x)W_\pbar(\bft(a)\bfu(x))\varrho(a)\bbI_{p^{-2n}\Zp}(a)\,\rmd x\rmd^\x a\\
=&\int_{\Qp^\x}\int_{\Qp}(\upsilon^{-1}\mu_\p)(a)\abs{a}^s\wh\phi_{\Nu_\p^{-1}}(x)W_\pbar(\bft(a)\bfu(x))\,\rmd x\rmd^\x a. 
\end{align*}

Put $\Pi_\pbar=\Mu_\pbar\Abs^s\boxplus \Nu_\pbar\Abs^{-s}$. 
We use the local functional equation for $\GL_2$ (see \S \ref{SS:eps}) to see that the last integral equals the ratio of  
\[\int_{\Qp}\int_{\Qp^\x}(\mu_\p^{-1}\upsilon)(a)\abs{a}^{-s}\wh\phi_{\Nu_\p^{-1}}(x)W_\pbar(\bft(a)J_1^{-1}\bfu(x))\Mu_{\pbar}\Nu_{\pbar}(a^{-1})\,\rmd^\x a\rmd x \]
divided by 
\[\gamma\biggl(s+\onehalf,\mu_\p\upsilon^{-1}\ot\Pi_{\pbar}\biggl)=\gamma\biggl(2s+\onehalf,\varrho\nu_+^{-1}\biggl)\gamma\biggl(\frac{1}{2},\varrho\chi_\pbar^{-1}\biggl). \]
Since 
\begin{align*}
\om&=\varrho\upsilon=\Mu_\p\Nu_\p\Mu_\pbar\Nu_\pbar, & 
\bft(a)J_1^{-1}\bfu(x)&=\bfn(-ax)\bft(-a), & 
W_{\pbar}(\bft(a))&=\bbI_{\Zp^\x}(a)
\end{align*} 
by Lemma \ref{L:FEfinite}, this integral equals  
\begin{align*}&\int_{\Qp^\x}(\varrho^{-1}\Nu_\p)(a)\abs{a}^{-s}W_\pbar(\bft(-a))\int_{\Qp}\wh\phi_{\Nu_\p^{-1}}(x)\bdpsi(-ax)\,\rmd x\rmd^\x a\\
=&\int_{\Qp^\x}\varrho(a)^{-1}\abs{a}^{-s}\bbI_{\Zp^\x}(a)W_\pbar(\bft(-a))\rmd^\x a=1. 
%\frac{L(s+\onehalf,\nu_+^{-1}\chi)L(\onehalf,\Omega_\pbar^{-1}\chi)}{L(\onehalf-s,\nu_+\chi^{-1})L(\onehalf,\Omega_\pbar\chi^{-1})}.
\end{align*}
On the other hand, we see by (\ref{tag:fq}) that 
\begin{align*}
B_{W^\Ord_\pi}(\cmpt_p^\setn)=&\int_{\Qp^\x}W^\Ord_\pi\biggl(\bft(a)\pMX{p^n}{-1}{0}{1}\biggl)\chi_{\pbar}(a)^{-1}\,\rmd^\x a\\
=&\int_{\Qp^\x}\varrho(ap^n)|ap^n|^{1/2}\bdpsi(-a)\bbI_{\Z_p}(ap^n)\chi_{\pbar}(a)^{-1}\,\rmd^\x a\\
=&\varrho(p^n)\abs{p^n}^{1/2}\frac{\vol(p^{-n}\Zp,\rmd a)}{\gamma\bigl(\frac{1}{2},\varrho\chi_\pbar^{-1}\bigl)}\vol(1+p^n\Zp,\rmd^\x a)\\
=&\varrho(p^n)\abs{p^n}^{1/2}\frac{\zeta_p(1)}{\gamma\bigl(\frac{1}{2},\varrho\chi_\pbar^{-1}\bigl)}.
\end{align*}

Now we consider the case where $p$ is inert in $F$ and $\pi$ is a principal series. %The assertion for special representations is well-known. 
Using the decomposition $\GL_2(\Q_p)=\Psi(F^\times)\cdot B(\Q_p)$, we have 
\begin{align*}
&Z_\cD(s,B_{\rho(t_n)W^\Ord_\pi})\\
=&\int_{\Qp}\int_{\Qp^\x}f_{\cD,s,p}(\eta\bft(a)\bfn(x))\om(a)^{-1}B_{W^\Ord_\pi}(\bft(a)\bfn(x)t_n)\abs{a}\rmd^\x a\rmd x.
\end{align*}
We proceed to compute
\begin{align*}
&f_{\cD,s,p}(\eta \bft(a)\bfn(x))\\
=&(\Mu\Nu)(\delta^{-1}) f_{\cD,s,p}\biggl(\pMX{1}{-\CMP}{-1}{\ol{\CMP}}\pMX{a}{ax}{0}{1}\biggl)\\
=&\Nu(\delta^{-1})\Mu(a)\abs{a}_F^{s+\onehalf}\int_{F^\x}\Phi_{\cD,p}(-ta,t(\ol{\CMP}-xa))(\Mu\Nu^{-1})(t)\abs{t}_F^{2s+1}\rmd^\x t\\
=&\Nu(\delta^{-1})\Nu(a)\abs{a}_F^{-s-\onehalf}\int_{F^\x}\Phi_{\cD,p}(-t,a^{-1}t\ol{\CMP}-xt)(\Mu\Nu^{-1})(t)\abs{t}_F^{2s+1}\rmd^\x t.
\end{align*}
Since $\Phi_{\cD,p}=\phi_{\Mu^{-1}}\ot\wh\phi_{\Nu^{-1}}$, we find that
\[f_{\cD,s,p}(\eta \bft(a)\bfn(x))=\mu(-1)\Nu(a\delta^{-1})\abs{a}_F^{-s-\onehalf}\wh\phi_{\Nu^{-1}}(a^{-1}\ol{\CMP}-x). \]
In particular, the function $x\mapsto \wh\phi_{\Nu^{-1}}(a^{-1}\ol{\CMP}-x)$ has a bounded support with respect to $a$. 
Hence for $n\gg_{\Nu} 0$
\begin{align*}
&\mu(-1)\Nu(\delta)Z_\cD(s,B_{\rho(t_n)W^\Ord_\pi})\\
=&\int_{\Qp}\int_{\Qp^\x}\Nu(a)\abs{a}^{-2s}\wh\phi_{\Nu^{-1}}(a^{-1}\ol{\CMP}-x)\om(a)^{-1}B_{W^\Ord_\pi}(\bft(a)t_n)\,\rmd^\x a\rmd x\\
=&\int_{\Qp^\x}\Nu(a)^{-1}\abs{a}^{2s}\varPhi_2(a)\om(a)B_{W^\Ord_\pi}(\bft(a^{-1})t_n)\,\rmd^\x a, 
\end{align*}
%\begin{align*}
%Z_\cD(s,B_{\rho(t_n)W^\Ord})=&\Nu(\delta)^{-1}\int_{\Qp^\x}\int_{\Qp^\x} \om^{-1}\Mu\Abs^{2s+1}(a)\int_{K^\x}\Phi_p(-ua,u(-xa+\ol{\CMP})\Mu\Nu^{-1}\Abs_K^{2s+1}(u)B_{W^{\Ord}}(\bft(a)t_n\pMX{1}{0}{p^{2n}x}{1})\,\rmd^\x u\abs{a}\rmd^\x a\rmd x\\
%=&\Nu(\delta)^{-1}\int_{\Qp}\int_{\Qp^\x}\int_{K^\x} \om\Nu^{-1}\Abs^{2s}(a^{-1})\Phi_p(-u,-xu+a^{-1}u\ol{\CMP})\Mu\Nu^{-1}\Abs_K^{2s+1}(u)B_{W^\Ord}(\bft(a)t_n)\,\rmd^\x u\rmd^\x a\rmd x.
%\end{align*}
%Recall that $\Phi_p=\phi_{\Mu^{-1}}\ot\wh\phi_{\Nu^{-1}}$, so the above integral equlas
where $\varPhi_2(a)\in\cS(\Qp)$ is defined by
\[\varPhi_2(a):=\int_{\Qp}\wh\phi_{\Nu^{-1}}(x+a\ol{\CMP})\,\rmd x. \]

Observe that if $\varPhi_2(a)\neq 0$, then 
\begin{align*}
\om(a)B_{W^\Ord_\pi}(\bft(a^{-1})t_n)&=\om(p^n)^{-1}B_{W^\Ord_\pi}(J_1^{-1}\bft(ap^{2n}))\\
&=\om(p^n)^{-1}\varrho(ap^{2n})\abs{a p^{2n}}^\onehalf\cZ(\wtd W), 
\end{align*}
where
\[\cZ(\wtd W)=\int_{\Qp^\x}\varrho(t)\abs{t}^\onehalf\bbI_{\Z_p}(atp^{2n})\wtd W(\bft(t)J_1)\,\rmd^\x t=\int_{\Qp^\x}\varrho(t)\abs{t}^\onehalf\wtd W(\bft(t)J_1)\,\rmd^\x t \]
for $n\gg_{\nu,\wtd W} 0$. 
Thus we find that $\mu(-1)\Nu(\delta)Z_\cD(s,B_{\rho(t_n)W^\Ord_\pi})$ is equal to 
\begin{align*}
(\om^{-1}\varrho^2)(p^n)\abs{p^n}\cZ(\wtd W)\int_{\Qp^\x}(\Nu^{-1}\varrho)(a)\abs{a}^{2s+\onehalf}\varPhi_2(a)\,\rmd^\x a.
\end{align*}
The last integral equals
\[\gamma\biggl(2s+\onehalf,\varrho\Nu_+^{-1}\biggl)^{-1}\int_{\Qp^\x}(\Nu\varrho^{-1})(a)\abs{a}^{\onehalf-2s}\wh\varPhi_2(a)\,\rmd^\x a\]
by (\ref{tag:fq}), where 
\begin{align*}
\wh\varPhi_2(a)=&\int_{\Qp}\int_{\Qp}\wh\phi_{\Nu^{-1}}(x+y\ol{\CMP})\bdpsi(ay)\,\rmd x\rmd y\\
=&\int_F\wh\phi_{\Nu^{-1}}(z)\bdpsi_F(a\delta^{-1}z)\,\rmd z=\phi_{\Nu^{-1}}(-a\delta^{-1}).
\end{align*}
We conclude that 
\[Z_\cD(s,B_{\rho(t_n)W^\Ord_\pi})=\gamma\biggl(2s+\onehalf,\varrho\Nu_+^{-1}\biggl)^{-1}(\om^{-1}\varrho^2)(p^n)\abs{p^n}\om(-1)\cdot \cZ(\wtd W).\]
On the other hand, for $n\gg 0$, 
\[B_{W^\Ord_\pi}(\cmpt^\setn_p)=\varrho(p^n)\abs{p^n}^\onehalf\cZ(\wtd W).\]
The following lemma will complete our proof. 
\end{proof}

\begin{lm}
$\cZ(\wtd W)\neq 0$. 
\end{lm}
\begin{proof}
Let $\xi:=\chi^{-1}\upsilon_F$. 
If $\xi$ is unramified, then so is $\xi|_{\Qp^\x}=\om^{-1}\upsilon^2=\varrho^{-1}\upsilon$, which implies that both $\pi$ and $\chi$ are unramified, so that $\wtd W$ is the spherical Whittaker function, and 
\[\cZ(\wtd W)=L(1,\pi^\vee\ot\varrho)\neq 0.\]
Suppose that $\xi$ is a ramified character. 
Since $\chi\om_F^{-\onehalf}$ is assumed to be unramified, we find that $c(\xi|_{\Qp^\x})=c(\varrho^{-1}\upsilon)=c(\om)>0$. 
Let $\wtd f\in \upsilon^{-1}\boxplus\varrho^{-1}$ be the unique section such that 
\[\wtd f\biggl(\pMX{a}{b}{0}{d}\Psi(t)\biggl)=\upsilon(a)^{-1}\varrho(d)^{-1}\abs{\frac{a}{d}}^\onehalf\chi^{-1}(t)\]
for $a,d\in \Q_p^\times$, $b\in\Q_p$ and $t\in F^\x$. 
Then we can choose $\wtd W(g):=W(g,\wtd f)$, and $\wtd W(\bft(a)J_1)$ equals
\[\int_{\Qp}^{\rm st}\wtd f\biggl(\pMX{-1}{0}{x}{-a}\biggl)\bdpsi(-x)\,\rmd x=\varrho(a)^{-1}\abs{a}^\onehalf\int_{\Qp}^{\rm st}\wtd f\biggl(\pMX{-1}{0}{x}{-1}\biggl)\bdpsi(-ax)\,\rmd x.\]
Since 
\[\pMX{-1}{0}{x}{-1}=\pMX{\rmN(x\CMP-1)^{-1}}{*}{0}{1}\Psi(x\CMP-1),\]
we find that
\[\wtd W(\bft(a)J_1)=\varrho(a)^{-1}\abs{a}^\onehalf\int_{\Qp}^{\rm st}\frac{\xi(x\CMP-1)}{\abs{x\CMP-1}_F^{1/2}}\bdpsi(-ax)\,\rmd x.\]
Put $\Phi_N(x):=\frac{\xi(x\CMP-1)}{\abs{x\CMP-1}_F^{1/2}}\bbI_{p^{-N}\Zp}(x)$. 
We have seen that 
\[\varrho(a)\abs{a}^{-\onehalf}\wtd W(\bft(a)J_1)=\lim_{N\to\infty}\wh\Phi_N(a). \]
Take an integer $B>c(\om)$. 
Then we have
\[\wh\Phi_N(a)=\wh\Phi_B(a)+\sum_{j=B+1}^N I_j \]
for $N>B$, where
\[I_j=\int_{p^{-j}\Z_p^\times}\frac{\xi(x\CMP-1)}{\abs{x\CMP-1}_F^{1/2}}\bdpsi(-ax)\,\rmd x. \]
Note that $I_j=0$ unless $j=\mathrm{ord}_p(a)+c(\om)$. 
Recall an additive character $\bdpsi^a$ defined by $\bdpsi^a(x)=\bdpsi(ax)$. 
Then 
\[I_{\ord_p(a)+c(\om)}=\xi(\theta)\varepsilon\biggl(-\frac{1}{2},\varrho\upsilon^{-1},\bdpsi^{-a}\biggl)=\xi(\theta)(\varrho\upsilon^{-1})(-a)|a|^{-1}\varepsilon\biggl(-\frac{1}{2},\varrho\upsilon^{-1}\biggl) \]
by (\ref{tag:conductor}). 
Since $\varrho\upsilon^{-1}$ is ramified, we see that $\cZ(\wtd W)$ is equal to 
\begin{align*}
&\int_{\Qp}\biggl(\wh\Phi_B(a)+\xi(\theta)(\varrho\upsilon^{-1})(-a)|a|^{-1}\varepsilon\biggl(-\frac{1}{2},\varrho\upsilon^{-1}\biggl)\bbI_{p^{B+1-c(\omega)}\Z_p}(a)\biggl)\abs{a}\,\rmd^\x a\\
=&\Phi_B(0)\zeta_p(1)\neq 0. \end{align*}
\end{proof}
\subsection{The explicit pull-back formula}
Now we are ready to give the explicit formula of $Z_\cD(s,\rho(\cJ_\infty t_n)\varphi_f)$. 
The notation is as in \subsecref{SS:4.3}.
Let $\tau_F$ be the quadratic Dirichlet character associated to the extension $F/\Q$. 

\begin{thm}\label{T:formula} 
Let $\lam$ be a Hecke character of $\A^\x$ of $p$-power conductor and $\brch$ be a finite order Hecke character of $\A_F^\x$ with $\brch|_{\A^\x}=1$ and conductor $\LV\cO_F$. Put $\chi=\om_F^\onehalf\brch$ and $$\cD=(\om_F^\onehalf\lam_F,\brch^{-1}\lam_F^{-1},k,\frakN,\frakc).$$For $n\gg 0$, we have
\begin{align*}
\frac{Z_\cD(s,\rho(\cJ_\infty t_n)\varphi_f)}{B_{\rho(\cJ_\infty)\breve\varphi_f}^\chi(\cmpt^{(\LV p^n)})}&=L^{\{p\}}\biggl(2s+\onehalf,\pi\ot\Nu_+^{-1}\biggl)\gamma\biggl(2s+\onehalf,\varrho_p\nu_{+,p}^{-1}\biggl)^{-1}\\
&\times (\om_p^{-1}\varrho_p^{})(p^n)\abs{p^n}_{\Qp}^\onehalf\frac{\zeta_p(2)}{\zeta_{F_p}(1)}\cdot \frac{L(1,\tau_F)\frakf_\infty\frakf_{\rm ram}\frakf_C}{\zeta_\Q(2)[\SL_2(\Z):\Gamma_0(N\LV)]},
\end{align*}
where $\frakf_\infty$, $\frakf_{\rm ram}$ and $\frakf_{\LV}$ are local fudge factors given by 
\begin{align*}&\frakf_\infty:=4(-4\sqrt{-1})^{-k}(\lam_\infty\phi_{\sg_1})(-1),\\ 
&\frakf_{\rm ram}:=\prod_{q\divides\Delta_F}\om_q^\onehalf\lam_q(\Delta_F^{-1})\abs{\Delta_F}_{\Q_q}^{-s-\onehalf},\\
&\frakf_{\LV}:=\prod_{q\divides C}\om_q^{\onehalf}(\LV^{-1})\frac{\varepsilon(0,\brch^{-1}_{\ol{\frakq}})}{\zeta_q(1)}. \end{align*}
\end{thm}

\begin{proof}
There exists a nonzero constant $c$ such that $B_\varphi^\chi(g)=c\prod_v B_{W_v}(g_v)$ for $\varphi\in\pi$ with $W_\varphi(g)=\prod_vW_v(g_v)$ by the uniqueness and the existence of the Waldspurger models. 
Put 
\begin{align*}
\varphi^\star&=\rho(\cJ_\infty t_n)\breve\varphi_f, &  
W_\infty^\star&=\rho(\cJ_\infty)W_{\pi_\infty}, & 
W_p^\star&=\rho(t_n)W_{\pi_p}^\Ord. 
\end{align*}
The Whittaker function of $\varphi^\star$ is given by 
\[W_{\varphi^\star}(g)=W^\star_\infty(g_\infty)\cdot W^\star_p(g_p)\cdot \prod_{q\divides \LV}\breve W_{\pi_q}(g_q)\prod_{\ell\ndivides p\LV}W_{\pi_\ell}(g_\ell). \]
It follows that 
\[B_{\rho(\cJ_\infty)\breve\varphi_f}^\chi(\cmpt^{(\LV p^n)})=cB_{W^\star_\infty}(\cmpt_\infty)\cdot B_{W_{\pi_p}^\Ord}(\cmpt_p^{(n)})\cdot \prod_{q\divides \LV}B_{\breve W_{\pi_q}}(\cmpt_q\cmpt_q^{(C)})\prod_{\ell\ndivides p\LV}W_{\pi_\ell}(\cmpt_\ell). \]
On the other hand, \thmref{T:KeatonPitale} gives 
\[Z_\cD(s,\rho(\cJ_\infty t_n)\varphi_f)=c\frac{L(1,\tau_F)}{\zeta_\Q(2)}\cdot Z_\cD(s,B_{W_\infty^\star})\cdot Z_\cD(s,B_{W_p^\star})\prod_{q\neq p} Z_\cD(s,B_{W_{\pi_q}}). \]
Theorem \ref{T:formula} now follows from Propositions \ref{P:spl}, \ref{P:nonsplit} and \ref{P:padic} with $\Mu=\om_F^\onehalf\lam_F$, $\Nu=\brch^{-1}\lam_F^{-1}$ and $\chi=\om_F^\onehalf\brch$.
\end{proof}
%!TEX root = pbk_jT.tex

\def\brch{\phi}
\def\bftheta{\vartheta}
\section{The construction of \padic twisted triple product  $L$-functions}\label{S:5}

\subsection{Notation}
Define the $p$-adic cyclotomic character by \[\cyc:\Q^\x\bksl \A^\x\to\Zp^\x,\quad \cyc(a)=\abs{a}_\A a_\infty^{-1}a_p.\]
Let $\Om:\Q^\x\bksl \A^\x\to\mu_{p-1}(\Cp)$ be the \Teich character. 
Fix embeddings $\iota_\infty:\Qbar\hookto\C$ and $\iota_p:\Qbar\hookto\Cp$ once and for all. The set of embeddings $\Sigma_\R=\stt{\sg_1,\sg_2}$ from $F$ to $\R$ is identified with $\Gal(F/\Q)$ via $\iota_\infty$. 

Let $\mathcal{O}=\mathcal{O}_L$ for some finite extension $L$ of $\Q_p$ containing $\iota_p(F)$. Let $\Lam=\cO\powerseries{1+p\Zp}$ and write $[\cdot]:1+p\Zp\to \Lam^\x$ for the inclusion of group-like elements. Let $\bfu=1+p$. For a variable $X$, let $\Dmd{\cdot}_X:\Zp^\x\to\Zp\powerseries{X}^\x$ be the character defined by 
\beq\label{E:defn1}\Dmd{a}_X:=(1+X)^{\frac{\log_pa}{\log_p\bfu}}.\eeq
Write $\rmN=\rmN
_{F/\Q}:F\to\Q$ for the norm map. If $\fraka$ is a fractional ideal of $F$ coprime to $p$, put $\Dmd{\fraka}_X=\Dmd{\rmN(\fraka)}_X$. If $\bfI$ is a finite extension of $\Lam$, a point $Q\in\Spec\bfI(\Cp)$ is called a locally algebraic point of weight $k$ and finite part $\ep$ if the map $Q|_{\Lam}:1+p\Zp\overset{[\cdot]}{\longto}\Lam^\x\overset{Q}{\longto}\Qbar_p^\x$ is given by $Q(x)=x^k\ep(x)$ for some integer $k\geq 1$ and a finite order character $\ep:1+p\Zp\to \mu_{p^\infty}(\Qbar_p)$. For a locally algebraic point $Q$ we denote by $k_Q$ the weight of $Q$ and $\ep_Q$ the finite part of $Q$.  Let $\frakX_\bfI^+$ be the set of locally algebraic points $Q$ in $\Spec\bfI(\Cp)$ with $k_Q\geq 1$. A locally algebraic point $Q\in\frakX_\bfI^+$ is called arithmetic if $k_Q\geq 2$. 
For every arithmetic point $Q\in \frakX_\bfI^+$, we shall view the finite part $\ep_Q$ as a Hecke character of $\A^\x$ via $\ep_Q(a):=\iota_\infty\iota_p^{-1}(\ep_Q(\cyc(a)\Om^{-1}(a)))$.
If $A$ and $B$ are two complete $\cO$-modules, we write $A\wh \ot B$ for $A\wh\ot_\cO B$ for simplicity.

\subsection{Preliminaries on Hida theory for modular forms}
Let ${\bf I}$ be a normal domain finite flat over $\Lambda$. Let $N$ be a positive integer prime to $p$ and let $\chi:(\Z/Np\Z)^\x \to\cO^\x$ be a Dirichlet character modulo $Np$. Denote by $\bfM(N,\chi,\bfI)$ the space of $\bfI$-adic modular forms of tame level $N$ and (even) branch character $\chi$, consisting of  formal power series $\bdsf(q)=\sum_{n\geq 1}\bfa(n,\bdsf)q^n\in \bfI\powerseries{q}$ with the following property: there exists an integer $a_\bdsf$ such that for every point $Q\in\frakX^+_\bfI$ with $k_Q\con 0\pmod{2}$ and $k_Q\geq a_\bdsf$, the specialization $\bdsf_Q(q)=\sum_{n\geq 1}Q(\bfa(n,\bdsf))q^n$ is the $q$-expansion of a cusp form $\bdsf_Q\in \cM_{k_Q}(Np^r,\chi\Om^{2-k_Q}\ep_Q)$ for some $r>0$. We call $\bdsf_Q$ the \emph{specialization} of $\bdsf$ at $Q$. For a positive integer $d$ prime to $p$, define $V_d:\bfM(N,\chi,\bfI)\to \bfM(Nd,\chi,\bfI)$ by $V_d(\sum_n\bda(n,\bdsf)q^n)=d\sum_n\bda(n,\bdsf)q^{dn}$.  Let $\bfS(N,\chi,\bfI)\subset \bfM(N,\chi,\bfI)$ be the space of $\bfI$-adic cusp forms, consisting of elements $\bdsf\in \bfM(N,\chi,\bfI)$ such that $\bdsf_Q$ is a cusp form for a Zariski dense subset $Q\in\frakX_\bfI^+$.

The space $\bfM(N,\chi,\bfI)$ is equipped with the action of the usual Hecke operators $T_\ell$ for $\ell\ndivides Np$ as in \cite[page 537]{Wiles88} and the operators $\bfU_\ell$ for $\ell\divides pN$ given by $\bfU_\ell(\sum_n \bfa(n,\bdsf)q^n)=\sum_n\bda(n\ell,\bdsf)q^n$. Recall that Hida's ordinary projector $\eord$ defined by 
\[\eord:=\lim_{n\to\infty}\bfU_p^{n!}.\]
is a convergent operator on the space of classical modular forms preserving the cuspidal part as well as on the spaces $\bfM(N,\chi,\bfI)$ and $\bfS(N,\chi,\bfI)$ (\cf\cite[page 537 and Proposition 1.2.1]{Wiles88}). The space $\eord\bfS(N,\chi,\bfI)$ consists of ordinary $\bfI$-adic forms defined over $\bfI$. A key result in Hida's theory of ordinary $\bfI$-adic cusp forms says that if $\bdsf\in \eord\bfS(N,\chi,\bfI)$, then $\bdsf_Q\in \eord\cS_{k_Q}(Np^e,\chi\Om^{2-k_Q}\ep_Q)$ for \emph{every} arithmetic point $Q\in\frakX_\bfI^+$. 
We call $\bdsf\in\eord\bfS(N,\chi,\bfI)$ a \emph{primitive Hida family} if $\bdsf_Q$ is a $p$-stabilized newform of tame conductor $N$ for every arithmetic point $Q\in\frakX_\bfI^+$.

For a divisor $M\divides N$, let $\bfT(N,M)\subset\End \eord\bfS(N,\chi,\bfI)$ be the $\bfI$-algebra generated by Hecke operators $\stt{T_q}_{q\ndivides Np}$ and $\stt{\bfU_q}_{q\divides Mp}$.  A classical result in Hida theory for modular forms asserts that $\bfT(N,\bfI)$ is free of finite rank over $\bfI$. Let $\bdsf\in\eord\bfS(N,\chi,\bfI)$ be a primitive Hida family. Then $\bdsf$ induces the $\bfI$-algebra homomorphism $\lam_\bdsf:\bfT(N,\bfI)\to \bfI$ with $\lam_\bdsf(T_q)=\bfa(q,\bdsf)$ for $q\ndivides Np$ and $\lam_\bdsf(\bfU_q)=\bfa(q,\bdsf)$ for $q\divides Np$. 
We denote by $\frakm_\bdsf$ the maximal ideal of $\bfT(N,\bfI)$ containing $\Ker\lam_\bdsf$ and by $\bfT_{\frakm_\bdsf}$ the localization of $\bfT(N,\bfI)$ at $\frakm_\bdsf$. It is the local ring of $\bfT(N,\bfI)$ through which $\lam_\bdsf$ factors. 
It is well-known that %$\bfT_{\frakm_\bdsf}$ is a local finite flat $\Lam$-algebra, and 
there is an algebra direct sum decomposition 
\begin{align*}
\wtd\lam_\bdsf&:\bfT_{\frakm_\bdsf}\ot_{\bfI}\Frac\bfI\iso\Frac\bfI\oplus \sB, & 
t&\mapsto \wtd\lam_\bdsf(t)=(\lam_\bdsf(t),\lam_\sB(t)),
\end{align*} 
where $\sB$ is a finite dimensional $(\Frac\bfI)$-algebra (\cite[Corollary 3.7]{Hida88Annals}). 
%Then we have  \[C(\bdsf)=\lam_\bdsf(\bfT_{\frakm_\bdsf}\cap \wtd\lam_\bdsf^{-1}(\Frac\bfI\oplus\stt{0})) \] by definition. \textcolor{red}{By the primitiveness of $\bdsf$, there exists a unique idempotent $1_\bdsf$ in $\bfT(N,\bfI)\ot_{\bfI}\Frac \bfI$ such that $\lam_\bdsf(1_\bdsf)=1$.}

\begin{Remark}\label{R:con.5}
 Recall that the congruence ideal $C(\bdsf)$ of $\bdsf$ is defined by 
 \[C(\bdsf):=\lam_\bdsf(\Ann_{\bfT_{\frakm_\bdsf}}(\Ker\lam_\bdsf))\subset\bfI. \]
 %\[C(\bdsf):=\lam_\bdsf(\{t\in \bfT(N,\bfI)\mid 1_\bdsf t=t\})\subset \bfI.\] 
By definition, $C(\bdsf)\cdot 1_\bdsf\subset \bfT(N,\bfI)$ and $C(\bdsf)$ is the annihilator of the congruence module of $\lam_\bdsf$ (see \cite[Definition 6.1]{Hida88AJM}). For each arithmetic point $Q\in\frakX_\bfI^+$, let $\wp_Q=\ker Q$. By the control theorem for the Hecke algebras and the congruence modules (\cf \cite[(0.4b), (5.8a)]{Hida88AJM}), we find that $Q(C(\bdsf))$ is the congruence ideal for $\lam_{\bdsf_Q}:\bfT(N,\bfI)/\wp_Q\to \bfI/\wp_Q$. In particular, this implies $Q(C(\bdsf))\neq 0$ and hence $1_\bdsf$ belongs to the localization $\bfT(N,\bfI)_{\wp_Q}$ at $\wp_Q$. This fact will be used to deduce the finiteness of our $p$-adic $L$-function in \defref{D:padicL.5} at arithmetic points.
\end{Remark}

\subsection{A two-variable $p$-adic family of Hilbert-Eisenstein series}
We shall make the identification
\[\Lam\wh\ot\Lam=\cO\powerseries{X,T},\quad X=([\bfu]-1)\ot 1,\,T=1\ot ([\bfu]-1). \]
Let $(\chi_1,\chi_2)$ be a pair of finite order Hecke characters of $\A_F^\x$ of level $p\cO_F$ and $p\LV\cO_\cK $. 
We assume that $(\chi_1,\chi_2)$ satisfies Hypothesis \ref{H:31} and $\chi_1\chi_2$ is totally even. A Hecke character $\chi$ of $\A_F^\x$ will be viewed as an ideal class character by \[\chi(\frakq):=\chi(\uf_\frakq)^{-1}\] for any prime ideal $\frakq$ away from the conductor of $\chi$. Define the $\Lam\wh\ot \Lam$-adic $q$-expansion by 
\[\bdsE(\chi_1,\chi_2)(X,T):=\sum_{\beta\in\frakd^{-1}_+,(p,(\beta))=1}\cA_\beta(\chi_1,\chi_2)q^\beta\in \Lam\wh\ot \Lam\powerseries{q^{\frakd^{-1}_+}},\]
where $\cA_\beta(\chi_1,\chi_2)\in\Lam\wh\ot\Lam$ is defined by
\begin{align*}\cA_\beta(\chi_1,\chi_2)=&\Dmd{(\beta)}_X\Dmd{(\beta)}^{-1}_{T}\chi^{-1}_{1}((\beta))\prod_{\frakq\ndivides \frakc p}\cP_{\beta,\frakq}(\chi_1\chi_2^{-1}(\frakq)\Dmd{\frakq}_{X}^{-1}\Dmd{\frakq}^{2}_{T})\\
&\times \prod_{\frakq\divides (\frakc,\beta)}\cQ_{\chi_1\chi_2^{-1},\frakq}(\Dmd{\pmq}_{X}^{-1}\Dmd{\frakq}^{2}_{T}),\end{align*}
where $\cP_{\beta,\frakq}$ and $\cQ_{\chi_1\chi_2^{-1},\frakq}$ are polynomials defined in \eqref{E:poly.3}.
If $R$ is an $\cO_F$-algebra, the theta operator $\theta_\sg\in\End(R\powerseries{q^{\frakd_+^{-1}}})$ for $\sg\in\Gal(F/\Q)$ is defined by  
\[\theta_\sg(\sum_{\beta}a_\beta q^\beta)=\sum_\beta \sg(\beta)a_\beta q^\beta. \] 
For $Q\in\frakX_\Lam$, let $\xi_Q$ be the finite order Hecke character of $\A_F^\x$ given by 
\[\xi_Q:=\ep_Q\Om^{-k_Q}\circ\rmN.\]

\begin{prop}\label{P:interpolateE}
For every $(\Qx,\Qz)\in \frakX_\Lam^+\times\frakX_\Lam^+$ with $k_\Qx\leq k_\Qz$, we have the interpolation 
\begin{align*}\bdsE(\chi_1,\chi_2)(\Qx,\Qz)
=&\begin{cases}\theta^{k_\Qx-k_\Qz}E^+_{2k_\Qz-k_\Qx}(\chi_1\xi_Q^{-1}\xi_P,\,\chi_2\xi_P^{-1})&\text{ if }2k_\Qz>k_\Qx, \\[0.5em]
\theta^{k_\Qx-1}E^-_{k_\Qx-2k_\Qz+2}(\chi_1\xi_\Qx^{-1}\xi_{\Qz},\,\chi_2\xi_{\Qz}^{-1})&\text{ if }2k_\Qz\leq k_\Qx,\end{cases}
\end{align*}
where $\theta=\theta_{\sg_1}\theta_{\sg_2}$ is the theta operator $\theta(\sum_{\beta}a_\beta q^\beta)=\sum_\beta \rmN(\beta)a_\beta q^\beta$.
\end{prop}
\begin{proof}
Let $\Mu=\chi_1\xi_Q^{-1}\xi_P$ and $\Nu=\chi_2\xi_P^{-1}$. Put $\bfk=2k_\Qz-k_\Qx$. For an integer $n$ prime to $p$, we have  
\begin{align*}\cA_\beta(\chi_1,\chi_2)(\Qx,\Qz)=&\rmN(\beta)^{k_\Qx-k_\Qz}\Mu^{-1}((\beta))\prod_{\pmq\ndivides \frakc p}\cP_{\beta,\pmq}(\Mu\Nu^{-1}(\pmq)\qv^\bfk)\\
&\times \prod_{\pmq\divides (\frakc,\beta)}\cQ_{\chi_1^{-1}\chi_2,\pmq}(\chi_1^{-1}\chi_2\Mu\Nu^{-1}(\pmq)\qv^{\bfk}).\end{align*}
Since $\chi_1^{-1}\chi_2^{}\Mu\Nu^{-1}$ is unramified outside $p$, one verifies that \[
\cQ_{\chi_1^{-1}\chi_2,\pmq}(\chi_1^{-1}\chi_2\Mu\Nu^{-1}(\pmq)X)=\cQ_{\Mu^{-1}\Nu,\pmq}(X).
\]
By \corref{C:FE}, we find that 
\begin{align*}\cA_\beta(\chi_1,\chi_2)(\Qx,\Qz)=&\begin{cases}\rmN(\beta)^{k_\Qx-k_\Qz}\cdot \sigma^+_\beta(\Mu,\Nu,\bfk)&\text{ if }\bfk>0,\\
\rmN(\beta)^{1-\bfk+k_\Qz-1}\cdot \sigma^-_\beta(\Mu,\Nu,2-\bfk)&\text{ if }\bfk\leq 0.
\end{cases}
\end{align*}
The proposition follows immediately.
\end{proof}

\subsection{The construction of the twisted triple $p$-adic $L$-function}\label{SS:setting}
For any $\cO_F$-algebra $R$, define the \emph{diagonal restriction map} by  
\begin{align*}
{\rm res}_{\cK/\Q}:R\powerseries{q^{\frakd_+^{-1}}}\to R\powerseries{q},\quad
\sum_{\beta\in\frakd_+^{-1}}a_\beta q^\beta\mapsto\sum_{n>0}\Big(\sum_{\substack{\beta\in\frakd_+^{-1}\\ \Tr_{\cK/\Q}(\beta)=n}}a_\beta\Big)q^n.\end{align*}
For an even integer $a$ and a finite order Hecke character \beq\label{E:brch.5}\brch:\cK^\x\bksl \A_\cK^\x/\wh\cO_\LV^\x\to\cO^\x\text{ such that }\brch_\sg(-1)=(-1)^\frac{j}{2}\text{ for }\sg\in\Sigma_\R,\eeq we define the two-variable $q$-expansion $\bdsE_\brch^{[a]}(X,T)\in\Lam\wh\ot \Lam\powerseries{q^{\frakd_+^{-1}}}$ by
\[\bdsE_\brch^{[a]}(X,T)=\bdsE(\Om_F^\frac{a-j}{2},\Om_F^{-\frac{a}{2}}\brch)(1+X)^{1/2}-1,(1+T)^{1/2}-1). \] 
We define $\bdsG_\brch^{[a]}(X,T)\in\Lam\wh\ot \Lam\powerseries{q}$ as the diagonal restriction 
\[\bdsG_\brch^{[a]}(X,T):={\rm res}_{\cK/\Q}\bigl(\bdsE_\brch^{[a]}(X,T)\bigl).\]

We regard $\Lam$ as a subring of $\Lam\wh\ot\Lam$ via $x\mapsto x\ot 1$. Let 
\[\frakX_\bfI^{++}:=\stt{Q\in \frakX_\bfI^+\mid k_Q\con 0\pmod{2}}\subset\frakX_\bfI^+.\] 
\begin{lm}\label{L:1.6}The $q$-expansion $\bdsG_\brch^{[a]}$ belongs to $\bfM(N\LV,\Om^{j-2},\Lam) \widehat{\otimes}_{\Lam}(\Lam\wh\ot\Lam)$.
\end{lm}
\begin{proof}
Let $Z=(1+T)(1+X)^{-1}-1$ and write 
\[\bdsG(X,Z)=\bdsG_\brch^{[a]}(X,(1+X)(1+Z)-1). \]
If  $\zeta\in \mu_{p^\infty}(\C)$ is a $p$-power root of unity, let $\bfal_\zeta:\A_F^\x\to\C^\x$ be the Hecke character $\bfal_\zeta(a)= \Dmd{\rmN(a)}_X|_{X=\zeta-1}$. By \propref{P:interpolateE}, for any point $Q\in\frakX_\bfI^{++}$, we have
\[\bdsG(Q,\zeta-1)=E^+_{k_Q/2}(\Mu_
{Q,\zeta},\Nu_{Q,\zeta})|_{\frakH}\in \cM_{k_Q}(\LV N,\Om^{j-2}\xi_Q),\]
where $\Mu_{Q,\zeta}=\Om_F^\frac{a-j}{2}\bfal_\zeta$ and $\Nu
_{Q,\zeta}=\Om_F^{-\frac{a}{2}}\bfal_\zeta^{-1}\xi_Q^{-\onehalf}$. 
This shows that 
\[\bdsG(X,\zeta-1)\in \bfM(N\LV,\Om^{j-2},\Lam)\ot_\cO\cO[\zeta]\] 
for every $\zeta\in\mu_{p^\infty}(\Cp)$.
%or equivalently, \[\bdsG(X,Z)\pmod{(1+Z)^{p^n}-1}\in \bfM(N\LV,\Om^{j-2},\Lam)\ot_\cO\cO\powerseries{Z}/((1+Z)^{p^n}-1)\] for all $n$.
We see that 
\[\bdsG\in \bfM(N\LV,\Om^{j-2},\Lam)\wh\ot\cO\powerseries{Z}=\bfM(N\LV,\Om^{j-2},\Lam)\wh\ot_\Lam(\Lam\wh\ot\Lam)\]
by \cite[Lemma 1 in page 328]{Hida93Blue}.
\end{proof}
In view of the above lemma, we can apply the ordinary projector $e\ot 1$ to $\bdsG_\brch^{[a]}$ and obtain an $\Lam$-adic ordinary modular form $e\bdsG_\brch^{[a]}:=(e\ot 1)\bdsG_\brch^{[a]}$ with coefficients in $\Lam\wh\ot\Lam$.  \begin{lm}We have $e\bdsG_\brch^{[a]}\in \eord\bfS(N,\Om^{j-2},\Lam)\wh\ot(\Lam\wh\ot\Lam).$
\end{lm}
\begin{proof}Notation is as the above \lmref{L:1.6}. For $(Q,\zeta)\in\frakX_\bfI^{++}\times \mu_{p^\infty}(\C)$ as above, let $\Mu=\Mu_
{Q,\zeta}$ and $\Nu=\Nu_{Q,\zeta}$. Then $\bdsG(Q,\zeta-1)$ is the diagonal restriction of the holomorphic Eisenstein series $E^+_{k_Q/2}(\Mu,\Nu)$. The adelic lift of $E^+_{k_Q/2}(\Mu,\Nu)$ is given by $E_\A(g,f_{\cD,s})|_{s=\frac{k_Q/2-1}{2}}$ with $\cD=(\Mu,\Nu,k_Q/2,\frakN,\frakc)$. By \eqref{E:FE.1}, the constant term function of $E_\A(g,f_{\cD,s})$ is given by $f_{\Mu,\Nu,\Phi_\cD,s}+f_{\Mu,\Nu,\wh\Phi_\cD,s}$, and  by \eqref{E:constant.1} its values at $g\in\GL_2(\A_F)$  all vanish whenever $g_p$ is upper triangular. The lemma now follows from \cite[Lemma 6.7]{HY2019}.
\end{proof}
\begin{defn}\label{D:padicL.5}Let $\bdsf\in  \eord\bfS(N,\Om^{j-2},{\bf I})$ be a primitive Hida family. The $p$-adic twisted triple product $L$-series $\cL_{\bdsE_\brch^{[a]},\bdsf}$ is defined by
\begin{align*} 
\cL_{\bdsE_\brch^{[a]},\bdsf}:=\text{ the first Fourier coefficient of } 1_{\bdsf}(\eord\bdsG_\brch^{[a]})\in (\bfI\wh\ot\Lam)\ot_{\bfI}\Frac\bfI.
\end{align*}
By \remref{R:con.5}, $\cL_{\bdsE_\brch^{[a]},\bdsf}(Q,P)$ is finite at every arithmetic point $Q\in\frakX_\bfI^+$ and $P\in  \Spec\Lam(\Cp)$.
\end{defn}
\begin{Remark}If we replace the Eisenstein series $\bdsE_\brch^{[a]}$ by a Hida family of Hilbert cusp forms over $F$, then the above construction yields the twisted triple product $p$-adic $L$-functions constructed in \cite{Ishikawa} and \cite{BF20JIMJ}.
\end{Remark}
\subsection{The interpolation formula}\label{SS:InterRS}
The weight space of critical points is defined by
\[\frakX^{\crit} := \left\{ (\Qx,\Qz)  \in\frakX_{\bfI}^+\times\frakX_{\Lambda}^+\mid k_\Qx\geq k_\Qz,\,k_\Qx\con k_\Qz\con 0\pmod{2}\right\}. \]
The purpose of this subsection is to give the precise formula of $\cL_{\bdsE_\brch^{[a]},\bdsf}(\Qx,\Qz)$. We begin with some notation. For an arithmetic point $Q$, denote by $\bdsf_Q^\circ$ the normalized newform of weight $k_Q$ and conductor $N_Q=N p^{n_Q}$ corresponding to $\bdsf_Q$. Let $\norm{\bdsf_Q^\circ}^2_{\Gamma_0(N_Q)}$ be the usual Petersson norm of $\bdsf_Q^\circ$ and let $\cE_p(\bdsf_Q,\Ad)\in\C^\x$ be the modified $p$-Euler factor for the adjoint motive associated with $\bdsf_Q$ defined in \cite[(3.10)]{Hsieh2017}. Define the modified period 
 \[{\rm Per}^\dagger(\bdsf_Q):= (-2\sqrt{-1})^{k_Q+1}\norm{\bdsf_Q^\circ}^2_{\Gamma_0(N_Q)}\cdot \cE_p(\bdsf_Q,\Ad)\in\C^\x. \]
Let $\varrho_{\bdsf_\Qx,p}:\Qp^\x\to\C^\x$ be the unique unramified character with 
\beq\label{E:unr.6}\varrho_{\bdsf_\Qx,p}(p)=\bfa(p,\bdsf_\Qx)p^{\frac{1-k_\Qx}{2}}.\eeq
\begin{defn}[The test vector]\label{D:test.6}
Let $\eord\bfS(N\LV,\Om^{j-2},\bfI)[\bdsf]$ be the subspace of $\eord\bfS(N\LV,\Om^{j-2},\bfI)$ consisting of ordinary $\bfI$-adic forms $\bdsh$ such that $t\bdsh=\lam_\bdsf(t)\bdsh$ for all $t\in \bfT(N\LV,N)$. For each prime factor $q$ of $\LV$, let $\stt{\al_q(\bdsf),\beta_q(\bdsf)}$ be the roots of the $q$-th Hecke polynomial 
\[H_q(x,\bdsf):=x^2-\bfa(q,\bdsf)x+q^{-1}\Om^j(q)\Dmd{q}_X. \]
We fix a choice of roots $\stt{\al_q(\bdsf)}_{q\divides\LV}$. 
Enlarging the coefficient ring $\cO$ if necessary, we can assume $\al_q(\bdsf)\in\bfI$. Let $\breve\bdsf$ be the unique Hida family in $\eord\bfS(N\LV,\Om^{j-2},\bfI)[\bdsf]$ such that $\bfa(1,\breve\bdsf)=1$ and $\bfU_q\breve\bdsf=\al_q(\bdsf)\breve\bdsf$ for $q\divides \LV$. 
%\[\breve\bdsf=\prod_{q\divides\LV}(1-\al_q(\bdsf)V_q)\bdsf\]. 
\end{defn}
The following interpolation formula asserts that $\cL_{\bdsE_\brch^{[a]},\bdsf}$ essentially interpolates the values of the toric period integral $B^{\chi_Q}_{\breve\bdsf_Q}(\cmpt^{(\LV p^n)})$ defined in \eqref{E:toric} at the special element $\cmpt^{(\LV p^n)}$ in \defref{D:cmpt}.
\begin{prop}\label{P:interp1}
Let $\frakf(X,T):=\Dmd{\Delta_F \LV}^{-\onehalf}_X\Dmd{\Delta_F}^\onehalf_T\in (\Lam\wh\ot\Lam)^\x$. For every $(\Qx,\Qz)\in\frakX^\crit$, we have
\begin{align*}
\cL_{\bdsE_\brch^{[a]},\bdsf}(\Qx,\Qz)=&\frac{(-2)(-\LV\delta\sqrt{-1})^\frac{k_\Qx}{2}L(1,\tau_F)}{\prod_{q\divides\LV}\zeta_q(1)}\cdot B_{\breve\bdsf_\Qx}^{\chi_Q}(\cmpt^{(\LV p^n)})\\
&\times (-\sqrt{-1})^{k_\Qz-1}\frac{ L^{\stt{p}}\bigl(k_\Qz-\frac{k_\Qx+1}{2},\pi_{\bdsf_Q}\ot\Om^{a-k_\Qz}\ep_{\Qz}\bigl)}{{\rm Per}^\dagger(\bdsf_Q)}\\
&\times\gamma\biggl(k_\Qz-\frac{k_\Qx+1}{2},\varrho_{\bdsf_Q,p}\Om_p^{a-k_\Qz}\ep_{\Qz,p}\biggl)^{-1}\frac{\zeta_p(1)\cdot\frakf(\Qx,\Qz)c_1}{\zeta_{F_p}(2)\rho_p(p^n)\abs{p^n}_{\Qp}^\frac{1}{2}},\end{align*}
where $\chi_\Qx:=\brch\cdot \ep_\Qx^{-\onehalf}\Om^{\frac{k_\Qx-2}{2}}\circ\rmN_{F/\Q}$ and $c_1$ is the constant
\[c_1=4(-1)^\frac{a-j}{2}\Om_p^{-\frac{j}{2}}(\LV)\Om^\frac{a-j}{2}_p(\Delta_F)\prod_{\frakq\divides\frakc}\varepsilon(0,\brch_{\frakq})\in\Zbar_{(p)}^\x.\]
Here $\gamma(s,\mu)$ is the gamma factor of the character $\mu=\varrho_{\bdsf_Q,p}\Om_p^{a-k_\Qz}\ep_{\Qz,p}$. % (See \subsecref{SS:eps}).
\end{prop}
\begin{proof}
We first note that since the specialization $\bdsf_\Qx$ at $Q$ is a $p$-stabilized newform of tame conductor $N$, by the multiplicity one for new and ordinary vectors, we have
\beq\label{E:52}1_{ \bdsf_\Qx} {\rm Tr}_{Nc/N}(\eord( \bdsG_\brch^{[a]}(\Qx,\Qz))) = \cL_{\bdsE_\brch^{[a]},\bdsf}(\Qx,\Qz)\cdot \bdsf_{\Qx}.\eeq
%Let $\bdsE=\bdsE(\Om_F^\frac{a-j}{2},\Om_F^{-\frac{a}{2}}\brch)((1+X)^{1/2}-1,(1+T)^{1/2}-1)$. 
We put
\[\om^\onehalf=\ep_\Qx^{-\onehalf}\Om^\frac{k_\Qx-j}{2}\text{ and }\lambda=\ep_\Qz^\onehalf\Om^\frac{a-k_\Qz}{2}.\]
Put $k_1=\frac{k_\Qx}{2}$ and $k_2=\frac{k_\Qz}{2}$. By \propref{P:interpolateE}, we have
\[\bdsE_\brch^{[a]}(\Qx,\Qz)=\begin{cases}
\theta^{k_1-k_2}E^+_{2k_2-k_1}(\om_F^\onehalf\lam_F,\lam_F^{-1}\brch)&\text{ if }2k_2>k_1,\\
\theta^{k_2-k_1-1} E^-_{k_1-2k_2+2}(\om_F^\onehalf\lam_F,\lam_F^{-1}\brch) &\text{ if }2k_2\leq k_1.
\end{cases}\]
Applying the argument in the proof of \cite[Lemma 6.5(iv)]{Hida88Fourier}, it is not difficult to see that for a Hilbert modular form $h$ over $F$ of weight $(k_1,k_2)$ and non-negative integers $a,b$, 
\[\eord \Hol\left((\delta_{k_1,\sg_1}^a\delta_{k_2,\sg_2}^b h)|_{\frakH})\right)=\eord \left((\theta_{\sg_1}^a\theta_{\sg_2}^b h)|_{\frakH}\right),\]
where $\delta_{k_1,\sg_1}^a\delta_{k_2,\sg_2}^b$ is the Maass-Shimura differential operator and ${\rm Hol}$ is the holomorphic projection as in \cite[(8a), page 314]{Hida93Blue}. 
It follows that
\begin{align}
\eord\bdsG_\brch^{[a]}(\Qx,\Qz)=\eord(\bdsE_\brch^{[a]} (\Qx,\Qz)|_{\frakH}) =\eord{\rm Hol}(E^\dagger|_{\frakH}), 
\end{align}
where
\beq\label{E:53}E^\dagger:=\begin{cases}\delta_{2k_2-k_1}^{k_1-k_2}E^+_{2k_2-k_1}(\om_F^\onehalf\lam_F,\lam_F^{-1}\brch) & \text{ if }2k_2>k_1\\
\delta^{k_2-k_1-1}_{k_1-2k_2+2} E^-_{k_1-2k_2+2}(\om_F^\onehalf\lam_F,\lam_F^{-1}\brch) &\text{ if }2k_2\leq k_1.
 \end{cases}
\eeq
where $\delta_{k}^{m}=\delta_{k,\sg_1}^m\delta_{k,\sg_2}^m$. Let $f:= \bdsf_{\Qx} \in \mathcal{S}_{k_\Qx}(Np^r,\epsilon_{\Qx}\Om^{j-k_\Qx})$ and let 
$$\varphi_{f} = \mathit{\Phi}(\bdsf_{\Qx}) \in \mathcal{A}_{k_\Qx}^{0}(Np^r,\omega),\quad \omega=\ep_\Qx^{-1}\Om^{k_\Qx-j}.$$
Let $n$ be a sufficiently large positive integer. Let $\mathcal{J}_\infty$ and $t_n\in\GL_2(\A)$ be the matrices introduced in \eqref{E:1.4}. Let $[-,-]:\cA^0_{k_\Qx}(Np^n,\om)\times \cA_{k_\Qx}(Np^n,\om)\to\C$ be the pairing defined by 
\[\bigl[\varphi_1,\varphi_2\bigl]:=\pair{\rho(\cJ_\infty t_n)\varphi_1\ot\om^{-1}}{\varphi_2},\]
where $\pairing$ is the pairing defined in \S \ref{SS:auto}. Pairing with the form $\varphi_{f}\otimes$ on the adelic lifts on both sides of \eqref{E:52}, we obtain that
\[\cL_{\bdsE_\brch^{[a]},\bdsf}(\Qx,\Qz)\cdot \bigl[\varphi_{f},\varphi_f\bigl]=\bigl[\varphi_{f},1_{\bdsf_{\Qx}}{\rm Tr}_{CN/N}\eord\mathit{\Phi}({\rm Hol}(E^\dagger|_{\frakH}))\bigl], \]
where $1_{\bdsf_Q}\in (\bfT(N,\bfI)/\wp_Q)\ot\C\subset \End\eord\cS_{k_Q}(Np^n,\om^{-1})$ is the specialization of $1_{\bdsf}$ at $Q$. Since the Hecke operators $\stt{T_q}_{q\ndivides Np}$ and $\bfU_p$, the holomorphic projection ${\rm Hol}$ and the trace map $\Tr_{\LV N/N}$ are self-adjoint operators with respect to the pairing $[-,-]$ (\cf the proof of \cite[Proposition 3.7]{Hsieh2017}),
%Note that $E^\dagger|_{\frakH}$ is a nearly holomorphic modular form of weight $k$ and its adelic $\mathit{\Phi}(E^\dagger|_{\frakH}) \in \mathcal{A}_{k_\Qx}(Np^r,\omega)$ has a decomposition
%$$\mathit{\Phi}(H) = {\rm Hol}(\mathit{\Phi}(H))+V_+\varphi_1'+V_+^2\varphi_2'+\cdots+V_+^n\varphi_n',$$
%where ${\rm Hol}(\mathit{\Phi}(H))$ and $\{\varphi_j\}_{j=1,\cdots,n}$ are holomorphic automorphic forms. It follows that ${\rm Hol}(\mathit{\Phi}(H)) = \mathit{\Phi}({\rm Hol}(H)).$ Let $1^*_{f}\in \mathbb{T}^{\rm ord}(Np^r,\Om^{2j-k_\Qx})$ be the specializations of $1^*_{\bdsf}$ at $\Qx$. As a consequence of strong multiplicity one theorem for modular forms, the idempotent $1_{f}=\eta_{f}^{-1}1^*_{f}\in \mathbb{T}^{\rm ord}(N p^r,\Om^{2j})\otimes_\mathcal{O}{\rm Frac}\mathcal{O}(\Qx)$ is generated by the Hecke operators $T_\ell$ for $\ell\nmid Np$, and this implies that hat $1_f$ is the left adjoint operator of $1_{\bdsf_\Qx}$ for the pairing $\pair{-\ot\om^{-1}}{-}$. the right hand side of (\ref{E:4.15}) equals
we thus obtain
\[\cL_{\bdsE_\brch^{[a]},\bdsf}(\Qx,\Qz)\cdot \bigl[\varphi_{f},\varphi_f\bigl]=[U_0(\LV N):U_0(N)]\cdot \bigl [\varphi_f,\varPhi(E^\dagger|\frakH)\bigl].%[\Gamma_0(\LV N):\Gamma_0(N)]\cdot \bigl[\varphi_f,E_\A(-,f_{\cD,s})\bigl]|_{s=\frac{2k_2-k_1-1}{2}}. 
\]
On the other hand, according to \eqref{E:53} and \propref{P:shift.1}, we have
\[\varPhi(E^\dagger|_\frakH)=E_\A(g,f_{\cD,s})|_{s=\frac{2k_2-k_1-1}{2}},\quad g\in\GL_2(\A),\]
where $\cD$ is the Eisenstein datum
\[\cD=(\om_F^\onehalf\lambda_F,\,\brch^{-1}\lambda_F^{-1},\,\frac{k_\Qx}{2},\,\frakc,\,\frakN). \]
Therefore we see that
\begin{align*}
&\cL^{[\bdsf]}_{\bdsE}(\Qx,\Qz)\cdot \bigl[\varphi_{f},\varphi_f\bigl]\\
=&[\Gamma_0(\LV N):\Gamma_0(N)]\cdot \<\rho(\mathcal{J}_\infty t_n)\varphi_f,E_\A(-,f_{\cD,s})\ot\omega^{-1}\>|_{s=\frac{2k_2-k_1-1}{2}}\\
=&[\Gamma_0(\LV N):\Gamma_0(N)]\cdot Z_\cD(s,\rho(\cJ_\infty t_n)\varphi_f)|_{s=\frac{2k_2-k_1-1}{2}}.
\end{align*}
By \cite[Lemma 3.6]{Hsieh2017}, we have
\begin{multline*}\bigl[\varphi_f,\varphi_f\bigl]=\<\rho(\mathcal{J}_\infty t_n)\varphi_{f}\otimes\omega^{-1},\varphi_f\>\\
=\frac{\zeta_\Q(2)^{-1}}{[\SL_2(\Z):\Gamma_0(N)]}\cdot(-2\sqrt{-1})^{-k_\Qx-1}\cdot {\rm Per}^\dagger(f)\cdot \frac{\om^{-1}_p\al^2_f(p^n)\abs{p^n}_{\Qp}\zeta_p(2)}{\zeta_p(1)}. 
\end{multline*}
Then we have the interpolation formula
\begin{multline*}
\cL_{\bdsE_\brch^{[a]},\bdsf}(\Qx,\Qz)
=Z_\cD(s,\rho(\cJ_\infty t_n)\varphi_f)|_{s=\frac{k_\Qz-k_\Qx/2-1}{2}}\\
\times \frac{\zeta_\Q(2)[\SL_2(\Z):\Gamma_0(\LV N)](-2\sqrt{-1})^{k_\Qx+1}}{{\rm Per}^\dagger(\bdsf_Q)}\cdot \frac{\zeta_p(1)}{\omega_p^{-1}\varrho_f^2(p^n)\abs{p^n}_{\Qp}\zeta_p(2)}
\end{multline*}
for any sufficiently large positive $n$. From the above equation and the formula of $Z_\cD(s,\rho(\cJ_\infty t_n)\varphi_f)$ in \thmref{T:formula} with the fudge factors given by 
\begin{align*}
\frakf_\infty&=4(-4\sqrt{-1})^{-k_\Qx/2}(-1)^\frac{-j+a-k_\Qz}{2},\\
\frakf_{\rm ram}&=\om_p^\onehalf\lam_p(\Delta_F)\Delta_F^\frac{k_\Qz-k_\Qx}{2}\delta^{\frac{k_\Qx}{2}}=\Dmd{\Delta_F}^{-\onehalf}_X(\Qx)\Dmd{\Delta_F}^\onehalf(\Qz)\cdot \Om_p^\frac{a-j}{2}(\Delta_F)\cdot \delta^\frac{k_\Qx}{2},\\
\frakf_\LV&=\om_p^\onehalf(\LV)\prod_{\frakq\divides\frakc}\frac{\varepsilon(0,\brch_{\frakq})}{\zeta_q(1)}=\Dmd{\LV}^{-\onehalf}_X(\Qx)\cdot \Om_p^{-\frac{j}{2}}(\LV)\prod_{\frakq\divides\frakc}\frac{\varepsilon(0,\brch_{\frakq})}{\zeta_q(1)}\cdot \LV^\frac{k_\Qx}{2},\end{align*}
we get the desired interpolation formula by noting that
\[\frakf_\infty\frakf_{\ram}\frakf_\LV
=\frac{\frakf(\Qx,\Qz)c_1}{\prod_{q\divides \LV}\zeta_q(1)}\cdot (-\sqrt{-1}\LV\delta)^\frac{k_\Qx}{2}\cdot (-2\sqrt{-1})^{-k_\Qx-1}(\sqrt{-1})^{-k_\Qz}.\]
%\begin{align*}
%\frakf_c=&\om_F^\onehalf(c)\epsilon(1/2,\chi_{0,\frakq})=\Dmd{c}^{-\onehalf}_Xc^{\frac{k_\Qx}{2}}\Om_p^{j}(c)\epsilon(1/2,\chi_{0,\frakq}).
%\end{align*}
\end{proof}

%\begin{lm}[Hida]\label{L:Hida} Let $f$ be a Hilbert modular form over a real quadratic field $F$. Then we have
%\[\eord \Hol\left((\delta_{\sg_1}^a\delta_{\sg_2}^b f)|_{\frakH})\right)=\eord \left((\theta_{\sg_1}^a\theta_{\sg_2}^bf)|_{\frakH}\right)\]
%\end{lm}
%\begin{proof}By \cite[Eq.(3), p.311]{Hida93Blue}, we have
%\[\delta^r_{k,\sg}=\theta^r_\sg+\sum_{i=1}^r a_i(-4\pi y_\sg)^i\theta_\sg^{r-i},\quad a_i\in\Q_+.\]
%Thus
%\[\delta^a_{k,\sg}\delta^b_{k,\sg_2}=\theta_1^a\theta_2^b+\sum_{i+j\neq 0} a_ia_j(-4\pi y_1)^{i}(-4\pi y_2)^{j}\theta_1^{r-i}\theta_2^{r-j}\]
%\end{proof}

\section{$p$-adic $L$-functions attached to modular forms and real quadratic fields}
In \cite{BD09Ann}, the authors construct square root $p$-adic $L$-functions associated with Hida families and real quadratic fields, interpolating the toric period integrals of elliptic cusp forms over real quadratic fields. The purpose of this section is to give a mild improvement of this construction and more general interpolation formulae. \subsection{Preliminaries on modular symbols}\label{SS:MS}
We review the theory of classical modular symbols in the \emph{semi-adelic} language. Let $\bfP:=\bfP^1(\Q)$ and $\frakD_0:=\Div^0\bfP\times \GL_2(\wh \Q)$. 
For each $r\in\bfP$ we denote by $\stt{r}$ its image in the divisor group of $\bfP$. Let $\gamma\in\GL_2(\Q)$ and $u\in\GL_2(\wh\Q)$ act on $D=(\MS{r}{s},g_\rmf)\in \frakD_0$ by 
\[\gamma D u:=(\MS{\gamma\cdot r}{\gamma\cdot s},\gamma g_\rmf u).\]
%Let $M$ be a $\GL_2(\Q)$-module. A classical $M$-valued modular symbol of level $N$ is a function $\xi: \frakD_0\to M$ such that 
%$\xi(\gamma  x u)=\gamma\cdot \xi(x)$ for all $\gamma\in \GL_2(\Q)$ and $u\in U_1(N)$. 
For a ring $R$, let $L_n(R)$ be the space of two-variable homogeneous polynomials of degree $n$ with coefficients in $R$. For $P=P(X,Y)\in L_n(R)$ and $g\in \GL_2(R)$, define 
\[P\Big|\pMX{a}{b}{c}{d}(X,Y)=P(aX+bY,cX+dY).\]
Let $L^*_n(R)=\Hom_R(L_n(R),R)$.
Moreover, if $R$ is a $\Zp$-algebra, let $\GL_2(\wh\Z)$ acts on $L^*_n(R)$ by $(\rho_n(u)\xi)(P)=\xi(P|u_p)$. For an integer $N$ and a Hecke character $\chi$ modulo $N$ valued in $R$, we denote by $\cM\cS_k(N,\chi,R)$ the space of $p$-adic modular symbols of weight $k$, level $N$ and character $\chi$, consisting of maps $\xi: \frakD_0\to L^*_{k-2}(R)$ such that \[\xi(\gamma D u)=\chi^{-1}(u)\cdot \rho_{k-2}(u_p^{-1})\xi(D)\]
for $\gamma\in\GL_2^+(\Q)$ and $u\in U_1(N)$.
This space is known to be a finitely generated $R$-module equipped with the Hecke action. The Hecke operators $T_\pme$ for $\pme\ndivides Np$ act on $\cM\cS_k(N,\chi,R)$ by the formula
\beq\label{E:Heckeop1}T_\pme \xi(D)=\xi\biggl(D\pDII{1}{\pme}\biggl)+\sum_{b\in \Z_\pme/\pme\Z_\pme}\xi\biggl(D\pMX{\pme}{b}{0}{1}\biggl).\eeq
Define the operator $\bfU_\pme$ for $\pme\divides N$ and $\pme\neq p$ by the formula
\begin{align}\label{E:Heckeop2}
\bfU_\pme \xi(D)=&\sum_{b\in \Z_\pme/\pme\Z_\pme}\xi\biggl(D\pMX{\pme}{b}{0}{1}\biggl)\text{ for }\pme \divides N.%\,\pme\not =p,\quad\bfU_p\Xi(D)=\sum_{b\in\Zp/p\Zp}\Xi(D\pMX{p}{b}{0}{1})
\end{align}
and the operator $\bfU_p$ by 
\[\bfU_p\xi(D)=\sum_{a\in \Zp/p\Zp}\rho_{k-2}\biggl(\pMX{p}{a}{0}{1}\biggl)\xi\biggl(D\pMX{p}{a}{0}{1}\biggl).\]
The ordinary projector $\eord:=\displaystyle\lim_{n\to\infty}\bfU_p^{n!}$ is a convergent operator on $\cM\cS_k(N,\chi,R)$. Choosing any element $\gamma\in\GL_2(\Q)$ with $\det\gamma<0$, we define an involution $[\bfc]$ on $\xi\in\cM\cS_k(N,\chi,A)$ by 
\[[\bfc]\xi(D):=\xi(\gamma\cdot D).\]
This definition does not depend on the choice of such $\gamma$. We define
\[\xi^+:=\biggl(\frac{1+[\bfc]}{2}\biggl)\xi;\quad \xi^-:=\biggl(\frac{1-[\bfc]}{2}\biggl)\xi. \]
\subsection{Modular symbols associated with modular forms}\label{SS:62}
To each classical cusp form $f=f(z,g_\rmf)\in \cS_k(N,\chi)$, we associate a classical modular symbol $\eta_f:\frakD_0\to L^*_{k-2}(\C)$ defined by 
\[\eta_f(\MS{r}{s},g_\rmf)(P):=\int_r^s f(z,g_\rmf)P(z,1)\rmd z.\]
It is easy to see that for $\al\in\GL_2^+(\Q)$ and $u\in U_0(N)$, 
\[\eta_f(\al D u)=\rho_{k-2}(\al)\eta_f(D)\chi^{-1}(u).\]
The involution $[\bfc]$ acts on the classical modular symbol $\eta_f$ by $[\bfc]\eta_f(D)=\rho_{k-2}(\gamma)\eta_f(\gamma D)$, where $\gamma\in\GL_2(\Q)$ is any element with $\det\gamma<0$. By definition, 
\[[\bfc]\eta_f(D)=-\ol{\eta_{f_\rho}(D)},\]
where $f_\rho(z,g_\rmf)=\ol{f\biggl(-\ol{z},\pDII{-1}{1}g_\rmf\biggl)}$. On the other hand, the associated $p$-adic modular symbol $\xi_f\in \cM\cS_k(N,\chi,\Cp)$ is defined by 
\beq\label{E:pms.6}\xi_f(D)(P)=\iota_p(\eta_f(D)(P|g_p^{-1}))\text{ for }D=(d,g_\rmf)\in\frakD_0.\eeq
If $f$ is a $\bfU_p$-eigenform with eigenvalue $\al\in\Zp^\x$, then $\xi_f$ is also an eigenvector of $\bfU_p$ with eigenvalue $\al$. Following the discussion in \cite[p.95]{Kitagawa94}, for each $D\in\frakD_0$ we define the \padic measure $\mu_f(D)(x)$ on $\Zp$ by the rule
\[\int_{a+p^n\Zp}\mu_f(D)(x)=\al^{-n}\xi_f\biggl(D\pMX{p^n}{a}{0}{1}\biggl)(Y^{k-2})\text{ for }n\in\Z^{\geq 0}. \]
\begin{lm}For any $P\in L_{k-2}(\Zbar_p)$, 
\[\int_{a+p^n\Zp}P(x,1)\mu_f(D)(x)=\al^{-n}\xi_f\biggl(D\pMX{p^n}{a}{0}{1}\biggl)\biggl(P\biggl|\pMX{p^n}{a}{0}{1}\biggl).\]
\end{lm}
\begin{proof}This is \cite[Lemma 4.6]{Kitagawa94}. We paraphrase the computation there in our semi-adelic formulation. Note that $\xi_f$ has bounded denominators in the sense that $p^A\cdot \xi_f\in \cM\cS_k(N,\chi,\Zbar_p)$ for some $A\gg 0$.
Let $0\leq j\leq k-2$ be an integer. For every $m>A+n$, we have
\begin{align*}
&\al^{-n}\xi_f\biggl(D\pMX{p^n}{a}{0}{1}\biggl)\biggl(X^jY^{k-2-j}\biggl|\pMX{p^n}{a}{0}{1}\biggl)\\
=&\al^{-m}\sum_{c=0}^{p^{m-n-1}}\xi_f\biggl(D\pMX{p^m}{a+p^nc}{0}{1}\biggl)\biggl(X^jY^{k-2-j}\biggl|\pMX{p^m}{a+p^nc}{0}{1}\biggl)\\
\con&\al^{-m}\sum_c(a+p^nc)^j\xi_f\biggl(D\pMX{p^m}{a+p^nc}{0}{1}\biggl)(Y^{k-2})\pmod{p^{m-A}\Zbar_p}.
\end{align*}
Therefore, we find that
\begin{align*}
&\int_{a+p^n\Zp}x^j \mu_f(D)(x)\\
=&\sum_{m\to\infty}\al^{-m}\sum_{c=0}^{p^{m-n-1}}(a+p^nc)^j\xi_f\biggl(D\pMX{p^m}{a+p^nc}{0}{1}\biggl)(Y^{k-2})\\
=&\al^{-n}\xi_f\biggl(D\pMX{p^n}{a}{0}{1}\biggl)\biggl(X^jY^{k-2-j}\biggl|\pMX{p^n}{a}{0}{1}\biggl).\end{align*}
This shows the lemma.
\end{proof}
\subsection{Hida theory for modular symbols}
We review the $\bfI$-adic symbols developed in \cite{Kitagawa94} in the semi-adelic formulation. Let $\bfI$ be a normal and finite domain over $\Lam=\cO\powerseries{X}$ with $X=[\bfu]-1$ and let $N$ be a positive integer coprime to $p$. Put
\[U_{1}(Np^\infty)=\stt{u\in U_1(N)\;\biggl|\; u_p=\pMX{a}{b}{0}{1},\,a\in\Zp^\x,\; b\in\Zp}.\]
For each non-negative integer $n$, let $\wp^{(n)}$ be the principal ideal of $\bfI$ generated by $(\bfu^{-2}(1+X)-1)^{p^n}-1$. Define the $\Lam$-adic Hecke character $\bfal_X:\Q^\x\bksl \A^\x\to \Lam^\x$ by \[\bfal_X(z)=\Dmd{\cyc(z)}_X\Dmd{\cyc(z)}^{-2},\]
where $\Dmd{\cyc(z)}_X=(1+X)^{\frac{\log_p\cyc(z)}{\log_p\bfu}}$ is defined in \eqref{E:defn1}. 
\begin{defn}Define the space of $\bfI$-adic modular symbols of tame level $N$ by \[\cM\cS(N,\bfI):=\prolim_{n}\dirlim_{m} \cM\cS_2(Np^m,\Abs_X,\bfI/\wp^{(n)}).\] In other words, $\cM\cS(N,\bfI)$ consists of functions $\Xi:\frakD_0\to \bfI$ such that 
\begin{itemize}
\item $\Xi(\gamma D u)=\Xi(D)$ for $\gamma\in\GL_2^+(\Q)$ and $u\in U_1(Np^\infty)$;
\item $\Xi(D z)=\bfal_X(z^{-1})\cdot\Xi(D)$ for $z\in\wh\Q^\x$;
\item $\Xi$ is continuous in the sense that for any $n$, there exists $r_n$ for which the function $\Xi:\frakD_0\to \bfI/\wp^{(n)}$ factors through $\frakD_0/U_1(Np^{r_n})$.
\end{itemize}
\end{defn}
The space $\cM\cS(N,\bfI)$ is an $\bfI$-module equipped with the action of Hecke operators $\stt{T_q}_{q\ndivides Np}$ and $\stt{\bfU_q}_{q\divides N}$ as in \eqref{E:Heckeop1} and \eqref{E:Heckeop2}, while the $\bfU_p$-operator is defined by \[\bfU_p\Xi(D)=\sum_{a\in\Zp/p\Zp} \Xi\biggl(D\pMX{p}{a}{0}{1}\biggl).\] 
For $(d,pN)=1$ we define the level-raising operator 
\[V_d:\cM\cS(N,\bfI)\to \cM\cS(Nd,\bfI)\] 
by \beq\label{E:V.6}V_d\Xi(D)=d^{-1}\cdot\Xi\biggl(D\pDII{d^{-1}}{1}\biggl).\eeq

The ordinary projector $\eord=\displaystyle\lim_{n\to\infty}\bfU_p^{n!}$ exists in $\End_{\bfI}\cM\cS(N,\bfI)$. The space $\eord\cM\cS(N,\bfI)$ consists of the ordinary $\bfI$-adic modular symbols. We remark that $\eord\cM\cS(N,\bfI)$ is nothing but $MS^{ord}(\bfI)=\Hom_{\Lam}(UM^{ord}(\cO),\bfI)$ defined in \cite[\S 5.5]{Kitagawa94}. 
The involution $[\bdc]$ on $\cM\cS(N,\bfI)$ is defined by $[\bfc]\Xi(D):=\Xi(\gamma D)$ for any $\gamma\in\GL_2(\Q)$ with $\det\gamma<0$. Put $$\eord\cM\cS(N,\bfI)^\pm:=(1\pm[\bfc])\eord\cM\cS(N,\bfI).$$

The following is proved in \cite[Proposition 5.7]{Kitagawa94}. 
\begin{thm}The space $\eord\cM\cS(N,\bfI)$ is free of finite rank over $\bfI$. \end{thm}
 We recall the $\bfI$-adic measure associated with ordinary $\bfI$-adic modular symbols. Let $\cC(\Zp,\bfI)$ be the space of continuous $\bfI$-valued functions on $\Zp$ and $\cD(\Zp,\bfI):=\Hom_{\bfI}(\cC(\Zp,\bfI),\bfI)$ be the space of $\bfI$-adic measures on $\Zp$. 
To each ordinary $\bfI$-adic modular symbol $\Xi\in\eord\cM\cS(N,\bfI)$, we associate a unique linear map $D\mapsto \mu_\Xi(D)(x)$ in $\Hom(\frakD_0,\cD(\Zp,\bfI))$ such that 
\beq\label{E:mea.6}\int_{\Zp}P(x)\mu_\Xi(D)(x):=\lim_{m\to\infty}\sum_{a=0}^{p^m-1}P(a)\bfU_p^{-m}\Xi\biggl(D\pMX{p^m}{a}{0}{1}\biggl)\in\bfI
\eeq
for $D\in\frakD_0$ and $P\in \cC(\Zp,\bfI)$. 
It is straightforward to verify that the right hand side is a $p$-adically convergent Riemann sum valued in $\bfI$. For $P\in\cC(\Zp,\bfI)$ and $u\in U_0(p)$ with $u_p=\pMX{a}{b}{c}{d}$, define 
\[P\,|\,u(x)=P\left(\frac{ax+b}{cx+d}\right)\bfal_X(cx+d). \]
\begin{lm}\label{L:equiv.6}
Let $P\in\cC(\Zp,\bfI)$. 
\begin{enumerate} 
\item For $m\in\Z^{\geq 0}$, \[\int_{p^m\Zp}P(x)\mu_\Xi(D)(x)=\int_{\Zp}P(p^mx)\mu_{\bfU_p^{-m}\Xi}\biggl(D\pDII{p^m}{1}\biggl)(x).\]
\item For $u\in U_0(pN)$, we have
\[\int_{\Zp}P(x)\mu_\Xi(Du)(x)=\int_{\Zp}P\,|\,u^{-1}(x)\mu_\Xi(D)(x).\]
\end{enumerate}
\end{lm}

\begin{proof}The verification of part (1) is straightforward by \eqref{E:mea.6}. 
To see Part (2), it suffices to show the equation for $u_p$ of the form $\pMX{1}{0}{c}{1}$ and $\pMX{a}{b}{0}{d}$. Let $u_p=\pMX{1}{0}{c}{1}$ with $c\in p\Zp$. By definition, the left hand side equals
\begin{align*}
&\lim_{m\to\infty}\sum_{a=0}^{p^m-1}P(a)\bfU_p^{-m}\Xi\biggl(D\pMX{1}{0
}{c}{1}\pMX{p^m}{a}{0}{1}\biggl)\\
=&\lim_{m\to\infty}\sum_{a=0}^{p^m-1}P(a)\bfU_p^{-m}\Xi\biggl(D\pMX{p^m}{a(1+ac)^{-1}}{0}{1}\pMX{(1+ac)^{-1}}{0}{cp^m}{1+ac}\biggl).\end{align*}
Making change of variable $a=z(1-cz)^{-1}$, we find that the last Riemann sum equals
\begin{align*}&\lim_{m\to\infty}\sum_{z=0}^{p^m-1}P(z(1-cz)^{-1})\bfU_p^{-m}\Xi\biggl(D\pMX{p^m}{z}{0}{1}\biggl)\bfal_X(1-cz)^{-1}\\
&=\int_{\Zp}P\,\biggl|\,\pMX{1}{0}{-c}{1}(x) \mu_\Xi(D)(x).
\end{align*}
The case for $u_p=\pMX{a}{b}{0}{d}$ is similar. We omit the details.
\end{proof}
For an arithmetic point $Q$ in $\frakX^+_{\bfI}$, we denote by $\wp_Q$ the kernel of the specialization $Q:\bfI\to \Cp$. 
Put $\cO(Q)=\bfI/\wp_Q$ and $r_Q=\max\stt{1,c_p(\ep_Q)}$. 
Here $c_p(\ep_Q)$ is the exponent of the $p$-conductor of $\ep_Q$. For any $\cO(Q)$-algebra $A$, we put
\[\cM\cS_Q^\Ord(A):=\eord\cM\cS_{k_\Qx}(Np^{r_Q},\Om^{2-k_\Qx}\ep_\Qx,A).\]
The following theorem is an integral version of the control theorem for $\bfI$-adic modular symbols proved in \cite[Theorem 5.13]{GS1993}. The result must be well-known to experts, but since we could not locate an exact statement in the literature, we provide some details for the sake of completeness. 
\begin{thm}[Control Theorem]\label{T:control.6}
For each arithmetic point $Q$, there is a Hecke-equivariant specialization isomorphism 
\begin{align*}{\rm sp}_Q\colon \eord\cM\cS(N,\bfI)/\wp_Q&\iso\cM\cS^\Ord_Q(\cO(Q)),\\
\Xi\pmod{\wp_\Qx}&\mapsto {\rm sp}_Q(\Xi):=\Xi_Q,
\end{align*}
where $\Xi_Q$ is the $p$-adic modular symbol of weight $k_Q$ defined by 
\[\Xi_Q(D)(P)=Q\biggl(\int_{\Zp}P(x,1)\mu_\Xi(D)(x)\biggl),\quad P(X,Y)\in L_{k_Q-2}(\cO(Q)).\]
We call $\Xi_Q$ the specialization of $\Xi$ at $Q$.
\end{thm}
 \begin{proof}First we note that $\Xi_Q$ is a $p$-adic modular symbol of weight $k_Q$ and character $\Om^{2-k_\Qx}\ep_\Qx$ by \lmref{L:equiv.6}. It is straightforward to verify that the map ${\rm sp}_Q$ is Hecke-equivariant, so $\Xi_Q$ belongs to $\cM\cS^\Ord_Q(\cO(Q))$. We proceed to show ${\rm sp}_Q$ is an isomorphism. Let $n=k_\Qx-2$, $\cO=\cO(Q)$ and $\chi_Q=\bfal_X\pmod{\wp_Q}=\cyc^n\Om^{-n}\ep_Q$. We have 
 \[\eord\cM\cS(N,\bfI)/\wp_Q=\prolim_t\dirlim_r\eord\cM\cS_2(Np^r,\chi_Q,\cO/p^t).\]
For any $\Zp$-module $R$, define $\iota_n: L^*_n(R)\to R,\quad \iota_n(\ell)=\ell(Y^n)$. By \cite[Corollary 5.2]{Kitagawa94}, $\iota_n$ induces a Hecke-equivariant isomorphism 
\[\iota_n\colon \eord\cM\cS_{k_\Qx}(Np^r,\ep_Q\Om^{-n},\cO/p^t)\iso\eord\cM\cS_2(Np,\chi_Q,\cO/p^t)\]
for $r\geq t$.
Note that $ \iota_n(\Xi_Q(D))=\Xi_Q(D)(Y^n)=Q(\Xi(D))$. We deduce that ${\rm sp}_Q$ is indeed given by the isomorphism
\begin{align*}&\eord\cM\cS(N,\bfI)/\wp_Q=\prolim_t\dirlim_r\eord\cM\cS_2(Np^r,\chi_Q,\cO/p^t)\\&\qquad\mathop{\iso}^{\iota_n^{-1}}\prolim_t\dirlim_r\eord\cM\cS_{k_Q}(Np^r,\ep_Q\Om^{-n},\cO/p^t)=\eord\cM\cS_{k_Q}(Np^{r_Q},\ep_Q\Om^{-n},\cO),\end{align*}
where the last equality is the base change property \cite[Lemma 1.8 and Corollary 2.2]{Hida88AJM} for ordinary $p$-adic modular symbols. 
\end{proof}

\subsection{The distribution-valued modular symbols of Greenberg and Stevens}\label{R:GS.6} Let $L_0'$ be the set of primitive elements in $\Zp\times\Zp$, i.e. elements in $\Zp\times \Zp$ which are not divisible by $p$. We recall the connection of $\Lam$-adic modular symbols and modular symbols with valued in the space $\cD(L_0')$ of $p$-adic measures on $L_0'$ described in \cite[\S 5]{GS1993}. For each $k\in\Cp$ with $\abs{k}_p\leq 1$, let $Q_k\in\Spec\Lam(\Cp)$ be the unique point with $Q_k([\bfu])=\bfu^k$ and
let $\sF_k$ be the set of homogeneous functions of degree $k$ on $L_0'$, \ie continuous functions $h:L_0'\to\Zp$ such that $h(ax,ay)=\Dmd{a}^k h(x,y)$ for all $a\in\Zp^\x$. 
Then to each $\Xi\in\eord\cM\cS(N,\Lam)$, we can associate a modular symbol $\mu^{\rm GS}_\Xi\in\Hom_{U_0(N)}(\frakD_0,\cD(L_0'))$ characterized by the property that we have
\begin{align*}\int_{\Zp\times\Zp^\x}h(x,y)\mu^{\rm GS}_{\Xi}(D)(x,y)&=Q_k\biggl(\int_{\Zp}h(x,1)\mu_\Xi(D)(x)\biggl);\\
\int_{\Zp^\x\times p\Zp}h(x,y)\mu^{\rm GS}_\Xi(D)(x,y)&=Q_k\biggl(\int_{\Zp}h(1,-py)\mu_{\bfU_p^{-1}\Xi}\biggl(D\pMX{0}{1}{-p}{0}\biggl)(y)\biggl)
\end{align*}
for any $k\in\Zp$ and $h\in\sF_{k-2}$. 
By a similar computation in \lmref{L:equiv.6}, one verifies that the map $\mu_\Xi^{\rm GS}$ is $U_0(N)$-invariant, namely for any $u\in U_0(N)$ 
\beq\label{E:Uinv.6}\int_{L_0'}h\,|\, u^{-1}(x,y)\mu^{\rm GS}_\Xi(D)(x,y)=\int_{L_0'}h(x,y)\mu^{\rm GS}_\Xi(D u)(x,y).\eeq

\subsection{The Mazur-Kitagawa two variable $p$-adic $L$-functions}
Let $\bdsf\in \eord\bfS(N,1,\bfI)$ be a primitive Hida family of tame conductor $N$ and let $\lam_\bdsf:\bfT(N,\bfI)\to\bfI$ be the corresponding homomorphism. For any integer $C$ prime to $N$, let $\eord\cM\cS(NC,\bfI)^\pm[\bdsf]$ be the space of $\bfI$-adic ordinary modular symbols $\Xi\in \eord\cM\cS(NC,\bfI)^\pm$ such that 
%Let $\frakm_\bdsf$ be the maximal ideal containing $\Ker\lam_\bdsf$ and let $\bfT_{\frakm_\bdsf}$ be the localization of the $\bfI$-adic ordinary cuspidal Hecke algebras at $\frakm_\bdsf$. Under some conditions on the Gorensteiness of the Hecke ring for $\bdsf$, the space $\eord\cM\cS(N,\bfI)^\pm_{\frakm_\bdsf}$ is free of rank one over $\bfT_{\frakm_{\bdsf}}$. It follows that there exists a unique $\bfI$-adic modular symbol $\Xi^\pm_\bdsf\in \eord\cM\cS(N,\bfI)^\pm$ such that
$t\cdot \Xi=\lam_\bdsf(t)\Xi$ for all $t\in\bfT(N\LV,N).$
The space $\eord\cM\cS(N,\bfI)^\pm[\bdsf]\ot_{\bfI}\Frac\bfI$ has rank one over $\Frac\bfI$ as $\bdsf$ is primitive of tame conductor $N$. For an arithmetic point $Q$, the space $\cM\cS^\Ord_Q(\cO(Q))^\pm[\bdsf_Q]$ is free of rank one over $\cO(Q)$. On the other hand, Shimura in \cite{Shimura1977} proved that $0\neq\xi^\pm_{\bdsf_Q}\in\cM\cS^\Ord_Q(\Cp)^\pm[\bdsf_Q]$. Therefore, having fixed a basis
$\beta_{\bdsf_Q}^\pm$ of $\cM\cS^\Ord_Q(\cO(Q))$, we can define the period $\Omega^\pm_{\bdsf_Q}\in\Cp^\x$ associated with the $p$-stabilized newform $\bdsf_Q$ by 
\[\xi^\pm_{\bdsf_Q}=\Omega_{\bdsf_Q}^\pm \beta^\pm_{\bdsf_Q}.\] %Thus we can write  \[
%\Xi^\pm_Q=\mho^\pm(\Xi_Q)\cdot \xi_{\bdsf_Q}^\pm\] for some $\mho^\pm(\Xi_Q)\in\Cp$.
 \begin{defn}[$p$-adic error terms]
 Let $\Xi\in \eord\cM\cS(N,\bfI)[\bdsf]$. We define the plus/minus error terms ${\rm Er}^\pm(\Xi_Q)\in\Cp$  by the equation \[
\Xi^\pm_Q=\frac{{\rm Er}^\pm(\Xi_Q)}{\Omega^\pm_{\bdsf_Q}}\cdot \xi_{\bdsf_Q}^\pm.\]
%define $\mho^\pm(\Xi_Q)\in \Cp$ by \[
%\Xi^\pm_Q=\mho^\pm(\Xi_Q)\cdot \xi_{\bdsf_Q}^\pm.\]
\end{defn}

%Set  \[\bfV:=\rho_{\bdsf}\ot\Abs_T\chi\] to be the $\bfI\powerseries{T}$-adic Galois representation associated with the Hida family $\bdsf$ twisted by the character $\Abs_T\Om^{a}$. %we recall that 
%\[L_p(\Xi,\chi)=\sum_{[z]\in (\wh\Z/c\wh\Z)^\x}\chi(z)\int_{\Zp^\x}\bfal_T(x)\chi(x)\mu_{\Xi^{(-1)^a}}(\MS{\infty}{0},\pMX{1}{z/c}{0}{1})(x)\in \bfI\powerseries{T}.\]
To $\Xi\in \eord\cM\cS(N,\bfI)[\bdsf]$ and a finite order Hecke character $\chi$ with $\chi(-1)=(-1)^i$, Kitagawa in \cite[Theorem 1.1]{Kitagawa94} associates the two-variable $p$-adic $L$-function $L_p(\Xi,\chi)\in \bfI\wh\ot\Lam$ satisfying the interpolation property: for every pair of arithmetic points $(Q,P)\in\frakX_\bfI^+\times\frakX_\Lam^+$ with $k_Q\geq k_P$, 
\begin{align}
L_p(\Xi^{(-1)^i},\chi)(\Qx,\Qz)=&(-\sqrt{-1})^{k_\Qz-1}\cdot\frac{ L^{\stt{p}}(k_\Qz-\frac{k_\Qx+1}{2},\pi_{\bdsf_Q}\ot\chi\Om^{-k_\Qz}\ep_{\Qz})}{\Omega_{\bdsf_\Qx}^{(-1)^{i}}}\notag \\
&\times\gamma\biggl(k_\Qz-\frac{k_\Qx+1}{2},\varrho_{\bdsf_\Qx,p}\ot\chi_p\Om_p^{-k_\Qz}\ep_{\Qz,p}\biggl)^{-1} {\rm Er}^{(-1)^{i}}(\Xi_\Qx), \label{E:interp3}
\end{align}
The $L$-functions associated with modular forms are related to the automorphic $L$-functions in the following way:
\[L\biggl(k_\Qz-\frac{k_\Qx+1}{2},\pi_{\bdsf_Q}\ot\chi\biggl)=2(2\pi)^{1-k_\Qz}\Gamma(k_\Qz-1)\cdot L(k_\Qz-1,\bdsf_\Qx\ot\chi).\]

%\[L_p(\Xi,\chi)(\Qx,\Qz)=(-1)^a(-\sqrt{-1})^{k_\Qz}\cdot\frac{L(0,\bfV_{\Qx,\Qz})}{\Omega_{\bdsf_\Qx}^{(-1)^m}}\cdot\cE_p(0,\bfV_{\Qx,\Qz})\cdot {\rm Er}^{(-1)^m}(\Xi_Q),\]
%where 
%\[\cE_p(s,\bfV_{\Qx,\Qz})=\frac{L_p(s,\Fil^+\bfV_{\Qx,\Qz})}{\varepsilon_p(s,\bfV_{\Qx,\Qz})L_p(1-s,(\Fil^+\bfV_{\Qx,\Qz})^\vee)}\cdot \frac{1}{L_p(s,\bfV_{\Qx,\Qz})}\]

\subsection{The square root $p$-adic $L$-function associated with a Hida family and a real quadratic field}
We review the construction of the square roots of $p$-adic $L$-functions attached to Hida families and real quadratic fields in \cite{BD09Ann}. Let $F_+$ be the group of totally positive elements in $F$ and let $\Cl(\cO_\LV):=F_+\bksl \wh F^\x/\wh\cO_\LV^\x$ denote the narrow ring class group of conductor $\LV$. For $t\in \wh F^\x$, write $[t]=F_+ t \wh\cO_\LV^\x$ for the class represented by $t$. Let $\ep_\LV$ be a generator of the unit group $F_+\cap \wh\cO_\LV^\x$. Let $P_\Psi(X,Y)=(X-\CMP Y )(X-\ol{\CMP}Y)$ and $\delta=\ol{\CMP}-\CMP=\sqrt{\Delta_F}$. Define $\bftheta_X\colon \Zp^\x\to\Lam^\x$ by 
\[\bftheta_X(x)=\Dmd{x}_X^\onehalf\Dmd{x}^{-1}.\]
So $\bftheta_X^2=\bfal_X|_{\Zp^\x}$. Let $\phi$ be a finite order Hecke character of $\A_F^\x$ as in \eqref{E:brch.5}. Equivalently, $\phi|_{\wh F^\x}$ is an even/odd character of $\Cl(\cO_\LV)$, depending on the sign of $\phi_\infty(\delta)=(-1)^\frac{j}{2}$ or the parity of $\frac{j}{2}$.
\begin{defn}\label{D:67.6}
Let $\Xi\in \eord\cM\cS(NC,\bfI)^\pm[\bdsf]$. For $D\in\frakD_0$, we define $\cL_{\Xi}(D)\in\bfI$ as follows:
if $p$ is split in $F$, put
\[\cL_{\Xi}(D)=\int_{\Zp^\x}\bftheta_X(x)\mu_{\Xi}(D)(x)\in\bfI;\]
if $p$ is inert in $F$, put
\begin{align*}\cL_\Xi(D)&=\int_{\Zp}\bftheta_X(P_\Psi(x,1))\mu_{\Xi}(D)(x)\\
&\quad +\al_\bdsf^{-1}\int_{\Zp}\bftheta_X(P_\Psi(1,-px))\mu_{\Xi}\biggl(D\pMX{0}{1}{-p}{0}\biggl)(x).\end{align*}
 Fixing any base point $r\in \bfP$, we define the (square root) $p$-adic $L$-function $\cL_{\Xi^{\pm}/F\ot\brch}\in\bfI$ for $\bdsf/F$ by 
\[\cL_{\Xi^\pm/F\ot\brch}:=\sum_{[t]\in \Cl(\cO_\LV)}\brch(t)\bftheta_X(\cyc(\rmN(t)))\cdot \cL_{\Xi^{\pm}}(\MS{r}{\Psi(\ep_\LV)r},\Psi(t)\cmpt^{(\LV)}_\rmf).\]
Note that the above definition does not depend on the choice of $r$ and does not depend on the representatives $[t]$ in $\Cl(\cO_\LV)$.
\end{defn}

\subsection{The interpolation formulae}
For an elliptic modular form $f\in \cS_k(Np^r,\om^{-1})$ and a finite order Hecke character $\chi$ of $\A_F^\x$ with $\chi|_{\A^\x}=\om$, writing $\varphi_f:=\varPhi(f)$ for the adelic lift of $f$, recall that in \eqref{E:toric} we have introduced  the toric period integral given by
\[B_{f}^{\chi}(g):=B_{\varphi_f}^{\chi}(g)=\int_{\A^\x F^\x\bksl \A_F^\x}\varphi_f(\Psi(t)g)\chi(t)\rmd t.\] Let $\breve\bdsf\in \eord\bfS(N\LV,1,\bfI)[\bdsf]$ be the \emph{test vector} in \defref{D:test.6}.  Then $\breve \bdsf$ can be expressed as \[\breve\bdsf(q)=\prod_{q\divides \LV}(1-\beta_q(\bdsf)V_q)\cdot\bdsf,\]
where $\beta_q(\bdsf)$ is the fixed choice of roots of the Hecke polynomial $H_q(x,\bdsf)$ of $\bdsf$ at $q$. Let $\Xi\in \eord\cM\cS(N,\bfI)[\bdsf]$ and define 
\[\breve \Xi:=\prod_{q\divides \LV}(1-\beta_q(\bdsf)V_q)\cdot \Xi\in \eord\cM\cS(NC,\bfI)[\bdsf], \]
where $V_d$ is the level-raising operator defined in \eqref{E:V.6}. The next result shows that  $\cL_{\breve \Xi/F\ot\brch}$ interpolates $p$-adically the toric period associated with $\breve\bdsf_Q$ for $Q\in\frakX_{\bfI}^{++}$.

%For each $q\divides\LV$, we let $\stt{\al_q(\bdsf),\,\beta_q(\bdsf)}$ be roots of $T^2-\bfa(q,\bdsf)T+q^{-1}\Om^{-j}(q)\Dmd{q}_X=0$. Put 
%\[\breve\bdsf:=\prod_{q\divides \LV}(1-\beta_q(\bdsf)V_q)\bdsf=\sum_{I}(-1)^I \beta_I(\bdsf)V_I\bdsf.\]
\begin{prop}\label{P:interp2}
For $Q\in\frakX_{\bfI}^{++}$ we set $\chi_\Qx:=\brch\cdot \ep_\Qx^{-\onehalf}\Om^{\frac{k_\Qx-2}{2}}\circ\rmN_{F/\Q}$. Let $\pm=\brch_\infty(\delta)=(-1)^\frac{j}{2}$. We have 
\begin{align*}\cL_{\breve\Xi^{\pm}/\cK\ot\brch}(\Qx)=&\frac{(-2)(-\LV\delta\sqrt{-1})^\frac{k_\Qx}{2}L(1,\tau_F)}{\prod_{q\divides\LV}\zeta_q(1)}\cdot \frac{B_{\breve\bdsf_\Qx}^{\chi_\Qx}(\cmpt^{(\LV p^n)})}{\Omega_{\bdsf_\Qx}^{\pm}}\cdot {\rm Er}^{\pm}(\Xi_\Qx)\\
&\times\frac{\zeta_p(1)}{\zeta_{F_p}(1)\al_{\bdsf_\Qx}^n\abs{p^n}_{\Qp}^\frac{k_\Qx}{2}}
\end{align*}
for any sufficiently large $n\geq  \max\stt{c_p(\chi_Q),1}$, where $\cmpt^{(\LV p^n)}\in\GL_2(\A)$ is the special element defined in \defref{D:cmpt} and $\al_{\bdsf_\Qx}=\bfa(p,\bdsf_\Qx)$ is the $\bfU_p$-eigenvalue of $\bdsf_\Qx$.
\end{prop}
\begin{proof}
For simplicity, we write $f=\breve \bdsf_\Qx$ and $\varphi=\varPhi(
f)$ and put\[k=k_\Qx,\quad \om^\onehalf=\ep_\Qx^{-\onehalf}\Om^{\frac{k_\Qx-2}{2}}.\]  Then $\chi_\Qx=\brch\om_F^{-\onehalf}$. The first step is to work on the right hand side of the assertion, expressing the toric period integral $B_f^{\chi_Q}(\cmpt^{(cp^n)})$ as a finite sum of the values of the classical modular symbol $\eta_f^\pm$ in \subsecref{SS:62}. Let $\bfm(y)=\pDII{y}{y^{-1}}$ for $y\in \R^\x$. For $t\in \wh F^\x$, define the partial period by 
\[L_{[t]}(\varphi):=\int_{\R_+/\ep_{\LV}^\Z}\sum_{[u]\in \wh\cO_\LV^\x/\wh\cO_{\LV p^n}^\x}\varphi(\cmpt_\infty\bfm(y)\Psi(tu)\cmpt_\rmf^{(\LV p^n)})\chi_\Qx(u)\rmd^\x y.\]
Then we see that the toric period $B_{\breve\bdsf_Q}^{\chi_Q}(\cmpt^{(\LV p^n)})$ equals
\begin{align*}\int_{\A^\x F^\x\bksl \A_F^\x}\varphi(\Psi(t)\cmpt^{(\LV p^n)})\chi_\Qx(t)\rmd t
=&\vol(\cO_\LV^\x) \frac{\zeta_{F_p}(1)}{p^n\zeta_p(1)}\sum_{[t]\in \Cl(\cO_\LV)}\chi_\Qx(t)L_{[t]}(\varphi),
\end{align*}
where $\vol(\wh\cO_\LV^\x)$ is the volume of the image of $\wh\cO_\LV^\x$ in $\wh\Q^\x\bksl \wh F^\x$ with respect to the quotient measure $\rmd t/\rmd^\x t_\infty$ explicitly given by   \[\vol(\wh\cO_\LV^\x)^{-1}=\sqrt{\Delta_F}L(1,\tau_F)\#(\Z/\LV\Z)^\x=L(1,\tau_F)\delta\LV\prod_{q\divides\LV}(1-q^{-1}).\]
By a direct computation, if $z=\cmpt_\infty\bfm(y)\cdot\sqrt{-1}=\cmpt_\infty\cdot y^2\sqrt{-1}$, then 
\[J(\cmpt_\infty\bfm(y),\sqrt{-1})^{-2}=P_\Psi(z,1)\cdot (-\sqrt{-1}\Delta_F)^{-1},\]
and $\rmd z=(2\delta\sqrt{-1}) \cdot J(\cmpt_\infty\bfm(y),\sqrt{-1})^{-2}\rmd^\x y$. It follows that
\begin{align*}&L_{[t]}(\varphi)
=(2\sqrt{-1})^{-1}(-\sqrt{-1}\Delta_F)^\frac{2-k}{2}\delta^\frac{k-2}{2}\abs{\LV p^n\rmN(t)}_{\wh\Q}^\frac{k-2}{2}\\
&\quad\times\int_r^{\Psi(\ep_\LV)r}\sum_{[u]\in \wh\cO_\LV^\x/\wh\cO_{\LV p^n}}\chi_\Qx(u)\cdot f(z,\Psi(tu)\cmpt_\rmf^{(\LV p^n)})P_\Psi(z,1)^\frac{k-2}{2}\rmd z\\
&=\ell_1\cdot \abs{\rmN(t)}_{\A}^\frac{k-2}{2}\sum_{[u]\in \wh\cO_\LV^\x/\wh\cO_{\LV p^n}^\x}\chi_\Qx(u)\eta_f(\MS{r}{\Psi(\ep_\LV)r},\Psi(tu)\cmpt_\rmf^{(\LV p^n)})(P_\Psi^\frac{k-2}{2}),
\end{align*}
where $r$ can be chosen to be any point in $\bfP$ and \[\ell_1:=(2\sqrt{-1})^{-1}(-\LV\delta\sqrt{-1})^\frac{2-k}{2}\abs{p^n}_{\Qp}^\frac{k-2}{2}.\] 
For $t\in \wh F^\x$, we set
\[D_t:=(\MS{r}{\Psi(\ep_\LV)r},\Psi(t)\cmpt_\rmf^{(\LV)})\in\frakD_0.\]
Putting 
\[\ell_2:= \delta^{-1}L(1,\tau_F)^{-1}\LV^{-1} \prod_{q\divides \LV}\zeta_q(1)\cdot \frac{\zeta_{F_p}(1)}{p^n\zeta_p(1)},\]
we have
\begin{align*}&B_{\breve\bdsf_Q}^{\chi_Q}(\cmpt^{(\LV p^n)})=\ell_2\sum_{[t]\in\Cl(\cO_\LV)}\chi_\Qx(t)L_{[t]}(\varphi)\\
&=\ell_1\ell_2\sum_{[t]\in\Cl(\cO_\LV)}\chi_Q(t)\abs{\rmN(t)}_{\A}^\frac{k-2}{2}\sum_{[u]\in \wh\cO_\LV^\x/\wh\cO_{\LV p^n}^\x}\chi_\Qx(u)\eta_f(D_{tu}\cmpt_p^{(n)})(P_\Psi^\frac{k-2}{2}). 
\end{align*}
On the other hand, if we replace the base point $r$ by $\Psi(\delta)r$, noting that $\rmN(\delta)<0$, we obtain that
\begin{align*}
B_{\breve\bdsf_Q}^{\chi_Q}(\cmpt^{(\LV p^n)})=&\ell_1\ell_2\sum_{[t]\in\Cl(\cO_\LV)}\chi_\Qx(t\delta_\rmf)\abs{\rmN(t)}_{\A}^\frac{k-2}{2}\\
&\times\sum_{[u]\in \wh\cO_\LV^\x/\wh\cO_{\LV p^n}^\x}\chi_\Qx(u)(-1)^\frac{k-2}{2}[\bfc]\eta_f(D_{tu}\cdot \cmpt_p^{(n)})(P_\Psi^\frac{k-2}{2}),\end{align*}
where $[\bfc]$ is the involution on classical modular symbols. %For $t\in \wh F^\x$, we set \[D_t:=(\MS{r}{\Psi(\ep_\LV)r},\Psi(t)\cmpt_\rmf^{(\LV)})\in\frakD_0.\]
Since $\chi_\Qx(\delta_\rmf)=(-1)^\frac{k-2}{2}\phi_\infty(\delta)=(-1)^{\frac{k+j}{2}-1}$, we conclude that
\beq\label{E:1.P}\begin{aligned}
B_{\varphi}^{\chi_\Qx}(\cmpt^{(\LV p^n)})
=&\ell_1\ell_2
\sum_{[t]\in \Cl(\cO_\LV)}
\chi_\Qx(t)\abs{\rmN(t)}_{\A}^\frac{k-2}{2}\\
&\times\sum_{[u]\in\wh\cO_\LV^\x/\wh\cO_{\LV p^n}^\x}\chi_\Qx(u)\eta_f^{(-1)^\frac{j}{2}}(D_{tu}\cmpt^\setn_p)(P_\Psi^\frac{k-2}{2}).\end{aligned}\eeq
The second step is to work on the left hand side of the assertion
\[\cL_{\breve\Xi^\pm/F}(Q)=\sum_{[t]\in F^\x\bksl \wh F^\x/\wh\cO_\LV^\x}\chi_\Qx(t)\abs{\rmN(t)}_\A^\frac{k-2}{2}\rmN(t_p)^\frac{k-2}{2}\cdot \sL_{\breve\Xi^\pm}(D_t)(Q).\]
Put $\mho_Q=\frac{{\rm Er}^\pm(\Xi_\Qx)}{\Omega_{\bdsf_\Qx}^\pm}$. In view of \eqref{E:1.P}, we need to verify the following interpolation formula
\beq\label{E:claim1.6}\begin{aligned}
\cL_{\breve \Xi^\pm}(D_t)(Q)=&\mho_Q \al_{\bdsf_Q}^{-n}\rmN(t_p)^\frac{2-k}{2}\\
&\times\sum_{[u]\in\wh\cO_\LV^\x/\wh\cO_{\LV p^n}^\x}\om^{-\onehalf}(\rmN(u))\eta_f^+(D_{tu}\cmpt_p^\setn)(P_\Psi^\frac{k-2}{2}),\end{aligned}
\eeq
where $\cmpt_p^\setn=\pMX{p^n}{-1}{0}{1}$ if $p$ is split, and $\cmpt_p^\setn=\pMX{0}{1}{-p^n}{0}$ if $p$ is inert. For $d\divides\LV$, it is straightforward to verify that $V_d\Xi_Q^\pm=\mho_Q\cdot \xi_{V_d\bdsf_Q}^\pm$, and hence $\breve\Xi_Q^\pm=\mho_Q\cdot \xi_{f}^\pm$. It follows that for $D=(\MS{r}{s},g_\rmf)\in\frakD_0$ and $P(X,Y)\in L_{k_Q-2}(\Zp)$, we have
\beq\label{E:meas2.6}
\begin{aligned}
&Q\biggl(\int_{a+p^n\Zp}P(x,1)\mu_{\breve\Xi^\pm}(D)(x)\biggl)\\
=&\al_{\bdsf_Q}^{-n}\cdot \breve\Xi^\pm_Q\biggl(D\pMX{p^n}{a}{0}{1}\biggl)\biggl(P|\pMX{p^n}{a}{0}{1}\biggl)\,\,\,\text{by \eqref{E:mea.6}}\\
=&\mho_Q\cdot \al_{\bdsf_Q}^{-n} \eta_{f}^\pm\biggl(D\pMX{p^n}{a}{0}{1}\biggl)(P| g_p^{-1})\qquad \text{by \eqref{E:pms.6}}.
\end{aligned}
\eeq

Now we verify \eqref{E:claim1.6} in the case where $p$ is split in $F$. 
By \defref{D:67.6}, $\cL_{\breve\Xi^\pm}(D_t)(Q)$ equals
\begin{align*}
%&Q(\left(\int_{\Zp^\x}\om^{-\onehalf}(x)x^\frac{k-2}{2}\mu_{\breve\Xi^+}(D_t)(x)\right)\\
&\sum_{a\in (\Zp/p^n\Zp)^\x}\om_p^{-\onehalf}(a)Q\biggl(\int_{a+p^n\Zp}x^\frac{k-2}{2}\mu_{\breve\Xi^\pm}(D_t)(x)\biggl)\quad(Q(\bftheta_X(x))=\om_p^{-\onehalf}(x)x^\frac{k-2}{2})\\
%=&\al^{-n}\mho\sum_{a\in(\Zp/p^n\Zp)^\x}\om^{-\onehalf}(a)\xi_{f}^+(D_t\cdot\pMX{p^n}{a}{0}{1})((XY)^\frac{k-2}{2}|\pMX{p^n}{a}{0}{1})\qquad\text{by \eqref{E:Hidafamily.6}}\\
=&\frac{\mho_Q}{\al_{\bdsf_Q}^n}\sum_{a\in(\Zp/p^n\Zp)^\x}\om_p^{-\onehalf}(-a)\eta_f^\pm\biggl(D_t\cdot\pMX{p^n}{-a}{0}{1}\biggl)((XY)^\frac{k-2}{2}|\cmpt_p^{-1}\Psi(t_p^{-1}))
\end{align*}
by \eqref{E:meas2.6}. 
Then \eqref{E:claim1.6} follows from the equations $\om_p^\onehalf(-1)=(-1)^\frac{k-2}{2}$, and
\[(XY)|\cmpt_p^{-1}\Psi(t_p^{-1})=(XY)|\pDII{t_{\pbar}}{t_{\p}}\pMX{1}{-\CMP}{-1}{\ol{\CMP}}=-\rmN(t_p^{-1})\cdot P_\Psi(X,Y).\]

In the inert case, $\cL_{\breve\Xi^\pm}(D_t)(Q)$ equals
\begin{align*}
&\sum_{a=0}^{p^n-1}\om^{-\onehalf}(\rmN(a-\CMP))\int_{a+p^n\Zp}\rmN(x-\CMP)^\frac{k-2}{2}\mu_{\breve\Xi^\pm_\Qx}(D_t)(x)\\
&+\al_{\bdsf_Q}^{-1}\sum_{a=0}^{p^{n-1}-1}\om^{-\onehalf}(\rmN(1+pa\CMP))\int_{a+p^{n-1}\Zp}\rmN(1-px\CMP)^\frac{k-2}{2} \mu_{\breve\Xi^\pm_\Qx}\biggl(D_t\pMX{0}{1}{-p}{0}\biggl)(x)\\
=&\rmN(t_p)^\frac{2-k}{2}\frac{\mho_Q}{\al_{\bdsf_Q}^n}\biggl(\sum_{a=0}^{p^n-1}\om^{-\onehalf}(\rmN(a-\ol{\CMP}))\eta_f^\pm\biggl(D_t\pMX{p^n}{a}{0}{1}\biggl)(P_\Psi^\frac{k-2}{2})\\
&+\sum_{a=0}^{p^{n-1}-1}\om^{-\onehalf}(\rmN(1+pa\ol{\CMP}))\eta_f^\pm\biggl(D_t\pMX{0}{1}{-p}{0}\pMX{p^{n-1}}{a}{0}{1}\biggl)(P_\Psi^\frac{k-2}{2})\biggl).
%=&\al^{-n}\mho\sum_{[u]\in \wh\cO_\LV^\x/\wh\cO_{\LV p^n}^\x}\om^{-\onehalf}(\rmN(u))\xi_f^+(D_{tu}\cdot\pMX{0}{1}{-p^n}{0})(P_\Psi^\frac{k-2}{2}).
\end{align*}
We thus obtain \eqref{E:claim1.6} from the observations below
\begin{align*}\pMX{p^n}{a}{0}{1}U_1(p^n)&=\Psi(a-\ol{\CMP})\pMX{0}{1}{-p^n}{0}U_1(p^n),\\
\pMX{0}{1}{-p}{0}\pMX{p^{n-1}}{a}{0}{1}U_1(p^n)&=\Psi(1+pa\ol{\CMP})\pMX{0}{1}{-p^n}{0}U_1(p^n).\end{align*}
This verifies \eqref{E:claim1.6} in both cases and finishes the proof.
\end{proof}

\section{Derivative of the twisted triple product $p$-adic $L$-function and Stark-Heegner points}
\subsection{Factorization of $\cL_{\bdsE_\brch^{[a]},\bdsf}$}
In this section, we show that the $p$-adic $L$-function $\cL_{\bdsE_\brch^{[a]},\bdsf}$ in \defref{D:padicL.5} can be essentially factorized into a product of the square root $p$-adic $L$-function $\cL_{\breve\Xi^- /F\ot \brch}$ for $\bdsf$ over $F$ and the Mazur-Kitagawa $p$-adic $L$-function $L_p(\Xi^+,\Om^a)$. We will use an auxiliary $p$-adic Rankin-Selberg $L$-function in the proof, so we first recall that for a primitive Hida family $\bdsg\in\bfJ\powerseries{q}$ with some normal domain $\bfJ$ finite over $\Lam$, there exists an element \[L^\bdsf_p(\bdsf\ot\bdsg)\in (\bfI\wh\ot \bfJ\wh\ot\Lam)\ot_{\bfI}\Frac\bfI\] such that for each point $(Q_1,Q_2,P)\in\frakX_\bfI^+\times\frakX_\bfJ^+\times \frakX^+_\Lam$ with $k_{Q_2}<k_P\leq k_{Q_1}$, we have the interpolation formula:
\beq\label{E:interp4}\begin{aligned}&L_p^\bdsf(\bdsf\ot\bdsg)(Q_1,Q_2,P)\\
=&(\sqrt{-1})^{1+k_{Q_2}-2k_P}\cdot \frac{L^{\stt{p}}(k_P-\frac{k_{Q_1}+k_{Q_2}}{2},\pi_{\bdsf_{Q_1}}\times\pi_{\bdsg_{Q_2}}\ot \Om^{-k_P})}{{\rm Per}^\dagger(\bdsf_{Q_1})}\\
&\times \gamma\biggl(k_P-\frac{k_{Q_1}+k_{Q_2}}{2},\varrho_{\bdsf_{Q_1},p}\ot \pi_{\bdsg_{Q_2},p}\biggl)^{-1},\end{aligned}\eeq
where $\varrho_{\bdsf_{Q_1},p}:\Qp^\x\to\C^\x$ is the unramified character defined by \eqref{E:unr.6} and $\gamma(s,\varrho_{\bdsf_{Q_1},p}\ot \pi_{\bdsg_{Q_2},p})$ is the gamma factor in \subsecref{SS:eps}. We call $L_p^\bdsf(\bdsf\ot\bdsg)$ the primitive Hida's three-variable Rankin-Selberg $p$-adic $L$-function associated with $\bdsf$ and $\bdsg$. Imprimitive three-variable $p$-adic $L$-functions (with some Euler factors removed) were first constructed by Hida in \cite[Theorem I]{Hida88Fourier}, and the primitive ones with the above form of the interpolation formula were proved in \cite[Theorem 7.1]{HsiehChen19} following the method in \cite{Hida88Fourier}. 
We first prove a preliminary result: 
%Consider the normalized twisted triple $p$-adic $L$-function defined by 
%\[ \cL^*_{\bdsE_\brch^{[a]},\bdsf}:= \cL_{\bdsE_\brch^{[a]},\bdsf}\cdot\in (\bfI\wh\ot \Lam)\ot_\bfI\Frac\bfI,\]

\begin{prop}\label{prop:CXi}
Let $\bdsf\in \eord\bfS(N,1,\bfI)$ be a primitive Hida family. 
For every $\Xi\in \eord\cM\cS(N,\bfI)[\bdsf]$ there is an element $C_\Xi\in\Frac\bfI$ which is holomorphic at every arithmetic point $Q\in\frakX_\bfI^+$ with the value 
\[C_\Xi(Q)=\frac{{\rm Per}^\dagger(\bdsf_Q)}{\Omega^+_{\bdsf_Q}\Omega^-_{\bdsf_Q}}\cdot {\rm Er}^+(\Xi_Q){\rm Er}^-(\Xi_Q). \]
\end{prop}

\begin{proof} 
Choose a Dirichlet character $\chi$ with $\chi(-1)=-1$ and an imaginary quadratic field $K$ where $p$ is split. Let $\chi_K:=\chi\circ\rmN_{K/\Q}$ be a finite order Hecke character of $\A_K^\x$. Let $\bdsg$ denote a primitive Hida family such that the weight one specialization $\bdsg_{Q_0}$ is a $p$-stabilized theta series $\theta_{\chi_K}^{(p)}$ associated with $\chi_K$. Define the two-variable $p$-adic $L$-function $L_p(\bdsf_{/K}\ot\chi_K)$ by 
\[L_p(\bdsf_{/K}\ot\chi_K):=(1\ot Q_0\ot 1)\bigl(L_p^\bdsf(\bdsf\ot\bdsg)\bigl)\in\bfI\wh\ot\Lam .\] 
Let $\frakX_\Lam^{(2)}$ be the set of arithmetic points of weight $2$. For $P\in\frakX_\Lam^{(2)}$, define \[C_{\Xi,P}:=(1\ot P)\biggl(\frac{L_p(\Xi^-,\chi)L_p(\Xi^+,\chi\tau_{K/\Q})}{L_p(\bdsf_{/K}\ot\chi_K)}\biggl)\in \Frac(\bfI\ot_\cO\cO(P)).\]
%Let $P_0$ be the point with $k_{P_0}=2$ and $\ep_{P_0}=1$ and set $C_\Xi=C_{\Xi,P_0}\in\Frac\bfI$. 
%Let $P\in \frakX_\Lam^{(2)}$. 
From the interpolation formulae \eqref{E:interp3} and \eqref{E:interp4}, we see that 
\[C_{\Xi,P}(Q)=\frac{{\rm Per}^\dagger(\bdsf_Q)}{\Omega^+_{\bdsf_Q}\Omega^-_{\bdsf_Q}}\cdot {\rm Er}^+(\Xi_Q){\rm Er}^-(\Xi_Q)\]
for all $Q\in\frakX_\bfI^+$ with $k_Q>2$. 
Hence $C_{\Xi,P}$ is independent of the choice of $P$ and can be denoted as $C_\Xi$. 
Thanks to Rohrlich's theorem \cite{Roh84Inv}, for any arithmetic point $Q\in \frakX_\bfI^+$ and $P\in\frakX_\Lam^{(2)}$, one can find an odd Dirichlet character $\chi$ such that $L_p(\bdsf_{/K}\ot\chi_K)(Q,P)\neq 0$, which implies that $C_\Xi$ is holomorphic at $Q$. 
\end{proof}

\begin{Remark}
If the residual Galois representation associated with $\bdsf$ is absolutely irreducible and $p$-distinguished, then the Gorensteiness of the local component of the Hecke algebra $\bfT(N,\bfI)$ corresponding to $\bdsf$ is known thanks to the work of Wiles et al. (\cite[Corollary 2, page 482]{Wiles95}). It follows that the $\bfI$-module $\eord\cM\cS(N,\bfI)^\pm[\bdsf]$ is free of rank one by \cite[Lemma 5.11]{Kitagawa94}. 
%there exist a $\bfI$-adic modular symbol $\Xi\in \eord\cM\cS(N,\bfI)^\pm[\bdsf]$. 
Choose a basis $\Xi^\pm$ in each space. 
It is determined up to multiple of $\bfI^\times$. 
Put $\Xi=\Xi^++\Xi^-$. Then $p$-adic error terms ${\rm Er}^\pm(\Xi_Q)$ are $p$-adic units for all $Q\in\frakX_\bfI^+$ by \cite[Proposition 5.12]{Kitagawa94}, and $C_\Xi$ is  a generator of the congruence ideal $C(\bdsf)$ by a result of Hida \cite[Theorem 0.1]{Hida88AJM}.
\end{Remark}

Now we are ready to prove the factorization. 

\begin{thm}\label{T:main.7}
Let $a$ be an even integer, $\phi:\Cl(\cO_\LV)\to \cO^\x$ an odd character of the exact conductor $\LV$, $\bdsf\in \eord\bfS(N,1,\bfI)$ a primitive Hida family of tame conductor $N$, and $\Xi\in \eord\cM\cS(N,\bfI)[\bdsf]$. 
Then we have 
\[C_\Xi\cdot \cL_{\bdsE_\brch^{[a]},\bdsf}=\cL_{\breve\Xi^- /F\ot \brch}\cdot L_p(\Xi^+,\Om^a)\cdot \frakf c_1,\]
where $\frakf\in (\Lam\wh\ot\Lam)^\x$ and the constant $c_1\in\Zbar_{(p)}^\x$ are defined in \propref{P:interp1} with $j=2$. \end{thm}

\begin{proof} 
Propositions \ref{P:interp1}, \ref{P:interp2}, \eqref{E:unr.6} and \eqref{E:interp3} immediately show that 
\begin{align*}
\cL_{\bdsE_\brch^{[a]},\bdsf}(\Qx,\Qz)=&\cL_{\breve\Xi^-/\cK\ot\brch}(\Qx)L_p(\Xi^+,\Om^a)(\Qx,\Qz)\frac{\Omega_{\bdsf_\Qx}^+\Omega_{\bdsf_\Qx}^-\cdot\frakf(\Qx,\Qz)c_1}{{\rm Per}^\dagger(\bdsf_Q){\rm Er}^+(\Xi_\Qx){\rm Er}^-(\Xi_\Qx)},
\end{align*}
which combined with Proposition \ref{prop:CXi} completes our proof. 
\end{proof}

\subsection{The derivative of $\cL_{\bdsE_\brch^{[2]},\bdsf}$}

We shall keep the notation in \subsecref{R:GS.6}.
Let $E$ be an elliptic curve over $\Q$ of conductor $pN$. 
There exists a primitive Hida family $\bdsf\in\bfI\powerseries{q}$ whose specialization $\bdsf_Q$ at some weight two point $Q\in\frakX_\bfI^+$ is the elliptic newform $f$ associated with $E$ by the modularity theorem \cite{Wiles95}. 
Here $\bfI$ is the local component of $\bfT(N,\bfI)$ corresponding to $\lam_{\bdsf}$. 
Let $\wtsp=\{k\in\Cp\mid \abs{k}_p\leq 1\}$. 
We write $j:\wtsp\hookto\Spec\Lam(\Cp)$ for the map $k\mapsto (Q_k:[x]\mapsto x^k)$.
Let $\wp_2$ be the kernel of $Q_2$. Then we have $\bfI\subset \bfI_{\wp_Q}=\Lam_{\wp_2}$ since $f$ has rational coefficients. This implies that there exists a neighborhood $\sU$ around $2\in\wtsp$ such that $j:\sU\hookto \Spec\bfI(\Cp)$. 
Define an analytic function on $\sU\times \wtsp$ by 
\[\cL_{\bdsE_\brch^{[a]},\bdsf}(k,s)=\cL_{\bdsE_\brch^{[a]},\bdsf}(Q_k,Q_{s}),\quad (k,s)\in \sU\times \wtsp.\]
% Let $\sA(\cU)$ be the ring of rigid analytic functions on $\cU$ and let $i^*_\bfu:\bfI\to \sA(\cU)$ be the pull-back of $j$. 
 
 \begin{thm}\label{T:main}Suppose that $p$ is inert in $F$ and $\brch:\Cl(\cO_F)\to\cO^\x$ is an odd narrow ideal class character, \ie $\LV=1$. 
 Then $\cL_{\bdsE_\brch^{[a]}\bdsf}(2,s)=0$ and 
 \begin{align*}\frac{\partial\cL_{\bdsE^{[2]}_\brch,\bdsf}}{\partial k}(2,s+1)&=\onehalf(1+\brch(\frakN)^{-1}w_N) \log_E P_\brch\cdot L_p(E,s) \frac{m_E^22^{\al(E)}}{c_f}\Dmd{\Delta_F}^\frac{s-1}{2},\end{align*}
where \begin{itemize}
\item $w_N\in\stt{\pm 1}$ be the sign of the Fricke involution at $N$ acting on $f$, 
\item $P_\brch\in E(F_p)\ot\Q(\brch)$ is the Stark-Heegner point in \cite[(182)]{Darmon01Ann},
\item $L_p(E,s)$ is the Mazur-Tate-Titelbaum $p$-adic $L$-function for $E$,
\item $c_f\in \Z^{>0}$ is the congruence number for $f$, $m_E\in\Q^\x$ is the Manin constant for $E$ and $2^{\al(E)}=[\rmH_1(E(\C),\Z):\rmH_1(E(\C),\Z)^+\oplus \rmH_1(E(\C),\Z)^-]$.
\end{itemize}
\end{thm}

\begin{proof}
For each $\Xi\in \eord\cM\cS(N,\Lam)\ot_{\Lam} A(\sU)[\bdsf]$ we put \[\cL_p(\Xi/F,\brch,k)=\cL_{\Xi^-/F\ot\brch}(Q_k),\quad k\in\sU.\]  Shrinking $\sU$ if necessary, we may assume that the function $\cL_p(\Xi/F,\brch,k)$ is analytic on $\sU$. Since $\pi_f$ is special at the inert prime $p$, it is well-known that the local root number of the base change ${\rm BC}_F(\pi_f)\ot\brch$ is $-1$, and hence the toric period $B^{\brch}_f$ must vanish by the dichotomy theorem of Saito and Tunnell (See \cite{SaitoH93Comp} and \cite{Tunnell83AJM}). 
\propref{P:interp2} shows that $\cL_p(\Xi/F,\brch,2)=0$, and so by Theorem \ref{T:main.7}, we have $\cL_{\bdsE^{[a]}_\brch,\bdsf}(2,s)=0$ for all even $a$, and get 
\[
C_\Xi(2)\frac{\partial\cL_{\bdsE^{[2]}_\brch,\bdsf}}{\partial k}(2,s+1)
=\frac{d\cL_p}{d k}(\Xi/F,\brch,2)\cdot L_p(\Xi^+,\Om^2)(2,s+1)\cdot \frakf(2,s+1) c_1 \]
by \thmref{T:main.7}. 

Now we fix the normalization of $\Xi$. 
The $\Lam_{\wp_2}$-module \[\eord\cM\cS(N,\bfI)^\pm[\bdsf]\ot_{\bfI}\bfI_{\wp_Q}=(\eord\cM\cS(N,\Lam)^\pm\ot\Lam_{\wp_2})[\bdsf]\] is free of rank one. Let $\Xi^\pm$ be the basis normalized so that the weight two specialization $\Xi^\pm_Q=\frac{\xi^\pm_f}{\Omega^\pm}$ with the periods $\Omega^\pm=(2\pi \sqrt{-1})^{-1}\Omega_E^\pm$, where $\Omega_E^\pm$ are the plus/minus periods for $E$ such that $\Omega_E^+$ and $(\sqrt{-1})^{-1}\Omega_E^-$ are real and positive.  
By the inspection on the interpolation \eqref{E:interp3}, we see easily that the associated Mazur-Kitagawa $p$-adic $L$-function $L_p(\Xi^+,\Om^2)(2,s+1)$ is the cyclotomic $p$-adic $L$-function $2 L_p(E,s)$ for the elliptic curve $E$. This extra $2$ comes from the factor $2$ in the definition of the archimedean $\Gamma$-factor $\Gamma_\C(s)=2(2\pi)^{-s}\Gamma(s)$. On the other hand, it is clear that $\frakf(2,s+1) c_1=4\Dmd{\Delta_F}^\frac{s-1}{2}$ with $a=j=2$, and by the formulae in \cite[p.255]{Hida81Inv}, we have
\[\norm{f}^2_{\Gamma_0(N)}=c_f m_E^{-2}2^{-2-\al(E)}\pi^{-2}(\sqrt{-1})^{-1}\Omega_E^+\Omega_E^-.\]
We thus obtain
\[C_\Xi(2)=\frac{{\rm Per}^\dagger(f)}{\Omega^+\Omega^-}=\frac{-(-2\sqrt{-1})^3\norm{f}^2_{\Gamma_0(N)}}{(-4\pi^2)^{-1}\Omega^+_E\Omega_E^-}=\frac{8c_f}{m_E^22^{\al(E)}}.\] 
Putting these together, we get the stated formula from the following lemma. 
\end{proof}

\begin{lm}
Assumptions being as in Theorem \ref{T:main}, if $\Xi$ is normalized as above, then we have 
\[\frac{d\cL_p}{dk}(\Xi/F,\brch,2)=\frac{1}{2}(1+\phi(\sg_\frakN)w_N)\log_E P_\brch.\]
\end{lm}

\begin{proof}
We will compute the derivative of $\cL_p(\Xi/F,\brch,k)$ at $k=2$ for the normalized $\Xi$ above. 
Let $\mu^{\rm GS}_{\Xi^-}(x,y)$ be the $p$-adic measure on $L_0'$ attached to $\Xi^-$ introduced in \subsecref{R:GS.6}. By definition, we have the expression 
\begin{align*}\cL_p(\Xi/F,\brch,k)&=\sum_{[t]\in\Cl(\cO_\LV)}\brch(t)\Dmd{\cyc(\rmN(t))}^{\frac{k-2}{2}}\\
&\times\int_{L_0'}\Dmd{(x-\CMP y)(x-\ol{\CMP} y)}^\frac{k-2}{2}\mu^{\rm GS}_{\Xi^-}(\MS{r}{\Psi(\ep_1)r},\Psi(t)\cmpt_\rmf)(x,y).\end{align*}
Here $\ep_1$ is the totally positive fundamental unit in $\cO_F^\x$ and $\cmpt_\rmf$ is the finite part of $\cmpt\in\GL_2(\wh\Q)$ defined in \subsecref{SS:optimal}. Choosing a branch of the $p$-adic logarithm $\log:F_p^\x\to F_p$, we obtain
\beq\label{E:72}\frac{d\cL_p}{dk}(\Xi/F,\brch,2)=\onehalf\sum_{[t]\in\Cl(\cO_F)}\brch(t)(J_{\CMP}[t]+J_{\ol{\CMP}}[t]),\eeq
where for $\tau\in\Cp$ with $\tau\not\in \Qp$, 
\begin{align*}J_{\tau}[t] &:=\int_{L_0'}\log(x- \tau y)\mu^{\rm GS}_{\Xi^-}(\MS{r}{\Psi(\ep_1)r},\Psi(t)\cmpt_\rmf)(x,y).\end{align*}
Let $\cJ=\pMX{-1}{\rmT(\CMP)}{0}{1}\in\GL_2(\Q)\hookto\GL_2(\wh\Q)$. Write $\cJ_p$ and $\cJ^{(p)}$ for its image in $\GL_2(\Qp)$ and $\GL_2(\wh\Q^{(p)})$ respectively and let  $\tau_N=\pMX{0}{1}{-N}{0}\in\GL_2(\wh\Q^{(p)})$ be the Fricke involution at $N$. 
Since $\cJ^2=1$ and $\cmpt_p=1$, one verifies that \[\cmpt \cJ_p=\cJ\cJ^{(p)}\cmpt_\rmf=\cJ\Psi(\sigma_\frakN)\cmpt_\rmf\cdot \tau_N\]
for an appropriate choice of a finite idele $\sigma_\frakN\in \wh F^\x$ such that $(\sg_\frakN\wh\cO_F\cap F)=\frakN$. 
 We have $\cJ(\CMP)=\ol{\CMP}$ and $\int_{L_0'}\mu_{\Xi^-}(D)(x,y)=0$ as $\bdsf_2$ is new at $p$. It follows from the $U_0(N)$-invariance \eqref{E:Uinv.6} that
\beq\label{E:73}\begin{aligned}
J_{\ol{\CMP}}[t]&=\int_{L_0'}\log_p(x- \CMP y)\mu^{\rm GS}_{\Xi^-}(\MS{r}{\Psi(\ep_1)r},\Psi(t)\cmpt \cJ_p)(x,y)\\
&=\int_{L_0'}\log(x- \CMP y)\mu^{\rm GS}_{\Xi^-}([\bfc](\MS{r}{\Psi(\ep_1^{-1})r},\Psi(t\sg_\frakN)\cmpt \tau_N))(x,y)\\
&=(-1)\cdot (-1)\cdot w_N\cdot J_{\CMP}[t\sg_{\frakN}\ep_1].
\end{aligned}\eeq
%For $t\in \wh F^\x$, write \[\Psi(t)\cmpt^{(1)}=\gamma_t\cdot u_t\] with $\gamma_t\in\GL_2^+(\Q)$ and $u_t\in U_0(pN)$. Put \[\tau_t=\gamma_t(\theta)\in F_p;\quad \Psi_t=\gamma_t^{-1}\Psi\gamma_t.\]
With the fixed choice of periods $\Omega^\pm$, it is straightforward to deduce from \cite[Corollary 2.6]{BD09Ann} that the $p$-adic logarithm of the Stark-Heegner point $P_\brch$ is given by 
\[\log_E P_\phi=\sum_{[t]\in\Cl(\cO_F)}\phi(t)J_{\CMP}[t].\]
%Indeed, for each $t\in\wh F^{(p)\x}$ we can write $\fraka=t\wh\cO_F\cap F$ and $\Psi(t)\cmpt=\gamma_\fraka^{-1} u_\fraka$ for $\gamma_\fraka\in \GL_2(\Z_{(p)})$ and $u_\fraka\in U_0(N)$. Then $[\fraka]\mapsto \Psi_\fraka=\gamma_\fraka\Psi\gamma_\fraka^{-1}$ gives a bijection between the narrow ideal class in $\Cl(\cO_F)$ and the optimal embedding $\Psi_\fraka:\cO_F\hookto R_N$. We the have $J_{\CMP}[t]=J_{\gamma_\fraka \CMP}$.
We thus obtain the formula for $\frac{d\cL_p}{dk}(\Xi/F,\brch,2)$ from \eqref{E:72} and \eqref{E:73}. 
\end{proof}

\begin{Remark}The same argument applies to more general ring class characters with split conductor (\ie $\LV\neq 1$ is a product of primes split in $F$), but the formulae are more complicated due to the non-canonical choice of the test vector $\breve\Xi$ in the construction of $\cL_{\breve\Xi^-/F\ot\brch}$.
\end{Remark}

%If $k\in\Z$ with $k\con 2\pmod{2(p-1)}$, we have 
%\[\cL_{\bdsf/\cK}(k)=\sum_{[t]\in\Cl(c)}\Omega(t)\abs{\rmN(t)c}_\A^\frac{k-2}{2}\int_{r}^{\Psi(\ep_c)r}\bdsf_k^{\rm Coh}(z,\Psi(t)\cmpt)P_\Psi(z,1)^\frac{k-2}{2}\rmd z.\] 

%\[f(z,\Psi(t)\cmpt)P_\Psi(z,1)^\frac{k-2}{2}=\varphi_f(\Psi(t_\infty)\gamma_\infty,\Psi(t)\cmpt)J(\Psi(t_\infty)\gamma_\infty,i)^{k-2}\rmd^\x t_\infty.\]
\bibliographystyle{amsalpha}
\bibliography{ref.bib}
\end{document}